\documentclass[9pt,reqno]{amsart}
\usepackage{amsmath,amssymb,latexsym,soul,cite,mathrsfs}
\usepackage{color,enumitem,graphicx}
\usepackage[colorlinks=true,urlcolor=blue,
citecolor=red,linkcolor=blue,linktocpage,pdfpagelabels,
bookmarksnumbered,bookmarksopen]{hyperref}
\usepackage[english]{babel}
\usepackage{enumitem}

\usepackage[left=1.5cm,right=1.5cm,top=1.7cm,bottom=1.6cm]{geometry}

\usepackage[hyperpageref]{backref}

\makeatletter
\newcommand{\leqnomode}{\tagsleft@true}
\newcommand{\reqnomode}{\tagsleft@false}
\makeatother

\date{}
\def\nd{\noindent}
\def\thend{\rule{3mm}{3mm}}

\newtheorem{theorem}{Theorem}[section]

\newtheorem{cor}{Corollary}[section]

\newtheorem{prop}{Proposition}[section]
\newtheorem{lem}{Lemma}[section]
\newtheorem{rmk}{Remark}[section]

\usepackage[hyperpageref]{backref}

\newcommand{\eps}{\epsilon}

\newcommand{\E}{\bar{\eps}}

\begin{document}
	\title[Singular elliptic systems via the nonlinear Rayleigh quotient]{Singular nonlocal elliptic systems via nonlinear Rayleigh quotient}
	\vspace{1cm}
	
	\author{Edcarlos D. Silva}
	\address{Edcarlos D da Silva \newline  Universidade Federal de Goi\'as, IME, Goi\^ania-GO, Brazil-----{\bf\it Email address: edcarlos@ufg.br}}
	
	\author{Elaine A. F. Leite}
	\address{Elaine A. F. Leite \newline Instituto Federal de Goi\'as, Campus Goi\^anica, Goi\^ania-GO, Brazil-----{\bf\it Email address: {elaine.leite@ifg.edu.br}} }
	
	\author{Maxwell L. Silva}
	\address{Maxwell L. Silva \newline Universidade Federal de Goi\'as, IME, Goi\^ania-GO, Brazil-----{\bf\it Email address:  maxwell@ufg.br} }


	\thanks{The first author was partially supported by CNPq with grant 309026/2020-2.}
	\begin{abstract}
		In the present work, we establish the existence of two positive solutions for singular nonlocal elliptic systems. More precisely, we consider the following nonlocal elliptic problem: 
		$$\left\{\begin{array}{lll} 
			(-\Delta)^su +V_1(x)u =   \lambda\frac{a(x)}{u^p} +  \frac{\alpha}{\alpha+\beta}\theta |u|^{\alpha - 2}u|v|^{\beta},  \,\,\, \mbox{in}  \,\,\, \mathbb{R}^N, \\
			(-\Delta)^sv +V_2(x)v=   \lambda \frac{b(x)}{v^q}+  \frac{\beta}{\alpha+\beta}\theta |u|^{\alpha}|v|^{\beta-2}v, \,\,\, \mbox{in}  \,\,\, \mathbb{R}^N, \\
			
		\end{array}\right. \;\;\;(u, v) \in H^s(\mathbb{R}^N) \times H^s(\mathbb{R}^N),$$
		where $ 0<p \leq q < 1<\;\alpha, \beta  \;,\;2<\alpha + \beta < 2^*_s$, $\theta > 0, \lambda > 0, N > 2s$, and $s \in (0,1)$. The potentials $V_1, V_2: \mathbb{R}^N \to \mathbb{R}$ are continuous functions which are bounded from below. Under our assumptions, we prove that there exists the largest positive number $\lambda^* > 0$ such that our main problem admits at least two positive solutions for each $\lambda \in (0, \lambda^*)$. Here we apply the nonlinear Rayleigh quotient together with the Nehari method. The main feature is to minimize the energy functional in Nehari set which allows us to prove our  results without any restriction on the size of parameter $\theta > 0$. Moreover, we shall consider the multiplicity of solutions for the case $\lambda = \lambda^*$ where degenerated points are allowed.
	\end{abstract}
	\maketitle
	
	\section{Introduction}
	In this work, we consider the existence and multiplicity of positive solutions for singular nonlocal elliptic systems. More specifically, we study the nonlocal elliptic problem:
	
	\begin{equation}\left\{\begin{array}{lll}\label{sistema Principal singular} 
			(-\Delta)^su +V_1(x)u =   \lambda\frac{a(x)}{u^p} +  \frac{\alpha}{\alpha+\beta}\theta |u|^{\alpha - 2}u|v|^{\beta},  \,\,\, \mbox{in}  \,\,\, \mathbb{R}^N, \\
			(-\Delta)^sv +V_2(x)v=   \lambda \frac{b(x)}{v^q}+  \frac{\beta}{\alpha+\beta}\theta |u|^{\alpha}|v|^{\beta-2}v, \,\,\, \mbox{in}  \,\,\, \mathbb{R}^N, \\
		\end{array}\right. \;\;\;(u, v) \in H^s(\mathbb{R}^N) \times H^s(\mathbb{R}^N).\;\;\tag{$S_\lambda$}
	\end{equation}
	The potentials $V_i: \mathbb{R}^N \to \mathbb{R}, i = 1, 2$ are continuous functions and $\lambda > 0, \theta > 0$. Here we also mention that $ 0<p \leq q < 1<\;\alpha, \beta  \;,\;2<\alpha + \beta < 2^*_s$, $2^*_s = 2 N/(N - 2s)$, $\theta > 0, \lambda > 0, N > 2s$, and $s \in (0,1)$. Later on, we shall discuss the hypotheses on $V, s, p, q$ and $\alpha, \beta$.
	
	For the scalar case we mention that nonlocal semilinear elliptic problems have been attracted many attention in the last years, see \cite{bisci,DipierroRegul, pala, guia,felmer,AIalMoscSqua2016,servadei,servadei1,servadei2,secchi,secchii} and reference therein. Furthermore, we observe there are several physical applications such as nonlinear optics. Furthermore, the fractional Laplacian operator has been accepted as a model for diverse physical phenomena such as diffusion-reaction equations, quasi-geostrophic theory, Schr\"odinger equations, Porous medium problems, see for instance \cite{aka,biboa,ya,consta,pablo,las}. For further applications such as continuum mechanics, phase transition phenomena, populations dynamics, image processes, game theory, see \cite{ber,cafa,las}. It is important to stress that  semilinear nonlocal reaction-diffusion equations have attracted some attention in the last few years. The main motivation for this kind of problem is to combine nonlinear and quasilinear nonlocal terms in order to model a nonlinear diffusion. On this subject we refer the reader to \cite{vasquez1,vasquez2,vasquez3} and references therein.

	Now, we mention that nonlocal elliptic systems have been widely studied considering some tools provided by variational methods, see \cite{ambro1,ambro2,brande,colorado,jotinha}.  For the local case, that is, assuming that $s = 1$ we refer the reader to the important works for elliptic systems \cite{defigueiredo,maia,oliveira,pompo}. In those works was proved several results on existence and multiplicity of solutions taking into account some hypotheses on the potential as well as in the nonlinearity. Recall that there exist some results on singular elliptic problems, see \cite{giai,mo}.  For further results on fractional elliptic system we refer the interested reader to \cite{guo,lu,Laskin1}.  It is important to recall that a pair $(u,v)$ is said to be a ground state solution for the System  \eqref{sistema Principal singular} when $(u,v)$ has the minimal energy among all nontrivial solutions. At the same time, a nontrivial solution $(u,v)$ is a bound solution for the System  \eqref{sistema Principal singular} whenever $(u,v)$ has finite energy.
	
	Recall that in \cite{colorado} the authors considered the following nonlocal elliptic system
	\begin{equation}\label{lindo}
		\left\{\begin{array}{lll}
			(-\Delta)^su + \lambda_1 u =   \mu_1|u|^{2p - 2}u +  \beta |u|^{p - 2}u |v|^{p},  \;\;\; \mbox{in}\;\;\; \mathbb{R}^N, \\
			(-\Delta)^sv + \lambda_2 v= \mu_2 |v|^{2 p - 2}v + \beta |u|^{p}|v|^{p-2}v,  \;\;\; \mbox{in}\;\;\; \mathbb{R}^N, \\
			(u, v) \in H^s(\mathbb{R}^N) \times  H^s(\mathbb{R}^N).
		\end{array}\right.
	\end{equation}
	where $N \in \{1, 2,3\}$, $\lambda_i, \mu_i > 0, i = 1,2, p \geq 2, (p-2)N/p < s < 1$. The authors proved several results on existence of ground and bound state solutions assuming that $p > 2$ or $p = 2$. The main ingredient in that work was to combine the Nehari method and the size of $\beta > 0$ in order to avoid semitrivial solutions. In fact, assuming that $\beta > 0$ is small, the authors proved also that the bound state for the Problem \eqref{lindo} is a semitrivial solution. On the other hand, assuming that $\beta > 0$ is large enough, the authors proved existence of ground state solutions $(u,v)$ for the Problem \eqref{lindo} where $u \neq 0$ and $v \neq 0$. In other words, for each $\beta > 0$ large enough, the authors ensured that for the Problem \eqref{lindo} there exists at least one non-semitrivial ground state solution.  At the same time, we observe that Problem \eqref{lindo} is superlinear at the origin and at infinity.  Motivated in part by the previous discussion we shall consider existence and multiplicity of solutions for the System  \eqref{sistema Principal singular} assuming that the nonlinear term admits a singular part and another which provides us the coupling term. Furthermore, for the coupling term is a more general function due to the fact that $1 < \alpha, \beta < 2^*_s$ where $\alpha$ and $\beta$ can be different. Hence, our main objective in the present work is to guarantee existence and multiplicity of solutions without any restriction on the size of $\theta$. For similar results on nonlocal elliptic problems we refer the reader also to \cite{ai,chen}. It is important to emphasize that $(0, 0)$ is not a trivial  solution for the System \eqref{sistema Principal singular}. On the other hand, given any weak solution $(u,v)$ for the System \eqref{sistema Principal singular}, we obtain that $u$ and $v$ are nonzero in $\Omega \subset \mathbb{R}^n$ for each subset $\Omega$ of positive Lebesgue measure. Our approach applies the minimization method in Nehari set which is related to the nonlinear Rayleigh quotient. Namely, there exists $\lambda^* > 0$ such that for each $\lambda \in (0, \lambda^*]$ the Nehari can be applied. 
	
	Singular elliptic problems considering local or nonlocal terms have been studied in recent years, see \cite{bai,canino, co,ga,gh,gia} and references therein. The main obstacle here is to apply variational methods due to the fact that the energy functional is only continuous. Indeed, looking for the singular term, the Gateaux derivatives for the energy functional are not well-defined in general. Many other types of research have been considered using some tools such as the sub and supersolution methods showing the existence and multiplicity of solutions for singular elliptic problems. On this subject, we refer the interested reader to \cite{giaa,hai,her,hi1,hi2}. 
	
	It is important to stress that the System \eqref{sistema Principal singular} has an associated energy functional where the first derivative and the second derivative do not make sense for each directions. However, we are able to use the Nehari method depending on the size of $\lambda > 0$. More specifically, we define the Nehari set $\mathcal{N}_{\lambda}$ where we can consider a minimization problem. Moreover, we split the Nehari set into three disjoint subsets given by $\mathcal{N}_{\lambda}^+,\mathcal{N}_{\lambda}^-$ and $\mathcal{N}_{\lambda}^0$. Recall also that a point $u$ is named a degenerate point whenever $u \in \mathcal{N}_{\lambda}^0$. Otherwise, the point $u$ is called as non-degenerated. Another difficulty in the present work is to avoid degenerate points where the standard minimization procedure for singular elliptic problems does not work anymore. Hence, for degenerate points, the Nehari set is not a natural constraint. In order to overcome this situation we prove that there exists $\lambda^* > 0$ such that for each $\lambda \in (0, \lambda^*)$ the energy functional does not admit any degenerate point. Furthermore, for $\lambda= \lambda^*$, we prove that the set of degenerate points is always not empty. Hence, our main contribution in the present work is to consider general singular nonlocal elliptic systems defined in the whole $\mathbb{R}^N$ exploring existence and multiplicity of solutions via the Nehari method and the Rayleigh quotient.  
	
	It is worthwhile to stress that there exists another extreme $\lambda_* < \lambda^*$ such that for each $\lambda \in (0, \lambda_*)$ the System \eqref{sistema Principal singular} admits at least one weak solution $(u,v)$ such that $E_{\lambda}(u,v) > 0$ where $E_{\lambda}$ denotes the energy functional. Furthermore, $E_{\lambda}(u,v) = 0$ for $\lambda = \lambda_*$ and $E_{\lambda}(u,v) < 0$ for each $\lambda \in (\lambda_*,\lambda^*)$. In order to use the nonlinear Rayleigh quotient we need to consider the fibering map $t \mapsto R_n(tu, tv)$ where $R_n$ is a suitable continuous functional. In the spirit of \cite{yavdat1,yavdat2} we shall prove that the map $t \mapsto R_n(tu, tv)$  has a unique critical point which is denoted by $t_n(u, v)$ where $(u,v) \neq (0,0)$. On the other hand, we need to prove that the functional $(u,v) \mapsto R_n(t_n(u, v)(u,v)$ is continuous to ensure the strong convergence of minimizers sequences in the Nehari set. The main idea here is to employ the Implicit Function Theorem showing that $(u,v) \mapsto R_n(t_n(u, v)(u,v)$ and $(u,v) \mapsto t_n(u,v)$ are in $C^0$ class.

	Another obstacle in the present work is to ensure that minimizers in the Nehari set provide us weak solutions for our main problem. The main feature here is to apply the same procedure discussed \cite{Yijing2001,silvasing2018}. More specifically, we shall prove that any minimizer $(u,v)$ in the Nehari set is a weak solution for the System \eqref{sistema Principal singular} using some auxiliary nonnegative test functions given by $((u + \varepsilon\phi_1)^+, (v + \varepsilon\phi_2)^+)$ where $(\phi_1, \phi_2)$ is any pair of functions in the Sobolev space. Moreover, looking for general testing functions, the nonlocal term in our main problem brings us some difficulties. Indeed, for any testing function $(\phi_1, \phi_2) \in X$, we need to control the behavior of the Gagliardo semi-norm for the pair $(u,v)$ considering some fine estimates. In this procedure, we prove also that $u, v > 0$ almost everywhere in $\mathbb{R}^N$. Furthermore, we ensure that $\mathcal{N}_{\lambda}^0 = \emptyset$ for each $\lambda \in (0, \lambda^*)$. Therefore, the System $\eqref{sistema Principal singular}$ has at least two positive weak solutions which stay in  $\mathcal{N}_{\lambda}^+$ and $\mathcal{N}_{\lambda}^-$, respectively.
	
	For the case $\lambda = \lambda^*$ we prove that $\mathcal{N}_{\lambda^*}^0 \ne \emptyset$. Hence, we prove existence and multiplicity of solutions for the System \eqref{sistema Principal singular} using an auxiliary sequence $(u_k, v_k) \in \mathcal{N}_{\lambda_k}^+$ and $(z_k, w_k) \in \mathcal{N}_{\lambda_k}^-$ where $0 < \lambda_k < \lambda^*$ and $\lambda_k \to \lambda^*$. The main difficulty here is to guarantee that our main problem does not admit any weak solution in $\mathcal{N}_{\lambda^*}^0$. This feature can be understood as a nonexistence result for weak solutions for the System \eqref{sistema Principal singular} in the set $\mathcal{N}_{\lambda^*}^0$.

	\subsection{Statement of the main results}
	As was told in the introduction the main objective in the present work is to find the existence and multiplicity of weak solutions for the System \eqref{sistema Principal singular}. In order to that we shall explore some extra conditions of $\lambda > 0$ and $p, q$.
	Throughout this work, we assume that 
	\begin{itemize}
		\item[($P_0$)]   $a \in L^{\frac{2}{1 + p}}(\mathbb{R}^N)$ , $b \in L^{\frac{2}{1 + q}}(\mathbb{R}^N)$ and $a(x),b(x) >0 \ \ \forall  x \in \mathbb{R}^N$;
		\item[($P_1$)] $a \notin L^{1}(\mathbb{R}^N)$ and $\beta \ge \frac{2^*_s}{2}(3 - \alpha - p)$;
		\item[($P$)]  $\alpha,\; \beta\;>\; 1,\;\; 0 < p \le q < 1, \;\;2\;      < \;\alpha + \beta \;< 2^*_s, \;\;      \theta > 0, \lambda > 0, N >        2s , s \in (0,1)$;
		\item [($V_0$)] $V_1(x)\ge V_0 > 0$ and $V_2(x)\ge V_0 > 0, \forall x \in \mathbb{R}^N$;
		\item [($V_1'$)]
		$|\{ x \in \mathbb{R}^n; V_i(x) \le M\}| < +\infty, \;\; \  i = 1, 2, \;\;\;\forall \;\; M > 0$.
	\end{itemize}

	Now, we consider our setting by choosing $X\! = \!X_1 \!\times\! X_2$. In the present work, we assume that the integrals are taken in $\mathbb{R}^N$ and
	$$X_j=\left \{u \in H^s(\mathbb{R}^N); \int V_j(x)u^2dx<\infty \right\}, \;\;j = 1, 2.$$  
	Here we consider the fractional Sobolev space as follows $$H^s(\mathbb{R}^N) = \left\{   u \in L^2(\mathbb{R}^N)\;;\; [u]<\infty, \int (-\Delta)^su\varphi dx = \int u(-\Delta)^s\varphi dx \forall \varphi \in C^{\infty}_c(\mathbb{R}^N) \right \}$$ where the Gagliardo seminorm of $u$ is given by $$[u]^2:=\iint\frac{\left[u(x)-u(y)\right]^2}{|x-y|^{N+2s}} dxdy.$$ 
	In $X$ we define the norm and the inner product as follows:
	$$\begin{array}{c} \Vert(u,v)\Vert ^2:=[(u, v)]^2 +\int  V_1(x) u^2 +V_2(x) v^2 dx;\\\\
		\left<(u,v),(\varphi, \psi)\right>  = \int  \int  \frac{\left[ u(x)-u(y)\right] \left[ \varphi(x) - \varphi(y)\right]}{\vert x-y\vert^{N+2s}} + \frac{\left[ v(x)-v(y) \right] \left[ \psi(x) - \psi(y)\right]}{\left [ x-y \right ]^{N+2s}} dxdy + \int  V_1 u \varphi +V_2  v \psi dx.
	\end{array}
	$$
	Notice also that the Gagliardo seminorm of the ordered pair $(u,v)$ is represented by $[(u, v)]^2:=[u]^2+[v]^2,\; (u, v) \in X.$ 
	Furthermore, by using ($V_0$) and ($V_1'$), we guarantee the continuity and the compactness of $X\hookrightarrow L^{r_1}(\mathbb{R}^N)\times L^{r_2}(\mathbb{R}^N)$ for $r_1, r_2 \in [2, 2^*_s)$ where $2^*_s=2N/(N-2s)$, see for instance \cite{bisci,wang}. 
	
	\begin{lem}\label{imersoes Sobolev 2, 2*s} (\cite[Lemma 1]{Vincenzo2}). Suppose $(V_0),(V_1'),s \in (0,1)$ and $2s < N$. Then there exists some  $C = C(N, p, s)>0$ such that
		$\Vert (u, v) \Vert_{r_1 \times r_2} \le C \Vert (u, v) \Vert$ with $(u,v) \in X$ holds for each $r_1, r_2 \in [2, 2^*_s].$
		Thus $X$ is continuously embedded in $L^{r_1}(\mathbb{R}^N) \times L^{r_2}(\mathbb{R}^N)$ where we take the usual norm
		$\Vert (u, v)\Vert_{r_1 \times r_2} := \Vert u \Vert_{r_1} + \Vert v \Vert_{r_2}.$
		Moreover, $X$ is compactly embedded into $L^{r_1}(\mathbb{R}^N)\times L^{r_2}(\mathbb{R}^N)$ for each $r_1, r_2 \in [2, 2^*_s)$.      
	\end{lem}

	The energy functional associated with the System \eqref{sistema Principal singular} denoted by $E_{\lambda}: X \to \mathbb{R}$ can be written as follows:
	\begin{eqnarray} \label{E cap 2}
		E_{_\lambda}(u,v) \nonumber =   \frac{1}{2}\Vert (u, v) \Vert^2 - \frac{\lambda}{1 - p}\int a(x)\vert u \vert ^{1 - p}dx - \frac{\lambda}{1 - q} \int b(x) \vert v \vert ^{1 - q}dx 
		-   \frac{\theta}{\alpha + \beta}\int\vert u \vert^{\alpha}\vert v \vert^{\beta} dx, \;\;\; (u, v) \in X.
	\end{eqnarray}
	We say that $(u, v)\in X$ is a weak solution to the System \eqref{sistema Principal singular} if and only if, for every $(\varphi_1, \varphi_2) \in X$, we obtain that 
	
	\begin{eqnarray}\label{equfracasingula}
		\left<(u, v), (\varphi_1, \varphi_2)\right> -  \frac{\theta}{\alpha + \beta}\int  \alpha\vert u\vert^{\alpha - 2}u\varphi_1\vert v\vert^{\beta}-\beta\vert u\vert^{\alpha}\vert v\vert^{\beta-2}v\varphi_2dx  - \lambda \int  a(x)\vert u\vert^{-p}\varphi_1 -  b(x)\vert v\vert^{-q}\varphi_2dx = 0.&&
	\end{eqnarray}
	
	\begin{prop}\label{uvnaonulas} Assume ($P_0$) and ($P$).  Let $(u, v) \in X$ be such that \eqref{equfracasingula} is satisfied. Then the sets $\{x\;;\;u(x)=0\}$ and $\{x\;;\;v(x)=0\}$ have zero Lebesgue measure.
	\end{prop}
	\begin{proof} Assume that $\Omega_o:=\{x\;;\;u(x)=0\}$ has positive measure. Choosing any positive function $\varphi\in C_o^{\infty}(\Omega)$ where $|supp\;\varphi\cap\Omega_o|>0$ we obtain that 
		$$\frac{1}{t}\int_{\Omega_o^c} a(x)\left[|u(x)+t \varphi(x) |^{1 - p}-|u(x)|^{1-p}\right]dx\to \int_{\Omega_o^c} a(x)|u|^{-1-p}u\varphi dx,\;t\to 0.$$ 
		On the other hand, by using the fact that $\lim\limits_{t \to 0}\frac{1}{t^p}\int_{\Omega_o} a(x)\varphi^{1 - p}(x) dx$ does not exist and taking into account that $P(u): = \int a(x)|u|^{1-p}dx$, we infer that the limit $$P'(u)\varphi = \lim\limits_{t \to 0} (\frac{1}{t}\int_{\Omega_o^c} a(x)\left[|u(x)+t \varphi(x) |^{1 - p}-|u(x)|^{1-p}\right]dx\;\;+\;\;\frac{1}{t}\int_{\Omega_o} a(x)|t \varphi(x) |^{1 - p}dx)$$
		does not make sense. This ends the proof.
	\end{proof} 
	Recall that $E_{_\lambda} \in C^0(X, \mathbb{R}^N)$ since we are looking singular elliptic problems. Then the energy functional is not differentiable anymore. However, we compute the expression $E'_{_\lambda}(u,v)(u,v)$ for each $(u,v) \in X$. Namely, we obtain that
	\begin{eqnarray} \label{E' cap 2}
		E'_{_\lambda}(u,v)(u,v)  =    \Vert (u, v) \Vert^2 - \lambda\int a(x)\vert u \vert ^{1 - p}dx - \lambda\int b(x) \vert v \vert ^{1 - q}dx
		-  \theta\int \vert u \vert^{\alpha}\vert v \vert^{\beta} dx.
	\end{eqnarray}
	Similarly, we compute also the second derivative as follows
	\begin{eqnarray} \label{E'' cap 2}
		E''_{_\lambda}(u,v)(u,v)^2 =   2\Vert (u, v) \Vert^2 - \lambda(1 - p)\int a(x)\vert u \vert ^{1 - p}dx - \lambda(1 - q)\int b(x) \vert v \vert ^{1 - q}dx- \theta(\alpha + \beta) \int \vert u \vert^{\alpha}\vert v \vert^{\beta} dx.
	\end{eqnarray}
	Now, we consider the Nehari set for the System \eqref{sistema Principal singular} in the following form: 
	\begin{equation}\label{Nsing} 
		\mathcal{N}_{\lambda} = \left\{(u, v) \in X\backslash (0, 0); E'_{_{\lambda}}(u, v)(u, v) = 0 \right\}.
	\end{equation}
	As usual, we split the Nehari set $\mathcal{N}_{\lambda}$ into three disjoint subsets. Namely, we shall write:
	\begin{equation}\label{N+sing} 
		\mathcal{N}_{\lambda}^+  = \left\{(u, v) \in \mathcal {N}_\lambda; E''_{_{\lambda}}(u, v)(u, v)^2 > 0\right\};
	\end{equation}
	\begin{equation}\label{N-sing} 
		\mathcal{N}^-_{\lambda} = \left\{(u, v) \in \mathcal {N}_\lambda; E''_{_{\lambda}}(u, v)(u, v)^2 < 0 \right\};
	\end{equation}
	\begin{equation}\label{N0sing} 
		\mathcal{N}^0_{\lambda} = \left\{(u, v) \in \mathcal {N}_\lambda; E''_{_{\lambda}}(u, v)(u, v)^2 = 0 \right\}.
	\end{equation}
	Hence, we obtain $(u, v) \in \mathcal{N}_{\lambda}$ if and only if $$\lambda =  \frac{\Vert(u,v)\Vert ^2 -\theta \int  \vert u \vert^\alpha \vert v \vert^\beta dx}{\int  a(x) \vert u\vert^{1 - p}dx + \int  b(x) \vert v \vert^{1 - q}dx }.$$
	Similarly, $E_{\lambda}(u, v) = 0$ holds if and only if 
	$$\lambda =  \frac{\frac{1}{2}\Vert(u,v)\Vert ^2 -\frac{\theta}{\alpha + \beta} \int  \vert u \vert^\alpha \vert v \vert^\beta dx}{\frac{1}{1 - p}\int  a(x) \vert u\vert^{1 - p}dx + \frac{1}{1 - q}\int  b(x) \vert v \vert^{1 - q}dx }.$$
	
	Now, taking the nonlinear Rayleigh quotient, we define an auxiliary set where the coupled term for the System \eqref{sistema Principal singular} does not vanish. In fact, we  consider the following set:
	$$\mathcal{A}=\left \{ (u, v) \in X;  \int  |u|^\alpha|v|^\beta dx > 0 \right\}.$$
	Furthermore, we define the functionals $R_n, R_e: \mathcal{A} \rightarrow \mathbb{R}$ of $C^0 class$, for each parameters $\lambda >0$ and $\theta$, as follows:
	\begin{equation}
		R_n(u,v)=\frac{\Vert(u,v)\Vert ^2 -\theta  \int  \vert u \vert^\alpha \vert v \vert^\beta dx}{\int  a(x) \vert u\vert^{1 - p}dx + \int  b(x) \vert v \vert^{1 - q}dx}\;\;\;\;\;and
		\;\;\;\;\; R_e(u, v) =\frac{\frac{1}{2}\Vert(u,v)\Vert ^2 -\frac{\theta}{\alpha + \beta}   \int  \vert u \vert^\alpha \vert v \vert^\beta dx}{\frac{1}{1 - p} \int  a(x) \vert u\vert^{1 - p}dx + \frac{1}{1 - q} \int  b(x) \vert v \vert^{1 - q}dx}.
	\end{equation}
	Hence, we define the following extremes:
	$$\lambda^* := \inf\limits_{(u,v)\in \mathcal{A}}\max\limits_{t>0}R_n(tu, tv) \ \ \ \ and \ \ \ \ \lambda_* := \inf\limits_{(u,v)\in \mathcal{A}}\max\limits_{t>0}R_e(tu, tv).$$ 
	Hence, we have some interactions between the fibers of the Rayleigh quotient and the energy functional. In fact, we obtain that
	\begin{rmk}\label{Rn lambda e E'sing}
		Let $(u, v) \in \mathcal{A}$ be fixed. Then, we obtain the following statements:\\
		i) $ R_n(u, v) = \lambda \Leftrightarrow E'_{_{\lambda}}(u, v)(u, v) = 0;$\hspace{3mm}
		ii)$ R_n(u, v) > \lambda \Leftrightarrow E'_{_{\lambda}}(u, v)(u, v) > 0$;\hspace{3mm}
		iii) $R_n(u, v) < \lambda\Leftrightarrow E'_{_{\lambda}}(u, v)(u, v) < 0$.\\
		iv) $R_e(u, v) = \lambda\Leftrightarrow E_{_{\lambda}}(u, v) = 0;$\hspace{9mm}
		v) $R_e(u, v) > \lambda \Leftrightarrow E_{_{\lambda}}(u, v) > 0;$\hspace{11mm}
		vi) $R_e(u, v) < \lambda \Leftrightarrow E_{_{\lambda}}(u, v) < 0$.
		
	\end{rmk}
	
	Now, we state our main results without any restriction on the parameter $\theta > 0$. Firstly, we shall consider the following result:
	\begin{theorem}\label{teor principal 01sing}
		Assume ($P_0$), ($P$), ($V_0$) and ($V_1'$). Then we have that $0 < \lambda_* < \lambda^* < \infty$. Furthermore,  for each $\lambda \in (0, \lambda^*)$, the System \eqref{sistema Principal singular} admits at least one ground state solution $(u, v) \in \mathcal{A}$. Moreover, $(u, v)$ satisfies the following statements:
		\noindent i)$ E''_{_{\lambda}}(u, v)(u, v)^2 > 0$, that is, $(u, v) \in \mathcal{N}_{\lambda}^+ \cap \mathcal{A}$; 
		ii)There exists $C < 0$ such that $E_{_{\lambda}}(u, v) \le C$.
	\end{theorem}
	
	Now, we shall consider the following result:
	\begin{theorem}\label{teor principal 02sing}
		Suppose ($P_0$), ($P$), ($V_0$),($V_1'$) and $\lambda \in (0, \lambda^*).$ Then System \eqref{sistema Principal singular} admits at least one weak solution $(z, w) \in \mathcal{A}$ satisfying the following assertions:
		i)$ E''_{_{\lambda}}(z, w)(z, w)^2 < 0$, that is, $(z, w) \in \mathcal{N}^-_{\lambda} \cap \mathcal{A} = \mathcal{N}^-_{\lambda}$;\hspace{5mm}ii) If $\lambda \in (0, \lambda_*)$ then $E_{_{\lambda}}(z, w) > 0$; iii) Assume also that $\lambda = \lambda_*$. Then, we obtain that $E_{_{\lambda}}(z, w) = 0$;
		iv) For each $\lambda \in (\lambda_*, \lambda^*)$ we deduce that that $E_{_{\lambda}}(z, w) < 0$.
	\end{theorem}

	As a consequence, we obtain that following result:
	\begin{theorem}\label{teor principal 03cap2}
		Suppose ($P_0$), ($P$), ($V_0$) and ($V_1'$).
		Then, the System \eqref{sistema Principal singular} has at least two weak solutions  $(u, v)$ and $(z, w)$ for each $\lambda \in (0, \lambda^*)$.
		Furthermore, the functions $u, v, z$ and $w$ are strictly positive a.e. in $\mathbb{R}^N$.
	\end{theorem}
	
	Now, assuming $\lambda = \lambda^*$, we use an auxiliary sequence proving that our main problem does not admit any weak solution in $\mathcal{N}_\lambda^0$. Precisely, we prove the following result: 
	
	\begin{theorem}\label{teor principal 04cap2}
		Suppose ($P_0$), ($P_1$) ($P$), ($V_0$), ($V_1'$) and $\lambda = \lambda^*$. Then, the System $(S_{\lambda^*})$ admits at least two weak solutions $(u_*, v_*) \in \mathcal{N}_{\lambda^*}^+$ and $(z_*, w_*) \in \mathcal{N}_{\lambda^*}^-$. Furthermore, the functions $u_*, v_*, z_*$ and $w_*$  are strictly positive a.e. in $\mathbb{R}^N$. 
	\end{theorem}	
	
	\subsection{Outline} The present paper is organized as follows: In the forthcoming section we use the nonlinear Rayleigh quotient in order to lead with the minimization method in the Nehari set. Section $3$ is devoted to the proof of existence and multiplicity of solutions for our main problem whenever $\lambda=\lambda^*$. In Section 4 we prove our main results taking into account the Nehari method.

	\subsection{Notation} Throughout this work we shall use the following notation:
	
	$\bullet$ $E''_{_\lambda}(u, v)((u, v)(u, v)) = E''_{_\lambda}(u, v)(u, v)^2$ denotes the second derivatives in the $(u,v)$ direction.
	
	$\bullet$ $\|\cdot\|_{r}$ and $\|\cdot\|_{\infty}$ are the norms in $L^{\infty}(\mathbb{R}^N)$ and $L^{r}(\mathbb{R}^N)$ for each $r \in [1, \infty)$.
	
	$\bullet$ $S_r$ is the best constant for the embedding $X\hookrightarrow L^r(\mathbb{R}^N)\times  L^r(\mathbb{R}^N),\;r \in [2, 2_s^*]$.
	
	$\bullet$ $B_\epsilon = B_\epsilon(u, v) = \{(w, z) \in X: \Vert (u, v) - (w, z) \Vert < \epsilon\}$ and $B_\delta(r) = \{x \in \mathbb{R}^N: \vert x - r \vert < \delta \}$.

	$\bullet$ $A(u, v) := \Vert(u, v) \Vert^2,\;\;\; B(u, v) := \int\vert u \vert^{\alpha} \vert v \vert ^{\beta}dx \;\;\; P(u) := \int a(x) \vert u \vert ^{1 - p}dx, \;\;\; Q(v) := \int b(x) \vert v \vert^{1 - q}dx.$

	\section{Preliminary results and variational setting}
	
	In this section, we shall consider some results related with the Nehari method together with the nonlinear Rayleigh quotient. Firstly, we begin defining the fibers maps in the following way:
	\begin{equation}\label{R_n cap2} 
		Q_n(t) = R_n (tu, tv) = \frac{t^2A(u,v) - \theta t ^{\alpha + \beta }B(u, v)}{t^{1 - p}P(u) + t^{1 - q}Q(v)};
	\end{equation}
	\begin{equation}\label{R_e cap2} 
		Q_e(t) = R_e(tu, tv) =\frac{\frac{1}{2}t^2A(u,v) -t^{\alpha + \beta}\frac{\theta}{\alpha + \beta}  B(u, v)}{\frac{1}{1 - p}t^{1 - p}P(u) + \frac{1}{1 - q}t^{1 - q}Q(v)}.
	\end{equation}
	Under our assumptions we emphasize that $0 < p \le q < 1$. Thus, we mention that
	\begin{equation}\label{QntpeqQntgrand}
		\lim\limits_{t \to 0} \frac{Q_n(t)}{t^{1 + q}} \ge \frac{A(u, v)}{P(u) + Q(v)} > 0
		\;\;\;\;\;\;and\;\;\;\;\;\;\;
		\lim\limits_{t \to +\infty} \frac{Q_n(t)}{t^{\alpha + \beta - 1 + p}} \le \frac{-\theta B(u, v)}{P(u) + Q(v)} < 0.
	\end{equation}
	The main idea is to employ the same strategy as was done in \cite{yavdat1,yavdat2,cho}. Firstly, we obtain the following result:
	
	\begin{prop}\label{tn(u,v) únicosing}
		Suppose ($P_0$) and ($P$). Then for every $(u, v)\in \mathcal{A}$ there exists a unique  $t = t_n(u, v)$, such that $Q_n'(t) = 0$. Furthermore, $t_n: \mathcal{A} \to \mathbb{R}$ is a function of class $C^0(\mathcal{A}, \mathbb{R})$.
	\end{prop}
	\begin{proof}
		According to Lemma \ref{apendice1} we write
		$A = A(u, v), \, B = B(u, v),\, C = P(u),\, D = Q(v)$ and $\eta = \alpha + \beta.$ 
		Then, the identity $Q_n'(t) = 0$ has a unique solution $t_n(u, v)$ which is a point of global maximum for $Q_n(t)$. Now, we define the function $F: (0, +\infty)\times X \to \mathbb{R}$ given by
		\begin{equation}\label{Equa implíc sing} 
			F(t, (w_1, w_2)) = R_n'(t(z_1 + w_1, z_2 + w_2))(t(z_1 + w_1, z_2 + w_2)), \;t > 0, \;(w_1, w_2) \in X.
		\end{equation}
		Notice also that 
		$Q_n'(1)= F(1, (0, 0)) = R_n'(z_1, z_2)(z_1, z_2) = 0$ and $\partial_t F(1, (0, 0)) \ne 0$. Hence, by using the Implicit Function Theorem \cite[Remark 4.2.3]{Dra}, we guarantee that there exists a unique function $t_n: B_{\varepsilon}(z_1, z_2) \to B_1 \subset \mathbb{R}$ which belongs to $C^0(B_{\varepsilon}(z_1, z_2), \mathbb{R})$ such that $F(t_n(w_1,w_2), (w_1,w_2)) = 0$ holds for each $(w_1,w_2) \in B_\epsilon(z_1,z_2)$. Recall that $(z_1, z_2) \in \mathcal{A}$ is arbitrary which ensures that $t_n \in C^0(\mathcal{A}, \mathbb{R})$. This ends the proof.  
	\end{proof}
	
	Consider the functional $\Lambda_n : \mathcal{A} \to \mathbb{R}$ given by
	\begin{eqnarray}\label{def Lambda_nsing}
		\Lambda _n(u, v) &:=& \max\limits_{t>0}R_n(tu, tv) = Q_n(t_n(u,v)(u, v))= \frac{ \Vert(t_n(u,v)(u, v))\Vert ^2 -\theta  \int \vert t_n(u, v) u\vert^\alpha \vert t_n(u,v) v\vert^\beta }{  \int a(x) \vert t_n(u,v)u \vert^{1 - p} +  \int b(x) \vert t_n(u,v)v \vert^{1 - q}}.
	\end{eqnarray}
	Recall that $\Lambda_n$ is given by composition of continuous functions. Therefore, we conclude that $\Lambda_n \in C^0(\mathcal{A}, \mathbb{R})$.
	\begin{rmk}
		For the particular case $p = q$ the solution of $Q_n'(t) = 0$ is given explicitly in the following form:
		$$t = t_n(u, v) = \left (\frac{(1 + p) A(u, v)}{\theta(\alpha + \beta - 1 + p)B(u, v) }\right )^{\frac{1}{\alpha + \beta - 2}}.$$
		Hence, we see that
		\begin{equation}\label{Lambda n p=q} 
			\Lambda _n(u, v) := C_{\alpha, \beta, p, \theta} \frac{A(u, v)^{\frac{\alpha + \beta - 1 + p}{\alpha + \beta - 2}}}{B(u, v)^{\frac{1 + p}{\alpha + \beta - 2}}\left (P(u) + Q(v)\right )};\;\;\;\;\; C_{\alpha, \beta, p, \theta} = 
			\frac{(1 + p)^{\frac{1 + p}{\alpha + \beta - 2}}(\alpha + \beta - 2)}{\theta^{\frac{1 + p}{\alpha + \beta - 2}}(\alpha + \beta - 1 + p)^{\frac{\alpha + \beta - 1 + p}{\alpha + \beta - 2}}}.
		\end{equation}
		
	\end{rmk}
	\begin{prop}\label{lambda zero homsing}
		Suppose ($P_0$) and ($P$).  Then the functional $\Lambda_n$ is zero homogeneous.
	\end{prop}
	\begin{proof}
		Consider $s>0$. Notice that $\Lambda _n(su, sv) = \sup\limits_{t>0}Q_n(tsu, tsv) = \sup\limits_{a>0} Q_n(au, av) = \Lambda _n(u, v).$ This ends the proof. 
	\end{proof}

	In order to study where the energy functional is positive, negative or zero we need to study the behavior of the function provided in \eqref{R_e cap2}. Analogously, as was done for the behavior of $Q_n$, we analyze the fibering map for $Q_e$. More precisely, we obtain that
	$$\lim\limits_{t \to 0} \frac{Q_e(t)}{t^{1 + q}} \ge \frac{\frac{1}{2}A(u, v)}{\frac{1}{1 - p} P(u) + \frac{1}{1 - q}P(v)} > 0;\;\;\;\;\; \lim\limits_{t \to \infty} \frac{Q_e(t)}{t^{\alpha + \beta - 1 + p}} \le -\frac{\frac{\theta}{\alpha + \beta} B(u, v)}{\frac{1}{1 - p} P(u) + \frac{1}{1 - q}P(v)} < 0.$$ 
	\begin{prop}\label{te(u,v) único sing}
		Assume ($P_0$) and ($P$). Then, for every $(u,v)\in \mathcal{A}$, there exists a unique $t = t_e(u, v)$ such that $Q_e'(t) = 0$. Moreover, $(u, v) \to t_e(u, v)$ is in $C^0(\mathcal{A}, \mathbb{R})$ class.
	\end{prop}
	\begin{proof} Here we apply the same ideas discussed in the proof of Proposition \ref{tn(u,v) únicosing}. We omit the details. 
	\end{proof} 
	Similarly, we obtain that $\Lambda_e \in C^0(\mathcal{A}, \mathbb{R})$ where 
	\begin{equation}\label{defLambda_esing} 
		\Lambda _e(u, v) = \max\limits_{t>0} R_e(tu, tv) = R_e(t_e(u, v) (u,v)).
	\end{equation}
	
	\begin{prop}\label{diferença impsing}
		Suppose ($P_0$), ($P$), ($V_0$), ($V_1'$). Let us consider $(u, v) \in \mathcal{A}$. Then we conclude that
		\begin{equation}\label{eq dif impsing} 
			Q_n(t) - Q_e(t)=\frac{t}{(1 - p)(1 - q)}\left ( \frac{(1 - q)t^{1 - p} P(u) + (1 - p) t^{1 - q}Q(v)}{t^{1 - p} P(u) + t^{1 - q} Q(v)} \right )Q_e'(t).
		\end{equation}
		
	\end{prop}
	\begin{proof} The proof follows by a simple computation. We omit the details. \end{proof}

	\begin{rmk} It follows from \eqref{eq dif impsing} and Remark \ref{ex3.18Elon analise reta} the following estimates: 
		\begin{equation}\label{dif inequalities} 
			t(1-p)^{-1}Q_e^{\prime}(t) \;\;\;\leq \;\;\; Q_n(t) - Q_e(t)\;\;\; \leq \;\;\; t(1-q)^{-1}Q_e^{\prime}(t).
		\end{equation} 
		Moreover, 
		$Q_n(t) > Q_e(t)$ if and only if $ Q_e^{\prime}(t) > 0$ which makes sense only for $0 < t < t_e(u, v)$.
		Similarly, $Q_n(t) < Q_e(t)$ holds if and only if $Q^{\prime}_e(t) < 0$ which makes sense only for $t > t_e(u, v)$. Furthermore, we mention that
		$Q_n(t) = Q_e(t)$ holds if and only if $Q_e^{\prime}(t) = 0$ which occurs only for $t = t_e(u, v)$.
	\end{rmk}
	
	\begin{rmk}
		For the particular case where $p = q$, the value $t_e(u, v)$ is explicitly given by 
		\begin{equation}\label{t_e sing} 
			t_e(u,v) = \left ( \frac{\left ( 1 + p \right )\left ( \alpha + \beta \right ) A(u, v)}{ 2\theta\left (\alpha +\beta - 1 + p \right ) B(u, v)} \right )^ \frac{1}{\alpha +\beta -2}.
		\end{equation}
		Furthermore, we mention that 
		\begin{equation}\label{Lambda_e = R_e (te(u, v)) q sing} 
			\Lambda _e(u, v) = \tilde{C}_{\alpha, \beta, p, \theta} \frac{A(u, v)^{\frac{\alpha +\beta - 1 + p}{ \alpha + \beta -2}}}{B(u, v)^{\frac{1 + p}{\alpha + \beta - 2}}\left (P(u) + Q(v)\right)};
		\end{equation}
		\begin{equation}\label{Ctilalphabetaqtetasing}
			\tilde{C}_{\alpha, \beta, q, \theta} = (1 - p) (\alpha + \beta - 2) \left ( \frac{(1 + p)(\alpha + \beta)}{\theta} \right )^{\frac{1 + p}{\alpha + \beta - 2}} \left ( \frac{1}{2(\alpha + \beta - 1 + p)} \right )^{\frac{\alpha + \beta - 1 +p}{\alpha + \beta - 2}}.
		\end{equation}
	\end{rmk}

	\begin{rmk}\label{Qn - Qe, p=qsing}
		Assume that $p = q$. Hence, we obtain that
		$Q_n(t) - Q_e(t)=t(1-p)^{-1}Q_e^{\prime}(t)$.
	\end{rmk}
	\begin{prop}\label{seqeintlongedozero} Assume that ($P_0$),($P$),($V_0$) and ($V_1'$) hold. Consider $(u_k, v_k) \in \mathcal{A}$ a minimizing sequence for $\lambda^*$. Then the auxiliary sequence $(\tilde{u}_k, \tilde{v}_k) \in \mathcal{A}$ given by
		$\tilde{u}_k = t_n(u_k, v_k) u_k,\,\tilde{v}_k = t_n(u_k, v_k)v_k,
		$
		is also a minimizing sequence. Moreover, there exist $\rho,\; \tilde{\rho} > 0$ such that $\Vert(\tilde{u}_k, \tilde{v}_k)\Vert \ge \tilde{\rho}$ and $\theta  \int \vert \tilde{u}_k \vert ^\alpha \vert \tilde{v}_k \vert ^\beta dx \ge \rho$.
	\end{prop}
	\begin{proof}
		Firstly, by using Proposition \ref{lambda zero homsing}, the functional $\Lambda_n$ is zero homogeneous and $(\tilde{u}_k, \tilde{v}_k)$ is also a minimizing sequence. Furthermore, we see that 
		$$ \left.\frac{d}{dt} R_n(t \tilde{u}_k, t\tilde{v}_k)\right|_{t=1}=0.$$
		Hence, we deduce that
		$$ \left ( 2A(\tilde{u}_k, \tilde{v}_k) - \left (\alpha + \beta \right )\theta B(\tilde{u}_k, \tilde{v}_k) \right )\left ( P(\tilde{u}_k) + Q(\tilde{v}_k) \right ) = \left ( A(\tilde{u}_k, \tilde{v}_k) - \theta B(\tilde{u}_k, \tilde{v}_k)\right )\left ( (1 - p)(P(\tilde{u}_k) + (1 - q)Q(\tilde{v}_k)\right ). $$
		Consequently, we obtain 
		$$2A(\tilde{u}_k, \tilde{v}_k) \left [ P(\tilde{u}_k) + Q(\tilde{v}_k) \right ] - \theta(\alpha + \beta) B(\tilde{u}_k, \tilde{v}_k)\left [ P(\tilde{u}_k) + Q(\tilde{v}_k) \right ] =$$
		$$A(\tilde{u}_k, \tilde{v}_k)\left [ (1 - p) P(\tilde{u}_k) + (1 - q)Q(\tilde{u}_k) \right ] - \theta B(\tilde{u}_k, \tilde{v}_k)\left [ (1 - p) P(\tilde{u}_k) + (1 - q)Q(\tilde{u}_k) \right ].$$
		Similarly, we see that
		\begin{equation}\label{integrates coupling}
			\theta B(\tilde{u}_k, \tilde{v}_k) = \frac{A(\tilde{u}_k, \tilde{v}_k ) \left [ \left (p + 1 \right ) P(\tilde{u}_k) + \left (q + 1\right ) Q(\tilde{v}_k) \right ]}{\left [ \left ((\alpha + \beta \right ) - 1 + p) P(\tilde{u}_k) + \left ((\alpha + \beta \right ) - 1 + q) Q(\tilde{v}_k)\right ]}.
		\end{equation}
		Now, by using Lemma \ref{ex3.18Elon analise reta} and taking into account the fact that the function $f(x) = \frac{1 + x}{(\alpha +\beta) - 1 + x}$ is increasing, we deduce that  
		\begin{equation}\begin{array}{lll}\label{intervalo da integralização}
				f(p) \;A(\tilde{u}_k, \tilde{v}_k) \;\;\leq\;\; \theta B(\tilde{u}_k, \tilde{v}_k) \;\;\leq \;\;f(q)\; A(\tilde{u}_k, \tilde{v}_k).  
			\end{array}
		\end{equation}
		Hence, by applying H\"older's inequality and Sobolev embedding, see Lemma \ref{imersoes Sobolev 2, 2*s}, we infer that
		\begin{equation}\begin{array}{lll}\label{desdeHolder1}
				f(p)\Vert (\tilde{u}_k, \tilde{v}_k) \Vert^2 \leq \theta \Vert \tilde{u}_k \Vert_{\alpha + \beta}^\alpha \Vert \tilde{v}_k \Vert_{\alpha + \beta}^\beta \leq \theta || (\tilde{u}_k,\tilde{v}_k)||_{\alpha+\beta}^{{\alpha+\beta}} \leq \theta S_{\alpha + \beta}^{\alpha + \beta}\Vert (\tilde{u}_k, \tilde{v}_k) \Vert^{\alpha + \beta}.
			\end{array}
		\end{equation}
		As a consequence, we see that 
		\begin{equation}\begin{array}{lll}\label{limitaçao inf do módulosing}
				\Vert (\tilde{u}_k, \tilde{v}_k) \Vert \ge \left ( \frac{f(p)}{\theta S^{\alpha +\beta}_{\alpha + \beta}}\right )^{\frac{1}{\alpha + \beta -2}} = \tilde{\rho}.
			\end{array}
		\end{equation}
		In view of \eqref{limitaçao inf do módulosing} and \eqref{intervalo da integralização} we obtain that the minimizing sequence is far away from the boundary of $\partial\mathcal{A}$. In fact, we observe that 
		\begin{equation}\label{hrosing} 
			\theta B(\tilde{u}_k, \tilde{v}_k)\;=\;\theta  \int \vert \tilde{u}_k \vert ^\alpha \vert \tilde{v}_k \vert ^\beta dx \;\geq\;\left(\frac{f^{\alpha+\beta}(p)}{\theta^2S_{\alpha+\beta}^{2(\alpha+\beta)}}\right)^{1/(\alpha+\beta-2)}:=\rho.
		\end{equation}
		This finishes the proof. 
	\end{proof}
	\begin{prop}
		Suppose ($P_0$), ($P$), ($V_0$) and ($V_1'$). Then there exists $C_\rho > 0$ such that $\lambda^*\ge C_\rho > 0.$ 
	\end{prop}
	\begin{proof}
		Firstly, by using \eqref{intervalo da integralização}, we see that 
		$$\Lambda_n(\tilde{u}_k, \tilde{v}_k) \geq [1 - f(q)] \frac{||(\tilde{u}_k, \tilde {v}_k)||^2}{P(\tilde{u}_k)+Q( \tilde {v}_k)} .$$
		It follows from Sobolev embedding and  H\"older inequality that 
		\begin{equation}\label{bounded}
			P(\tilde{u}_k)\leq S_2^{1-p}||a||_{2/(p+1)}||(\tilde{u}_k, \tilde{v}_k)||^{1-p}\leq S||(\tilde{u}_k, \tilde{v}_k)||^{1-p}, \,\, Q(\tilde{v}_k)\leq S_2^{1-q}||b||_{2/(q+1)}||(\tilde{u}_k, \tilde{v}_k)||^{1-q}\leq S||(\tilde{u}_k, \tilde{v}_k)||^{1-q}
		\end{equation}
		where $S := \max\{\Vert a \Vert_{\frac{2}{1 + p}}S_2^{1 - p}, \Vert b \Vert_{\frac{2}{1 + q}}S_2^{1 - q} \}$. Therefore, by using Remark \ref{ex3.18Elon analise reta},
		we obtain 
		$$\Lambda_n(\tilde{u}_k, \tilde{v}_k) \ge \frac{1 - f(q)}{2S} \min \{\Vert (\tilde{u}_k, \tilde{v}_k)\Vert^{1 + p}, \Vert (\tilde{u}_k, \tilde {v}_k)\Vert^{1 + q}\}.$$
		Now, by using \eqref{limitaçao inf do módulosing}, we see that 
		$\Lambda_n(\tilde{u}_k, \tilde{v}_k) \ge \frac{1 - f(q)}{2S} \max \left\{ \tilde{\rho}^{1 + q}, \tilde{\rho}^{1 + p} \right \}=: C_{\tilde{\rho}}.$ This ends the proof. 
	\end{proof}
	\begin{prop}\label{Limitadasing}
		Suppose ($P_0$), ($P$), ($V_0$) and ($V_1'$). Then the minimizing sequence provided by Proposition \ref{seqeintlongedozero} is bounded.
	\end{prop}
	\begin{proof}
		Up to a subsequence, we see that $\Lambda_n(\tilde{u}_k,\tilde{v}_k) \le \lambda^* + \frac{1}{k}.$ Without loss of generality we suppose that $\Vert(\tilde{u}_k, \tilde{v}_k)\Vert \ge 1$. Therefore, by using \eqref{bounded}, we obtain  
		\begin{equation}\label{desholdersing}
			\begin{array}{lll}
				\begin{array}{c}       
					\Vert(\tilde{u}_k,\tilde{v}_k)\Vert ^2 \le f(q)\Vert (\tilde{u}_k, \tilde{v}_k) \Vert ^2 + \left ( \lambda^* +\frac{1}{k}\right ) S \left ( \Vert (\tilde{u}_k, \tilde{v}_k) \Vert ^{p+1} + \Vert (\tilde{u}_k, \tilde{v}_k) \Vert ^{q+1} \right). 
				\end{array}
			\end{array}
		\end{equation}
		The last assertion implies that
		\begin{equation}\begin{array}{lll}
				\begin{array}{c}
					\Vert(\tilde{u}_k,\tilde{v}_k)\Vert \le \left (\frac{\left ( \lambda^* +\frac{1}{k}\right )2S}{1 - f(q)} \right )^{\frac{1}{1 - q}}.
				\end{array}
			\end{array}
		\end{equation}
		This ends the proof. 
	\end{proof}
	\begin{prop}\label{lambda frac Sem Infsing}
		Assume ($P_0$), ($P$), ($V_0$) and ($V_1'$). Then the sequence given by Proposition \ref{seqeintlongedozero} satisfies
		\begin{equation}\label{fsirest}
			\Lambda_n(\tilde{u},\tilde{v}) \le \lim\inf\limits_{k \to \infty}\Lambda_n(\tilde{u}_k,\tilde{v}_k),   
		\end{equation}
		where $(\tilde{u}_k, \tilde{v}_k) \rightharpoonup (\tilde{u}, \tilde{v})$ in $X$.
	\end{prop}
	\begin{proof}		
		From the boundedness of minimizing sequence we suppose that $(\tilde{u}_k, \tilde{v}_k) \rightharpoonup (\tilde{u}, \tilde{v})$ in $X$.  Using the compact embedding, $X_i\hookrightarrow \hookrightarrow L^{r_i}(\mathbb{R}^N)$ for $r_i \in [2, 2^*_s)$, $i = 1, 2$,we obtain the following statements:

		\begin{equation}\left\{\begin{array}{lll}
				\tilde{u}_k \rightharpoonup \tilde{u}, & in\;\; X_1;\\
				\tilde{u}_k \rightarrow  \tilde{u}, & in \;\; L^{r_1}({\mathbb{R}^N}); \\
				
			\end{array}\right.\;\;and\;\;\left\{\begin{array}{lll}\label{Seq limitada uksing} 
				\tilde{u}_k(x) \rightarrow  \tilde{u}(x) & a.e. \;\; {\mathbb{R}^N}; \\
				|\tilde{u}_k| \le h_{r_1} & \in L^{r_1}({\mathbb{R}^N})\;\; r_1\in [2, 2^*_s).
			\end{array}\right.
		\end{equation}
		Notice that 
		$$\left |\int a(x)\left[ \vert \tilde{u}_k \vert ^{1 - p}- \vert \tilde{u} \vert ^{1 - p}\right]dx\right |\;\;\leq\;\; \Vert a \Vert_{\frac{2}{1 + p}} \left ( \int  \left | \vert \tilde{u}_k \vert ^{1 - p} - \vert \tilde{u} \vert ^{1 - p} \right |^\frac{2}{1 - p}dx\right )^{\frac{1 - p}{2}}.$$
		Furthermore, we infer that
		\begin{eqnarray*}
			\left | \vert \tilde{u}_k \vert ^{1 - p} - \vert \tilde{u} \vert ^{1 - p} \right |^\frac{2}{1 - p} &\le& \left | \vert \tilde{u}_k \vert ^{1 - p} + \vert \tilde{u} \vert ^{1 - p} \right |^\frac{2}{1 - p} \le 2^{\frac{2}{1 - p}}\left (\vert \tilde{u}_k \vert ^2 + \vert \tilde{u} \vert ^2 \right ) \leq 2^{\frac{2}{1 - p}}\left (\vert h_2 \vert ^2 + \vert \tilde{u} \vert ^2\right )\in L^1(\mathbb{R}^N).
		\end{eqnarray*}
		Now, by applying the Dominated Convergence Theorem, we deduce that  
		$$\lim\limits_{k\to \infty} \left |\int a(x) \vert \tilde{u}_k \vert ^{1 - p}dx - \int a(x) \vert \tilde{u} \vert ^{1 - p}dx\right |\le  \lim\limits_{k \to \infty}\Vert a \Vert_{\frac{2}{1 + p}} \left ( \int  \left | \vert \tilde{u}_k \vert ^{1 - p} - \vert \tilde{u} \vert ^{1 - p} \right |^\frac{2}{1 - p}dx\right )^{\frac{1 - p}{2}}= 0. $$
		Similarly, we show that $\lim\limits_{k \to \infty}\int b(x) \vert \tilde{v}_k \vert ^{1 - q}dx = \int b(x) \vert \tilde{v} \vert ^{1 - q}dx.$
		From now on, by using the H\"older inequality with the exponents $r = \frac{\alpha + \beta}{\alpha}$ and $r' = \frac{\alpha + \beta}{\beta}$,  we obtain that 
		\begin{eqnarray*}
			\int h_{r_1}^\alpha h_{r_2}^\beta dx &\leq& 	 \Vert h_{\alpha + \beta} \Vert_{\alpha + \beta}^\alpha \Vert h_{\alpha + \beta} \Vert_{\alpha + \beta}^\beta < \infty.\nonumber
		\end{eqnarray*}The Dominated Convergence Theorem implies
		\begin{equation}\label{=integ}
			\lim\limits_{k\to \infty} \int  |\tilde{u}_k|^\alpha |\tilde{v}_k|^\beta dx = \int  |\tilde{u}|^\alpha |\tilde{v}|^\beta dx.  
		\end{equation}
		Putting all these things together and using that the norm is weakly lower-semicontinuous we infer that  
		$$\Lambda_n(\tilde{u},\tilde{v}) \le \liminf\limits_{k \to \infty}\Lambda_n(\tilde{u}_k,\tilde{v}_k).$$
		This ends the proof. 
	\end{proof}
	
	\begin{rmk}\label{E E' E'' infer sem frac}
		If $(u_k, v_k) \rightharpoonup (u, v)$ in $X,$ using the same argument employed in the proof of Proposition \ref{lambda frac Sem Infsing}, we obtain that 
		$$\!\!\!E_{_{\lambda}}(u, v)\! \leq\! \liminf_{k \to \infty}E_{_{\lambda}}(u_k, v_k); \; E'_{_{\lambda}}(u, v)(u, v)\! \leq\! \liminf_{k \to \infty}E'_{_{\lambda}}(u_k, v_k)(u_k, v_k);\;E''_{_{\lambda}}(u, v)(u, v)^2 \!\leq \!\liminf_{k \to \infty}E''_{_{\lambda}}(u_k, v_k)(u_k, v_k)^2.$$
	\end{rmk}
	\begin{prop}\label{Longe de zerosing}
		Suppose hypotheses ($P_0$), ($P$), ($V_0$) and ($V_1'$). Then $\lambda^* = \inf\limits_{(z, w) \in \mathcal{A}}\Lambda_n(z,w)$ is attained.
	\end{prop}
	\begin{proof}
		Firstly, using Propositions \ref{Limitadasing}, the minimizing sequence provided by Proposition \ref{seqeintlongedozero} that is denoted by $(\tilde{u}_k, \tilde{v}_k)$ is bounded. Thus, we deduce that  $(\tilde{u}_k, \tilde{v}_k) \rightharpoonup (\tilde{u}, \tilde{v})$ in $X$ where $ (\tilde{u},\tilde{v})\in \mathcal{A}$. It follows also from Proposition \ref{lambda frac Sem Infsing} that  
		$$\lambda^*\le \Lambda_n(\tilde u, \tilde v) \le \liminf_{k \longrightarrow +\infty} \Lambda_n(\tilde{u}_k, \tilde{v}_k) = \lambda^*$$
		Hence, $\lambda^* = \Lambda_n(\tilde u, \tilde v)$. The last assertion implies that $\lambda^*$ is attained. This ends the proof.  
	\end{proof}
	\begin{lem}
		Assume ($P_0$), ($P$), ($V_0$) and ($V_1'$). Then $\Lambda_n$ is continuous and unbounded from above.    
	\end{lem}
	\begin{proof}
		Let us consider functions in $\mathcal{A}^+ \cap S$ where $\mathcal{A}^+ = \{(u, v) \in \mathcal{A}: u, v > 0\}$ and $S = \{(u, v)\in X: \Vert(u, v) \Vert = 1\}$. Define the continuous function $\mathcal{G}: X \to \mathbb{R}$ given by
		$\mathcal{G}(u, v) = P(u) + Q(v), \;\; (u, v) \in X.$
		In view of the last assertion, we obtain $\mathcal{G}^{-1}((0, \infty))\cap S$ is a relative open set in $S$. Furthermore, $t\;\;\mapsto\;\; \mathcal{G}(tu, tv) \;=\; t^{1 - p}P(u) + t^{1 - q}Q(v)$ is increasing for $t > 0$. 
		Under these conditions, we mention that $\mathcal{G}^{-1}((0, \infty)) \cap S \ne S$. Therefore, there exists a sequence $(u_k, v_k) \in \mathcal{G}^{-1}((0, \infty)) \cap S$ such that $\mathcal{G}(u_k, v_k) \to 0$ in $\mathbb{R}$. Now, we shall use $t_k$ instead of $t_n(u_k, v_k)$ to simplify the notation. Hence, we obtain that
		\begin{equation}\label{Lambdanilimit}
			\lim\limits_{k \to \infty} \Lambda_n(u_k, v_k) = \lim\limits_{k \to \infty} \frac{ \Vert(t_k u_k, t_k v_k)\Vert ^2 -\theta  \int  \vert t_k u_k\vert^\alpha \vert t_k v_k\vert^\beta dx}{  \int a(x) \vert t_k u_k \vert^{1 - p}dx +  \int b(x) \vert t_k v_k \vert^{1 - q}dx}.
		\end{equation}
		Recall also that 
		$$ \left.\frac{d}{dt} R_n(t u_k, t v_k)\right|_{t=t_k}=0.$$
		Now, by using \eqref{integrates coupling} for $t = t_n(u_k, v_k) =: t_k$, we deduce that
		$$\theta  B(t_ku_k, t_kv_k) = \frac{ \Vert(t_ku_k, t_kv_k)\Vert^2\left[(1 + p)P(t_ku_k) + (1 + q)  Q(t_kv_k) \right]} {(\alpha + \beta - 1 + p)P(t_ku_k) + (\alpha + \beta - 1 + q) Q(t_kv_k)}.$$
		In light of \eqref{Lambdanilimit} we deduce that
		$$\Lambda_n(t_ku_k, t_kv_k) = \frac{\Vert(t_ku_k,t_k v_k)\Vert^2\left[ 1 - \frac{(1 + p)P(t_ku_k) + (1 + q)  Q(t_kv_k)}{(\alpha + \beta - 1 + p)P(t_ku_k) + (\alpha + \beta - 1 + q)  Q(t_kv_k)} \right]}{P(t_ku_k) +  Q(t_kv_k)}.$$
		Note that for each $(u_k, v_k) \in S$, we have that $\Vert(u_k, v_k) \Vert = 1$. As a consequence, we obtain that
		$$\Lambda_n(t_ku_k, t_kv_k) = \frac{t_k^2 (\alpha + \beta - 2)}{(\alpha + \beta - 1 + p)t_k^{1 - p}P(u_k) + (\alpha + \beta - 1 + q) t_k^{1 - q} Q(v_k)}.$$
		According to Proposition \ref{seqeintlongedozero} we infer that $\Vert(t_k u_k, t_k v_k)\Vert \ge \tilde{\rho}$ holds. The last statement implies that $t_k \ge \tilde{\rho}$. Hence, by using hypothesis ($P$), we see that 
		$$\Lambda_n(t_ku_k, t_kv_k) \geq \frac{t_k^2}{\max\{t_k^{1 - p}, t_k^{1 - q}\}}\frac{\alpha + \beta -2}{(\alpha + \beta - 1 + q)}\frac{1}{\mathcal{G}(u_k, v_k)}\;\;\to\;\;\infty.$$
		This ends the proof. 
	\end{proof}
	\begin{prop}\label{tn-,tn+sing}
		Assume ($P_0$), ($P$), ($V_0$) and ($V_1$). Then for each $\lambda \in (0, \lambda^*)$ and $(u, v) \in \mathcal{A}$ fixed the fibering map $\gamma_{_{\lambda}}(t) : = E_{_{\lambda}}(tu,tv)$ has exactly two distinct critical points $t_n^+(u, v)$ and $t_n^-(u, v)$ such that $0 < t_n^+(u, v) < t_n(u, v) < t_n^-(u, v)$. Furthermore, we deduce the following statements:
		\begin{itemize}
			\item[$i)$] It holds $ t_n^+(u, v)$ is a local minimum for $\gamma_{_{\lambda}}$ while $t_n^-(u, v)$ is a local maximum. Moreover, $t_n^{\pm}(u, v)(u, v)\in \mathcal{N}_{\lambda}^{\pm}$;
			\item[$ii)$] The functions $(u, v)\mapsto t_n^+(u, v)$ and $(u, v)\mapsto t_n^-(u, v)$ belong to $C^0(\mathcal{A}, \mathbb{R}).$
		\end{itemize}
	\end{prop}
	\begin{proof}
		$(i)$ Firstly, by using \eqref{QntpeqQntgrand} and Lemma \ref{apendice1}, we show that $t \mapsto R_n(tu, tv)$ is increasing for $0 < t < t_n(u, v)$ and decreasing for $t > t_n(u, v).$ Since $\lambda < R_n(t_n(u, v)(u, v))$ we have that $R_n(tu, tv) = \lambda$ has exactly two roots $0 < t_n^+(u, v) < t_n(u, v) < t_n^-(u, v)$. Notice that $t_n^{\pm}(u, v)$ are critical points of  $\gamma_{_{\lambda}}(t) := E_{_{\lambda}}(tu, tv)$, see Remark \ref{Rn lambda e E'sing}. Under these conditions, we have $Q_n'(t_n^+(u, v)) > 0$ and $Q_n'(t_n^-(u, v)) < 0$. Therefore, using Remark \ref{relação entre R_n' e E''sing}, we conclude that $E''_{_{\lambda}}(t_n^+(u, v)(u, v))(t_n^+(u, v)(u, v))^2 > 0$ and $E''_{_{\lambda}}(t_n^-(u, v)(u, v))(t_n^-(u, v)(u, v))^2 < 0$. According to \eqref{N+sing} and \eqref{N-sing} we get $t_n^+(u, v)(u, v) \in \mathcal{N}^+_\lambda$ and $t_n^-(u, v)(u, v) \in \mathcal{N}^-_\lambda$.
		
		$(ii)$ Suppose $\lambda \in (0, \lambda^*)$. Thus, $\mathcal{N}^0_{\lambda} = \emptyset$ and  $\mathcal{N}_{\lambda} = \mathcal{N}^+_{\lambda} \cup \mathcal{N}^-_{\lambda}$. 
		Furthermore, $Q_n \in C^1((0, \infty), \mathbb{R})$, $Q_n'(t_n^+(u, v)) > 0$ and $Q_n'(t_n^-(u, v)) < 0$. Let $(z_1, z_2) \in \mathcal{N}^-_{\lambda}$ be fixed. In order to ensure that $(u, v)\mapsto t_n^+(u, v)$ and $(u, v)\mapsto t_n^-(u, v)$ are in $C^0(\mathcal{A}, \mathbb{R})$, we consider the function $F: (0, +\infty) \times X \to \mathbb{R}$ given by 
		$$
		F(t, (u, v)) =  A(t(z_1 + u, z_2 + v)) - \lambda P(t(z_1 + u)) - \lambda Q(t(z_2 + v)) - \theta B(t (z_1 + u, z_2 +v)).
		$$
		Recall also that
		$$F(1, (0, 0)) = E'_{_{\lambda}}(z_1, z_2)(z_1, z_2) = 0, \,\,\,\, \frac{\partial}{\partial t}F(1, (0, 0)) = E''_{_{\lambda}}(z_1, z_2)(z_1, z_2)^2 < 0.$$\\
		It is important to emphasize that $F\in C^0((0, +\infty) \times X, \mathbb{R})$. Hence, using the Implicit Function Theorem \cite[Remark 4.2.3]{Dra}, 
		there is a function $f \in C^0(B_\varepsilon(0,0), (1 - \delta, 1 + \delta))$  such that
		$$F(f(u, v), (u, v)) = 0,,, (u, v) \in B_{\varepsilon}(0, 0). $$
		Recall also that $B_\varepsilon(z_1, z_2) = \{(w_1, w_2) \in X : \Vert (w_1, w_2) - (z_1, z_2) \Vert < \varepsilon\}$.
		From now on, due to the continuity of the derivative, we infer that $$\frac{\partial}{\partial t}F(f(u, v), (u, v)) < 0, (u,v) \in B_\varepsilon(z_1, z_2).$$
		Now, choosing $(\bar{z_1}, \bar{z_2}) = f(u, v)(z_1 +u, z_2 + v)$, we obtain that
		$$E''_{_{\lambda}}(\bar{z_1}, \bar{z_2})(\bar{z_1}, \bar{z_2})^2 = [f(u, v)]^{-1}\frac{\partial}{\partial t}F(f(u, v), (u, v)) < 0.$$ 
		The last assertion implies that $f(u, v)(z_1 + u, z_2 + v) \in \mathcal{N}^-_{\lambda}$. Hence, we see that
		$$t_n^-(z_1 + u, z_2 + v) = f(u, v), \;\; (u, v) \in B_{\varepsilon}(0, 0).$$
		In fact, we observe that $(z_1 + u, z_2 + v) \in B_\varepsilon(z_1, z_2)$ and $(u, v) \in B_\varepsilon(0, 0)$. Therefore, $t_n^- \in C^0(B_\varepsilon(z_1, z_2), \mathbb{R})$ holds for each $\lambda \in (0, \lambda^*).$
		Since $(z_1, z_2\!) \in \mathcal{N}^-_{\lambda}\!$ is arbitrary, we know that $t_n^-\! \in\! C^0(\mathcal{A}, \mathbb{R}).$ Similarly, $t_n^+ \!\in C^0(\mathcal{A}, \mathbb{R})$ holds true. This finishes the proof.
	\end{proof}
	
	\begin{prop}\label{tR_n'G' = E''sing} 
		Consider $(u, v)\in \mathcal{A}$ and $t>0$ in such a way that $\lambda = R_n(tu,tv)$. Let us consider the function $G: \mathcal{A}\rightarrow \mathbb{R}$ given by $G(u,v) = P(u) + Q(v).$ Then
		$$\frac{d}{dt}R_n(tu, tv) = \frac{E''_{_{\lambda}}(tu, tv) (tu, tv)^2}{t G(tu, tv)}.$$ 
	\end{prop}
	\begin{proof}
		Firstly, we mention that
		\begin{equation}\label{tG1 sing} 
			t\frac{d}{dt}G(tu, tv) = (1 - p) P(tu) + ( 1 - q)Q(tv).
		\end{equation}
		Recall also that
		$$G(tu, tv) R_n(tu, tv) = t^2\Vert (u, v) \Vert^2 - \theta t^{\alpha + \beta}\int  \vert u \vert^{\alpha} \vert v \vert^{\beta}dx.$$
		Hence, we obtain that
		\begin{eqnarray*}
			t\frac{d}{dt}[G(tu, tv)] R_n(tu, tv) + tG(tu, tv) \frac{d}{dt}[R_n(tu, tv)] &=& 2t^2\Vert(u, v) \Vert^2 -\theta(\alpha + \beta) t^{\alpha +\beta}\int  \vert u \vert^{\alpha} \vert v \vert^{\beta}dx.
		\end{eqnarray*}
		Now, by using the fact that $R_n(tu, tv) = \lambda$, we infer that
		$$t\frac{d}{dt}G(tu, tv) \lambda + t G(tu, tv) \frac{d}{dt}R_n(tu, tv) = 2\Vert(tu, tv)\Vert^2 - \theta(\alpha + \beta) \int \vert tu \vert^{\alpha} \vert tv \vert ^{\beta}.$$
		In light of \eqref{tG1 sing} we mention that
		$$ \lambda(1 - p)P(tu) + \lambda(1 - q) Q(tv) + tG(tu, tv) \frac{d}{dt}R_n(tu, tv) = 2 \Vert(tu, tv)\Vert^2 - \theta(\alpha + \beta) B(tu, tv).$$
		Furthermore, we observe that
		\begin{eqnarray*}
			tG(tu, tv) \frac{d}{dt}R_n(tu, tv)=2 \Vert(tu, tv)\Vert^2 - \theta(\alpha + \beta) B(tu, tv) - \lambda(1 - p)P(tu) - \lambda(1 - q) Q(tv) = E''_{_{\lambda}}(tu, tv) (tu, tv)^2.
		\end{eqnarray*}
		This ends the proof. 
	\end{proof}
	\begin{rmk}\label{relação entre R_n' e E''sing} Assume $\lambda = R_n(tu, tv)$. Then $\frac{d}{dt}R_n(tu, tv)$ has the same sign as $E''_{_{\lambda}}(tu, tv)(tu, tv)^2.$\end{rmk}
	
	\begin{prop}
		Assume ($P_0$), ($P$), ($V_0$) and ($V_1$). Then we obtain that $\mathcal{N}^0_{ \lambda^*} \ne \emptyset$.
	\end{prop}
	\begin{proof}
		Recall that $\Lambda_n$ is zero homogeneous. Without loss of generality we suppose that $\lambda^* = \Lambda_n(\tilde{u}, \tilde{v})=R_n(\tilde{u}, \tilde{v})$. In other words,  $t_n(\tilde{u}, \tilde{v})=1$ holds for some $(\tilde{u}, \tilde{v}) \in \mathcal{A}$.  As a consequence, $\frac{d}{dt}R_n(t\tilde{u}, t\tilde{v}) = 0$ at $t=1$. Now, by using Remark \ref{relação entre R_n' e E''sing}, we obtain that  $(\tilde{u}, \tilde{v}) \in \mathcal{N}^0_{\lambda^*}$. This finishes the proof.
	\end{proof}
	\begin{rmk} It holds that $\lambda = \lambda^*$ is the smallest positive value such that $\mathcal{N}^0_{\lambda} \ne \emptyset$.
	\end{rmk}
	\begin{prop}\label{N0vaziosing}
		Assume ($P_0$), ($P$), ($V_0$) and ($V_1'$). Suppose also that $\lambda \in ( 0, \lambda^*)$. Then $\mathcal{N}^0_{\lambda} = \emptyset$.
	\end{prop}
	\begin{proof}
		Assume by contradiction that there exists $(u, v) \in \mathcal{N}^0_{\lambda}$. Hence, $t_n(u, v) = 1$ which implies that $\lambda=R_n(u, v) =R_n(t_n(u, v)(u, v)) =\Lambda_n(u, v) \geq \lambda^* > \lambda$. This is a contradiction showing the desired result. 
	\end{proof}

	Now, we prove a relation between 
	$\frac{d}{dt}R_e(tu, tv)$ and $E'_{_{\lambda}}(tu, tv)(tu, tv)$. Define the function $G: \mathcal{A}\rightarrow \mathbb{R}$ given by  $$G(u,v) = \frac{1}{1 - p}P(u)  + \frac{1}{1 - q}Q(v).$$
	
	\begin{prop}\label{tR_e'G = E'sing} 
		Consider $(u, v)\in \mathcal{A}$ where $\lambda = R_e(tu,tv)$ holds true for some $t > 0$. Then
		$$\frac{d}{dt}R_e(tu, tv) = \frac{1}{t}\frac{E'_{_{\lambda}}(tu, tv)(tu, tv)}{G(tu, tv)}.$$ 
	\end{prop}
	\begin{proof} The proof follows the same ideas discussed in the proof of Proposition \ref{tR_n'G' = E''sing}. The details are omitted.  
	\end{proof}
	Now, we show whether the function $t \mapsto E(tu,tv)$ intersects the $t$ axis. In fact, we consider the following result:
	\begin{rmk}\label{signal R_e' = signal E'sing} 
		Assume $\lambda = R_e(tu,tv)$. Then $\frac{d}{dt}R_e(tu, tv) $ has the same sign as $E'_{_{\lambda}}(tu, tv)(tu, tv)$. Furthermore, $\frac{d}{dt}R_e(tu, tv) $ is zero if and only if $E'_{_{\lambda}}(tu, tv)(tu, tv)$ is zero.
	\end{rmk}
	Now, we are able to prove the following result:
	\begin{prop}
		Assume ($P_0$, ($P$), ($V_0$) and ($V_1'$). Then we obtain that
		$\Lambda_n(u, v) > \Lambda_e(u, v), \;\; (u, v) \in \mathcal{A}.$
		As a consequence, we deduce that $0 < \lambda_* <\lambda^* < + \infty$.
	\end{prop}
	\begin{proof}
		According to the Proposition \ref{diferença impsing} we infer that
		$$Q_n(t) - Q_e(t)=\frac{t}{(1 - p)(1 - q)}\left ( \frac{(1 - q)t^{1 - p} P(u) + (1 - p) t^{1 - q}Q(v)}{t^{1 - p} P(u) + t^{1 - q} Q(v)} \right )\frac{d}{dt}Q_e(t).$$
		On the other hand, we know that $Q_n(t_n(u, v)) > Q_n(t)$ holds for each $t \ne t_n(u, v)$. Therefore, we see that 
		$$\Lambda_n(u, v) - \Lambda_e(u, v) = Q_n(t_n(u, v)) - Q_e(t_e(u, v)) > Q_n(t_e(u, v)) - Q_e(t_e(u, v)) = 0.$$
		Here, was used the fact that $\frac{dQ_e(t)}{dt}\! =\! 0$ for $t\! =\! t_e(u, v)$. Under these conditions, $\Lambda_e(u, v) < \Lambda_n(u, v)$. Moreover, by using Proposition \ref{Longe de zerosing}, we infer also that $\lambda^* = \Lambda_n(\tilde{u}, \tilde{v})$ and $$\lambda^* = \inf\limits_{(u, v)\in \mathcal{A}} \Lambda_n(u,v) = \Lambda_n(\tilde u, \tilde v)> \Lambda_e(\tilde u, \tilde v) \ge \inf\limits_{(z, w) \in \mathcal{A}}\Lambda_e(z, w)  = \lambda_*.$$
		This ends the proof. 
	\end{proof}
	\begin{prop}\label{Lambdae zero homsing}
		Assume ($P_0$) and ($P$). Then the functional $\Lambda_e$ is zero homogeneous.
	\end{prop}
	\begin{proof}
		For any $s>0$, we see that $\Lambda_e(su, sv) = \sup\limits_{t>0}Q_e(tsu, tsv) = \sup\limits_{a>0} Q_e(au, av) = \Lambda_e(u, v)$. This finishes the proof.
	\end{proof}
	
	\begin{rmk}\label{ex3.18Elon analise reta} 
		Consider nonzero numbers $a_1,a_2,b_1,b_2\in\mathbb{R}^+$ where $b_1,b_2>0$. Then, we obtain that 
		$$\min\left\{\frac{a_1}{b_1}, \frac{a_2}{b_2}\right\} \le \frac{a_1 + a_2}{b_1 + b_2} \le \max\left\{\frac{a_1}{b_1}, \frac{a_2}{b_2}\right\}.$$
	\end{rmk}

	\begin{lem}\label{apendice1} Consider a function $f: \mathbb{R} \to \mathbb{R}$ given by $f(t) := \frac{At^2 - Bt^\eta}{Ct^{p} + Dt^{q}},t>0,$ where $A, B, C, D>0$. Then there exists a unique critical point of f which correspondents to a global maximum point.
	\end{lem}

	\begin{prop}\label{ukbar vkbar Limitadasing}
		Suppose ($P_0$), ($P$), ($V_0$) and ($V_1'$). Then there exists a bounded minimizing sequence for $\lambda_*$. 
	\end{prop}
	\begin{proof}
		The proof follows the same ideas employed in the proof of Proposition \ref{Limitadasing}. Consider the $(\bar{u}_k, \bar{v}_k) \in \mathcal{A}$ given by
		$\bar{u}_k=t_e(u_k, v_k)u_k$ and $\bar{v}_k=t_e(u_k, v_k)v_k$. 
		Hence, by using Proposition \ref{Lambdae zero homsing}, we obtain that $(\bar{u}_k, \bar{v}_k)$ is also a minimizing sequence for $\lambda_*$. Notice also that
		\begin{equation}\label{deriv Re em 1sing} 
			\left.\frac{d}{dt}R_e(t \bar{u}_k, t\bar{{v}_k})\right|_{t=1}=0.
		\end{equation}
		In particular, we mention that
		$$\left (A(\bar{u}_k, \bar{v}_k) - \theta B(\bar{u}_k, \bar{v}_k)\right )\left ( \frac{1}{1 - p}P(\bar{u}_k)  + \frac{1}{1 - q}Q(\bar{v}_k) \right )= 
		\left ( \frac{1}{2}A(\bar{u}_k, \bar{v}_k) - \frac{\theta}{\alpha +\beta} B(\bar{u}_k, \bar{v}_k)\right )\left ( P(\bar{u}_k) 
		+ Q(\bar{v}_k) \right ).$$
		The last identity is equivalent to
		$$\theta B(\bar{u}_k, \bar{v}_k) = \frac{\left (\frac{1}{2} - \frac{1}{1 - p} \right )P(\bar{u}_k) + \left (\frac{1}{q} - \frac{1}{2}\right )Q(\bar{v}_k)}{\left (\frac{1}{\alpha + \beta} - \frac{1}{1 - p } \right )P(\bar{u}_k)  + \left (\frac{1}{\alpha + \beta} - \frac{1}{1 - q } \right )Q(\bar{v}_k)}A(\bar{u}_k, \bar{v}_k).$$
		Notice that
		$$\frac{\frac{1}{2} - \frac{1}{1 - x}}{\frac{1}{\alpha + \beta} - \frac{1}{1 - x}} = \frac{\alpha + \beta}{2}\frac{1 + x}{\alpha + \beta - 1 + x} = \frac{\alpha + \beta}{2}f(x) \;\;\; \mbox{where}\;\;\;f(x) = \frac{1 + x}{\alpha + \beta - 1 + x},$$ 
		is an increasing function. Now, by using Remark \ref{ex3.18Elon analise reta}, we obtain the following estimate:
		\begin{equation}\label{estimativa do acoplamentosing} 
			\frac{\alpha + \beta}{2}f(p)\Vert(\bar{u}_k, \bar{v}_k) \Vert ^2 \le \theta \int  \vert \bar{u}_k \vert ^\alpha \vert \bar{v}_k \vert ^\beta dx \le \frac{\alpha + \beta}{2}f(q)\Vert(\bar{u}_k, \bar{v}_k) \Vert ^2.
		\end{equation}
		Similarly, by using \eqref{limitaçao inf do módulosing}, we infer that
		\begin{equation}\label{estimativa do seq bar using} 
			\Vert(\bar{u}_k, \bar{v}_k) \Vert \ge \left (  \frac{(\alpha + \beta)f(p)}{2\theta S^{\alpha +\beta}_{\alpha + \beta}}\right )^{\frac{1}{\alpha + \beta -2}}:=\hat{\rho}.
		\end{equation}
		As a consequence, by using estimate \eqref{estimativa do seq bar using} and \eqref{estimativa do acoplamentosing}, we deduce that
		\begin{equation}\label{estimativa do acoplamento2sing} 
			\theta \int  \vert \bar{u}_k \vert ^\alpha \vert \bar{v}_k \vert ^\beta dx \geq \left (\frac{(\alpha + \beta)^{\alpha + \beta }f(p)^{\alpha + \beta }}{2^{\alpha + \beta }\theta^2 S_{\alpha + \beta}^{2(\alpha + \beta)}}\right )^{\frac{1}{\alpha + \beta - 2}} =  \bar{\rho} > 0.
		\end{equation}
		Furthermore, we infer that $\Lambda_e (\bar{u}_k, \bar{v}_k) \le \lambda_* +\frac{1}{k}.$ Thus, we obtain that $$ \frac{\frac{1}{2}\Vert (\bar{u}_k, \bar{v}_k)\Vert^2 - \frac{\theta}{\alpha + \beta}  \int  \vert\bar{u}_k\vert^\alpha \vert\bar{v}_k\vert^\beta dx}{\frac{1}{1 - p} \int  a(x)\vert \bar{u}_k \vert^{1 - p}dx 
			+  \frac{1}{1 - q} \int  b(x)\vert \bar{v}_k \vert^{1 - q}dx} \le \lambda_* +\frac{1}{k}.$$
		The last assertion implies that
		$$||(\bar{u}_k, \bar{v}_k)||^2\leq \frac{2}{\alpha+\beta}\theta B(\bar{u}_k, \bar{v}_k)+2\left(\frac{1}{1-p}P(\bar{u}_k)+\frac{1}{1-q}Q(\bar{v}_k)\right)$$
		
		Now, by using \eqref{bounded} and \eqref{estimativa do acoplamentosing}, we deduce that 
		\begin{equation}\label{desig barra uk, vksing}
			||(\bar{u}_k, \bar{v}_k)||^2\leq f(q)||(\bar{u}_k, \bar{v}_k)||^2+2\frac{2S}{1-q}\left(||(\bar{u}_k, \bar{v}_k)||^{1-p}+||(\bar{u}_k, \bar{v}_k)||^{1-q}\right)(\lambda_*+\frac{1}{k})
		\end{equation}
		Without loss of generality we assume that $\Vert(\bar{u}_k, \bar{v}_k) \Vert \geq 1$. Hence, we infer that
		$$\Vert(\bar{u}_k,\bar{v}_k)\Vert \le \left (\frac{4 S\left ( \lambda_* +\frac{1}{k} \right )}{(1 - f(q))(1 - q)}\right )^{\frac{1}{p+1}}.$$
		This ends the proof.
	\end{proof}
	\begin{prop}\label{Lambdae FSIsing}
		Assume ($P_0$),($P$),($V_0$) and ($V_1'$). Then the weak limit $(\bar{u}, \bar{v})$ of sequence provided by Proposition \ref{ukbar vkbar Limitadasing} satisfies $$\Lambda_e(\bar{u}, \bar{v}) = \liminf\limits_{k \to \infty}\Lambda_e(\bar{u}_k, \bar{v}_k) = \lambda_*$$
	\end{prop}
	\begin{proof} 
		Firstly, by using Proposition \ref{ukbar vkbar Limitadasing}, we suppose that $(\bar{u}_k, \bar{v}_k)  \rightharpoonup  (\bar{u}, \bar{v})$ in $X$. Hence, using the same arguments discussed in the proof of the Proposition \ref{Longe de zerosing}, we obtain that
		$$\lambda_* \le \Lambda_e(\bar{u}, \bar{v}) \le \liminf\limits_{k \to \infty}\Lambda_e(\bar{u}_k, \bar{v}_k) = \lambda_*.$$ 
		This ends the proof. 
	\end{proof}
	
	\begin{prop}\label{lambda*logzerosing}
		Assume ($P_0$), ($P$), ($V_0$) and ($V_1'$). Then there exists $\bar{C}_{\delta}$ such that $\lambda_*\ge \bar{C}_{\delta} > 0$. 
	\end{prop}
	\begin{proof}
		Firstly, by using \eqref{estimativa do acoplamentosing}, we obtain in the following estimate
		\begin{eqnarray*}
			\Lambda_e({u}_k, 
			{v}_k) & = & R_e(t_e(u_k,v_k)(u_k, v_k)) = R_e(\bar{u}_k, \bar{v}_k)
			=\frac{\frac{1}{2}A(\bar{u}_k,\bar{v}_k) - \frac{\theta}{\alpha + \beta}  B(\bar{u}_k,\bar{v}_k)}{\frac{1}{1 - p}P(\bar{u}_k) + \frac{1}{1 - q}Q(\bar{v}_k) }. 
		\end{eqnarray*}
		Now, by using the H\"older inequality, \eqref{estimativa do acoplamentosing} together with \eqref{bounded}, we deduce that
		$$\Lambda_e({u}_k, 
		{v}_k) \geq \frac{\frac{1}{2}\Vert(\bar{u}_k,\bar{v}_k)\Vert ^2 -\frac{f(q)}{2}  \Vert(\bar{u}_k, \bar{v}_k)\Vert^2}{\frac{S}{1 - q}\left( \Vert(\bar{u}_k,\bar{v}_k)\Vert ^{1-p}+\Vert(\bar{u}_k,\bar{v}_k)\Vert ^{1-q}\right)}= \frac{(1-q)(1-f(q))}{2S}\frac{\Vert(\bar{u}_k,\bar{v}_k)\Vert ^2}{\Vert(\bar{u}_k,\bar{v}_k)\Vert ^{1-p}+\Vert(\bar{u}_k,\bar{v}_k)\Vert ^{1-q}}.$$
		Furthermore, by applying Remark \ref{ex3.18Elon analise reta} and   \eqref{estimativa do seq bar using}, we rewrite the last estimate as follows:
		$$\Lambda_e({u}_k, 
		{v}_k) \geq \frac{(1-q)(1-f(q))}{4S}\min\{\Vert(\bar{u}_k,\bar{v}_k)\Vert ^{1+p},\Vert(\bar{u}_k,\bar{v}_k)\Vert ^{1+p}\}\geq \frac{(1-q)(1-f(q))}{4S}\min\{\hat{\rho}^{1+p},\hat{\rho}^{1+p}\}:=\bar{C}_{\delta}.$$
		This ends the proof. 
	\end{proof}
	
	\begin{prop}\label{N-longe zerosing}
		Suppose ($P_0$), ($P$), ($V_0$) and ($V_1'$). Assume also that $\lambda \in ( 0, \lambda^*)$. Then there exists $C > 0$ such that $\Vert (u, v) \Vert \ge C:= C(p, \theta, \alpha , \beta) > 0$ holds for any $(u, v) \in \mathcal{N}^-_{\lambda}$.
	\end{prop}
	\begin{proof}
		Let us consider $(u, v) \in \mathcal{N}^-_{\lambda}$. Thus, we obtain that
		\begin{equation}\label{equação de Neharising} 
			\lambda P(u) + \lambda Q(v)= \Vert (u, v) \Vert ^2  - \theta B(u, v).
		\end{equation}
		Moreover, we know that
		\begin{eqnarray*}\label{equação de Nehari -}
			2\Vert (u, v) \Vert ^2 - \theta ( \alpha +\beta) B(u, v) & < & \lambda(1 - p) P(u)  + \lambda(1 - q) Q(v) < (1 - p)\left[\Vert (u, v) \Vert ^2  - \theta B(u, v)\right].
		\end{eqnarray*}
		Here was used the the estimate \eqref{equação de Neharising}. Hence, we obtain that
		\begin{equation}\label{limitação u vsing} 
			(1 + p)\Vert (u, v) \Vert ^2  <  \theta (\alpha + \beta - 1 + p)\int  |u|^\alpha |v|^\beta dx.
		\end{equation}
		Since $\int  |u|^\alpha |v|^\beta dx \le S^{\alpha + \beta}_{\alpha + \beta}\Vert(u, v)\Vert^{\alpha + \beta}$ we mention also that
		\begin{eqnarray*}\label{2ª equação de Nehari -}
			\Vert(u, v)\Vert & \ge &  \left (\frac{1 + p}{\theta [\alpha + \beta - 1 + p]S^{\alpha + \beta}_{\alpha + \beta}}\right )^{\frac{1}{\alpha + \beta - 2}}\;:=\;C, \;\;\; (u, v) \in \mathcal{N}_{\lambda}^-.
		\end{eqnarray*}
		This ends the proof. 
	\end{proof}
	\begin{prop}
		Assume ($P_0$), ($P$), ($V_0$) and ($V_1'$). Let us consider $(u_k, v_k)\subset \mathcal{N}^-_{\lambda}$ a minimizing sequence for $C_{\mathcal{N}_{\lambda}^-}$ where $(u_k, v_k) \rightharpoonup (u, v)$ in $X$. Then there exists $\delta_C > 0$ such that $\int  |u|^\alpha|v|^\beta dx \geq \delta_C > 0$.
	\end{prop}
	\begin{proof}
		Initially, by using \eqref{limitação u vsing} and Proposition \ref{N-longe zerosing}, we infer that
		$$\int  \vert{u}_k\vert^\alpha\vert{v}_k\vert^\beta dx \ge \frac{(1 + p)}{\theta(\alpha + \beta - 1 + p)} C^2=:\delta_C > 0.$$
		Therefore, the desired result follows by using the Dominated Convergence Theorem. This ends the proof.
	\end{proof}
	\begin{prop}\label{Neharilambda*}
		Suppose ($P_0$), ($P$), ($V_0$) and ($V_1'$). Let us consider $\lambda = \lambda^*$ and $(u, v) \in \mathcal{N}_{\lambda^*}^0$. Then, for any $ (\psi_1, \psi_2) \in X$, we obtain that
		\begin{eqnarray}\label{eqauxiliarsing1}
			2\left <(u, v), (\psi_1, \psi_2)\right> - \theta \int  \alpha\vert u\vert^{\alpha - 2} u \psi_1 \vert v \vert^\beta  - \beta  \vert u\vert^{\alpha} \vert v \vert^{\beta - 2} v \psi_2dx 
			- \lambda^* \int (1 - p) a(x) u^{-p}\psi_1 + (1 - q) b(x) v^{-q}\psi_2dx = 0.&&
		\end{eqnarray}
	\end{prop}
	\begin{proof}
		Firstly, taking into account Proposition \ref{Longe de zerosing}, we obtain the characterization of $\mathcal{N}_{\lambda^*}^0$. In what follows, we shall split the proof of \eqref{eqauxiliarsing1} into three steps.  Let $(u, v) \in \mathcal{A}$ be fixed such that $\lambda^*=\Lambda_n(u, v)$. Hence, we write $\Lambda_n(u, v)$ as $$\Lambda_n(u, v) = f(z, w) g(z, w),\;\;\;\;where\;\;\;\;f(z, w) := \frac{1}{P(z) + Q(w)}\;\;\; \mbox{and} \;\;\; g(z, w) := A(z, w) + \theta B(z, w).$$ Under these conditions,  we write $z := t_n(u, v)u$ and $w := t_n(u, v) v $.

		\textbf{Step 1}. Here we shall prove that $\left <f'(z, w), (\psi_1, \psi_2)\right>$ exists for each $(z, w) \in \mathcal{N}^0_{\lambda^*}$ where $z, w \ge 0$ and for any $(\psi_1, \psi_2) \in X_+$. It follows from continuity of $B$ that $\int  \vert z + t\psi_1\vert^{\alpha} \vert w + t\psi_2\vert^{\beta} dx > 0$ is satisfied for each small $t>0$. In other words, we obtain that $(z + t\psi_1, w + t\psi_2) \in \mathcal{A}$. Hence, by using the fact that $(z,w) \mapsto A(z, w)$ and $(z,w) \mapsto B(z, w)$ are in $C^1$ class, we obtain that
		$$\left <g'(z, w), (\psi_1, \psi_2) \right > = \lim\limits_{t \to \infty} \frac{g(z+ t\psi_1, w + t\psi_2) - g(z, w)}{t}.$$
		Recall also that $(u, v)$ is a minimizer for the functional $\Lambda_n$. As a consequence, we obtain that 
		$$\Lambda_n(u, v) = f(z, w)g(z, w) = \lambda^* = \inf\limits_{(z, w) \in \mathcal{A}}\Lambda_n(z, w).$$
		Furthermore, we mention that
		$\Lambda_n(u + t\psi_1, v + t\psi_2) - \Lambda_n(u, v) \ge 0$
		is satisfied for each $t \ge 0$.
		
		From now on, we shall use the following notation $\bar{z} = t_n(u + t\psi_1, v + t\psi_2)(u + t\psi_1)$ and $\bar{w} = t_n(u + t\psi_1, v + t\psi_2)(v + t\psi_2)$. Now, we infer also that
		$$0 \le f(\bar{z}, \bar{w})g(\bar{z}, \bar{w}) - f(z, w)g(z, w), \,\, 0 \le f(\bar{z}, \bar{w})g(\bar{z}, \bar{w}) - f(z, w)g(z, w) + f(\bar{z}, \bar{w})g(z, w) - f(\bar{z}, \bar{w})g(z, w).$$
		As a consequence, we obtain that
		\begin{equation}\label{desig sing imp*}
			f(\bar{z}, \bar{w})(g(\bar{z}, \bar{w}) - g(z, w)) \ge - g(z, w)(f(\bar{z}, \bar{w}) - f(z, w)). 
		\end{equation}
		At this stage, we define the following functionals:
		$$L(t) := f(t_n(z+t\psi_1, w + t\psi_2)(z + t\psi_1, w + t\psi_2) )= \frac{1}{P(\bar{z}) + Q(\bar{w})}.$$
		Hence, by using the Mean Value Theorem, there exists $\theta \in (0, t)$ 
		such that
		$$\frac{L(t) - L(0)}{t} = L'(\theta) = \frac{f(\bar{z}, \bar{w}) - f(z, w)}{t}, \lim\limits_{t \to 0}\frac{L(t) - L(0)}{t} = \lim\limits_{t \to 0} L'(\theta) = L'(0) = \lim\limits_{t \to 0}\frac{f(\bar{z}, \bar{w}) - f(z, w)}{t}$$
		and
		$$L'(\theta) = -\frac{P'(\bar{z})\psi_1 + Q'(\bar{w})\psi_2}{(P(\bar{z}) +  Q(\bar{w}))^2}.$$
		Notice also that
		\begin{equation}\label{L'(0)}
			L'(0)= - \frac{(1 - p) \int a(x)\vert z \vert^{-1 - p}z \psi_1 dx + (1 - q) \int b(x)\vert w \vert^{-1 - q}w \psi_2 dx}{(P(z) +  Q(w))^2}.  
		\end{equation}
		Hence, by using \eqref{desig sing imp*}, we infer that
		\begin{equation}\label{desi}
			f(z, w)\lim\limits_{t \to 0}\frac{g(\bar{z}, \bar{w}) - g(z, w)}{t} \ge - g(z, w)\liminf\limits_{t \to 0}\frac{f(\bar{z}, \bar{w}) - f(z, w)}{t}. 
		\end{equation}
		As a consequence, we deduce that
		\begin{equation}\label{desi}
			\infty > f(z, w)\left < g'(z, w), (\psi_1, \psi_2) \right > \ge - g(z, w)\liminf\limits_{t \to 0}\frac{f(\bar{z}, \bar{w}) - f(z, w)}{t}. 
		\end{equation}
		Now, we mention that $(\bar{z}, \bar{w}) \to (z, w)$ as $t \to 0^+$. Under these conditions, by using \eqref{L'(0)} and the last estimate, we obtain that
		$$\infty > g(z, w)\liminf\limits_{t\to 0^+} \frac{(1 - p) \int a(x)\vert z \vert^{-1 - p}z \psi_1 dx + (1 - q) \int b(x)\vert w \vert^{-1 - q}w \psi_2 dx}{(P(z) +  Q(w))^2}.$$
		In view of Fatou's lemma we see that
		$$ \infty > g(z, w)\frac{ \int a(x) \liminf\limits_{t \to 0^+} \frac{\vert \bar{z}\vert^{1 - p} - \vert z\vert^{1 - p}}{t}dx +  \int b(x) \liminf\limits_{t \to 0^+} \frac{\vert \bar{w}\vert^{1 - q} - \vert w\vert^{1 - q}}{t}dx}{(P(z) + Q(w))^2}.$$
		Similarly, we define the following functionals:
		$$L_1(z) = z^{1 - p},\;\;\;\;\; L_1'(z)\psi_1 = (1 - p) z^{-p}\psi_1, \;\;\;\;\; G_1(x) = z^{-p}(x),$$ 
		$$L_2(w) = w^{1 - q}, \;\;\;\;\;L_2'(w)\psi_2 = (1 - q) w^{-q}\psi_2,\;\;\;\; G_2(x) = w^{-q}(x).$$
		Here, we emphasize that $(u, v)\in \mathcal{N}^0_{\lambda^*}$ and $z, w \ge 0$, a.e., in $\mathbb{R}^N$. Furthermore, we mention that
		\begin{equation}\nonumber
			G_1(x) =
			\left \{
			\begin{array}{cc}
				z^{-p}(x), & \mbox{if} \;\;\; z(x) \ne 0; \\
				\infty, & \mbox{if} \;\;\; z(x) = 0.  \\
			\end{array}
			\right.
			\;\;\;\;\;\;and\;\;\;\;\;
			G_2(x) =
			\left \{
			\begin{array}{cc}
				w^{-q}(x), & \mbox{if} \;\;\; w(x) \ne 0; \\
				\infty, & \mbox{if} \;\;\; w(x) = 0. \\
			\end{array}
			\right.
		\end{equation}
		Hence, by choosing $\psi_1, \psi_2 \in X_+$, we deduce that $G_1(x) = z^{-p}(x)$ and $G_2(x) = w^{-q}(x)$ holds for each $ x \in \mathbb{R}^N$. Therefore,  $z > 0$ and $w > 0$ a.e. in $\mathbb{R}^N$. Under these conditions, for each $(\psi_1, \psi_2) \in X_+$, we conclude that
		$$0 < \int a(x)z^{-p}\psi_1 dx < \infty \;\;\;\;\mbox{and}\;\;\;\;0 < \int b(x)w^{-q}\psi_2 dx < \infty.$$
		As a consequence, we obtain that
		$$\left < P'(z), \psi_1 \right > = (1 - p)\int a(x)\vert z \vert^{-p}\psi_1dx \;\;\;\; \mbox{and}\;\;\;\;\left < Q'(w), \psi_2 \right > = (1 - q)\int b(x)\vert w \vert^{-q}\psi_2dx.$$
		Recall also that $f(z, w) = \frac{1}{P(z) + Q(w)}.$
		Hence, we infer that
		\begin{equation}\label{f'sing}
			\left < f'(z, w), (\psi_1, \psi_2) \right > = -\frac{\left < P'(z), \psi_1 \right > + \left < Q'(w), \psi_2 \right >}{(P(z) + Q(w))^2}, \, \, (\psi_1, \psi_2) \in X_+.   
		\end{equation}
		This finishes the proof of Step 1.
		
		\textbf{Step 2.} Here we shall prove that the expression of the first member of \eqref{eqauxiliarsing1} is nonnegative for each $(\psi_1, \psi_2) \in X_+$. 
		In order to do that, we shall prove that
		\begin{eqnarray}\label{ineqsingauxi}
			2\left <(z, w), (\psi_1, \psi_2)\right> - \theta\alpha \int  \vert z\vert^{\alpha - 2} z \psi_1 \vert w \vert^\beta dx - \theta\beta \int  \vert z\vert^{\alpha - 2} z \vert w \vert^\beta \psi_2dx +  && \nonumber\\
			- \lambda^* \left[(1 - p)\int a(x) z^{-p}\psi_1dx + (1 - q)\int b(x) w^{-q}\psi_2dx\right] \ge 0,  &&
		\end{eqnarray}
		holds for each $(\psi_1, \psi_2) \in X_+$. Recall that $(z,w) \mapsto g(z, w)$ is in $C^1$ class. Hence, we deduce that
		$$\left <g'(z, w), (\psi_1, \psi_2)\right> = A'(z, w)(\psi_1, \psi_2) - \theta B'(z, w)(\psi_1, \psi_2), \,\, (\psi_1, \psi_2) \in X.$$
		Furthermore, by using \eqref{desi} and \eqref{f'sing}, we infer that
		$$g(z, w) \frac{L(z, w)}{(P(z) + Q(w))^2} \le f(z, w)\left <g'(z, w)(\psi_1, \psi_2)\right>.$$
		Here, we observe that $L(z, w)= \int a(x)(1 - p)z^{-p}\psi_1 dx + \int b(x)(1 - q)w^{-q}\psi_2 dx$. 
		Similarly, we mention that 
		$$L(z, w) \le \frac{f(z, w)\left <g'(z, w)(\psi_1, \psi_2)\right>}{g(z, w) }(P(z) + Q(w))^2,\, \int a(x)(1 - p)z^{-p}\psi_1 dx + \int b(x)(1 - q)w^{-q}\psi_2 dx < \infty.$$
		As a consequence, all the terms in the weak formulation given in \eqref{eqauxiliarsing1} are finite. Recall also that 
		$$\int a(x)(1 - p)z^{-p}\psi_1 dx < \infty \,\, \mbox{and} \, \,\int b(x)(1 - q)w^{-q}\psi_2 dx < \infty.$$
		In view of the last assertion we obtain that  $z > 0$ and $w > 0$  a.e. in $\mathbb{R}^N$. As a consequence, we obtain that 
		$$\left< f'(z, w), (\psi_1, \psi_2) \right> = - \frac{(P'(z)\psi_1 + Q'(w)\psi_2)}{(P(z) + Q(w))^2} = - (f(z, w))^2(P'(z)\psi_1 + Q'(w)\psi_2).$$
		Furthermore, by using \eqref{desi}, we see that
		$$f(z, w)\left< g'(z, w), (\psi_1, \psi_2) \right> + g(z, w)\left< f'(z, w), (\psi_1, \psi_2) \right> \ge 0, \;\; (\psi_1, \psi_2)\in X^+.$$
		Let us apply the last assertion which proves that $\left< g'(z, w), (\psi_1, \psi_2) \right>$ is well defined. In fact, we obtain that 
		\begin{eqnarray*}
			\left< g'(z, w), (\psi_1, \psi_2) \right> & = & 2\left<(z, w), (\psi_1, \psi_2) \right> - \theta \alpha \int \vert z\vert^{\alpha - 2}z \psi_1 \vert w\vert^{\beta}dx- \theta \beta\int \vert z\vert^{\alpha} \vert w\vert^{\beta-2}w\psi_2dx.  
		\end{eqnarray*}
		Then, by using the function $f(z, w)$, we deduce that
		$$f(z, w)\left(2\left<(z, w), (\psi_1, \psi_2) \right> - \theta \alpha \int \vert z\vert^{\alpha - 2}z \psi_1 \vert w\vert^{\beta}dx - \theta \beta\int \vert z\vert^{\alpha} \vert w\vert^{\beta-2}w\psi_2dx\right.  + $$
		$$\left. - f(z, w)g(z, w)\left[(1 - p)\int a(x)\vert z\vert^{-p}\psi_1dx + (1 - q)\int b(x)\vert w\vert^{-q}\psi_2dx \right]\right) \ge 0.$$
		Moreover, we have that $\lambda^* = f(z, w)g(z, w)$ and $f(z, w) > 0$. Hence, by using the last estimate, we obtain that
		\begin{eqnarray*}
			2\left <(z, w), (\psi_1, \psi_2)\right> - \theta\alpha \int  \vert z\vert^{\alpha - 2} z \psi_1 \vert w \vert^\beta dx - \theta\beta \int  \vert z\vert^{\alpha - 2} z \vert w \vert^\beta \psi_2dx +  && \\
			- \lambda^* \left[(1 - p)\int a(x) z^{-p}\psi_1dx + (1 - q)\int b(x) w^{-q}\psi_2dx\right] \ge 0,
		\end{eqnarray*}
		holds for each $(\psi_1, \psi_2) \in X_+$.
		
		\textbf{Step 3.}  In this step we follow the strategy of \cite{Yijing2001,silvasing2018} proving that \eqref{eqauxiliarsing1} is verified for each  $(\psi_1, \psi_2) \in X$. Recall that \eqref{ineqsingauxi} is verified for non-negative functions. Hence, by taking any functions $(\psi_1, \psi_2) \in X$, we obtain that $\phi_1 = (z + \varepsilon\psi_1)^+$ and $\phi_2 = (w + \varepsilon\psi_2)^+$ satisfies $(\phi_1, \phi_2) \in X_+$. As a consequence, by using Step 2 and \eqref{ineqsingauxi}, we see that
		\begin{eqnarray}\label{eqauximpsing}
			0 \le 2\left <(z, w), (\phi_1, \phi_2)\right> - \theta\alpha \int  \vert z\vert^{\alpha - 2} z \phi_1 \vert w \vert^\beta dx - \theta\beta \int  \vert z\vert^{\alpha} \vert w \vert^{\beta - 2}w \phi_2dx + && \nonumber\\
			- \lambda^* \left[(1 - p)\int a(x) z^{-p}\phi_1dx + (1 - q)\int b(x) w^{-q}\phi_2dx\right].
		\end{eqnarray}
		Let us analyze each term given just above separately. First, we consider the following decomposition:
		\begin{eqnarray}\label{prodintsing}
			\left<(z,w),(\phi_1, \phi_2)\right> & = &\iint  \frac{\left[ z(x)-z(y)\right] \left[ \phi_1(x) - \phi_1(y)\right]}{\vert x-y\vert^{N+2s}} + \frac{\left[ w(x)-w(y) \right] \left[ \phi_2(x) - \phi_2(y)\right]}{\vert x-y \vert^{N+2s}} dydx    + \int  V_1 z \phi_1 +V_2  w \phi_2 dx\nonumber\\ 
			& & = I + I' + II + II'.\nonumber
		\end{eqnarray}
		Let us analyze the estimates for $I$ and $II$. Here, we observe that
		$$ I = \iint \frac{\left[ z(x)-z(y)\right] \left[ \phi_1(x) - \phi_1(y)\right]}{\vert x-y\vert^{N+2s}} dydx.$$
		Recall also that $\phi_1= (z + \varepsilon\psi_1)^+$. Hence, we obtain that
		$$ I = \iint \frac{\left[ z(x)-z(y)\right] \left[ (z + \varepsilon\psi_1)^+(x) - (z + \varepsilon\psi_1)^+(y)\right]}{\vert x-y\vert^{N+2s}} dydx.$$
		Now, by using the that $u = u^+ + u^-$, we write
		\begin{eqnarray*}
			I\! & = &\!\iint \frac{\left[ z(x)-z(y)\right] \left[ (z + \varepsilon\psi_1)(x) - (z + \varepsilon\psi_1)(y)\right]}{\vert x-y\vert^{N+2s}} dydx \!-\! \iint  \frac{\left[ z(x)-z(y)\right] \left[ (z + \varepsilon\psi_1)^-(x) - (z + \varepsilon\psi_1)^-(y)\right]}{\vert x-y\vert^{N+2s}} dydx\nonumber\\
			& = & S_1 + S_2.
		\end{eqnarray*}
		As a consequence, we mention that
		\begin{eqnarray}\label{S1}
			S_1 & = & \iint \frac{\left[ z(x)-z(y)\right] \left[ z(x) + \varepsilon\psi_1(x) - z(y) - \varepsilon\psi_1(y)\right]}{\vert x-y\vert^{N+2s}} dydx \nonumber\\
			& = &\iint \frac{\left[ z(x)-z(y)\right] \left[ z(x) - z(y)\right]}{\vert x-y\vert^{N+2s}} dydx + \varepsilon \iint  \frac{\left[ z(x)-z(y)\right] \left[\psi_1(x) - \psi_1(y)\right]}{\vert x-y\vert^{N+2s}} dydx, \nonumber
		\end{eqnarray}
		and 
		$$S_2 = - \iint  \frac{\left[ z(x)-z(y)\right] \left[ (z + \varepsilon\psi_1)^-(x) - (z + \varepsilon\psi_1)^-(y)\right]}{\vert x-y\vert^{N+2s}} dydx.$$
		Furthermore, by using the same strategy developed in \cite{Yijing2001}, we also write 
		$$A_{x}^+ = \{x \in \mathbb{R}^N; (z + \varepsilon\psi_1)(x) \ge 0 \}\;\;\; \mbox{and}\;\;\; A_{x}^- = \{x \in \mathbb{R}^N; (z + \varepsilon\psi_1)(x) < 0 \}.$$
		Notice also that $\mathbb{R}^N = A^+_x \cup A^-_x$ and $A^+_x \cap A^-_x = \emptyset$. Hence, we obtain that
		\begin{eqnarray*}
			S_2 & = & - \int\limits_{A_x^-} \int  \frac{\left[ z(x)-z(y)\right] \left[ (z + \varepsilon\psi_1)(x) - (z + \varepsilon\psi_1)^-(y)\right]}{\vert x-y\vert^{N+2s}} dydx - \int\limits_{A_x^+} \int  \frac{\left[ z(x)-z(y)\right] \left[ 0 - (z + \varepsilon\psi_1)^-(y)\right]}{\vert x-y\vert^{N+2s}} dydx. \nonumber
		\end{eqnarray*}
		Under these conditions, we split the domain of integration of the second integral into $A_y^+$ and $A_y^-$. Namely, we consider
		\begin{eqnarray*}
			S_2 & = & - \int\limits_{A_x^-} \int\limits_{A_y^-} \frac{\left[ z(x)-z(y)\right] \left[ (z +\varepsilon\psi_1)(x) - (z + \varepsilon\psi_1)(y)\right]}{\vert x-y\vert^{N+2s}} dydx - \int\limits_{A_x^-} \int\limits_{A_y^+} \frac{\left[ z(x)-z(y)\right] \left[ (z + \varepsilon\psi_1)(x) - 0\right]}{\vert x-y\vert^{N+2s}} dydx\nonumber\\
			&& - \int\limits_{A_x^+} \int\limits_{A_y^-} \frac{\left[ z(x)-z(y)\right] \left[ 0 - (z + \varepsilon\psi_1)(y)\right]}{\vert x-y\vert^{N+2s}} dydx \nonumber \\
			& = & L_1 + L_2 + L_3 \nonumber.
		\end{eqnarray*}
		Now, we mention that
		\begin{eqnarray*}
			L_1 
			&=& - \int\limits_{A_x^-} \int\limits_{A_y^-} \frac{\left[ z(x)-z(y)\right] \left[ z(x) - z(y)\right]}{\vert x-y\vert^{N+2s}} dydx - \varepsilon\int\limits_{A_x^-} \int\limits_{A_y^-} \frac{\left[ z(x)-z(y)\right] \left[ \psi_1(x) - \psi_1(y)\right]}{\vert x-y\vert^{N+2s}} dydx. \nonumber 
		\end{eqnarray*}
		As a consequence, we obtain that
		\begin{equation}\label{L1sing}
			L_1 \le - \varepsilon\int\limits_{A_x^-} \int\limits_{A_y^-} \frac{\left[ z(x)-z(y)\right] \left[ \psi_1(x) - \psi_1(y)\right]}{\vert x-y\vert^{N+2s}} dydx.  \end{equation}
		Recall also that $L_2=L_3$ and 
		$$L_2 = - \int\limits_{A_x^-} \int\limits_{A_y^+} \frac{\left[ z(x)-z(y)\right] \left[ z(x) + \varepsilon\psi_1(x)\right]}{\vert x-y\vert^{N+2s}} dydx. $$ 
		Thus, we see that
		\begin{eqnarray*}
			L_2 &=& - \int\limits_{A_x^-} \int\limits_{A_y^+} \frac{\left[ z(x)-z(y)\right]^+ \left[ z(x) + \varepsilon\psi_1(x)\right]}{\vert x-y\vert^{N+2s}} dydx -\int\limits_{A_x^-} \int\limits_{A_y^+} \frac{\left[ z(x)-z(y)\right]^- \left[ z(x) + \varepsilon\psi_1(x)\right]}{\vert x-y\vert^{N+2s}} dydx. \nonumber
		\end{eqnarray*}
		It is important to stress that $[z(x) - z(y)]^+$ needs to be considered only in the set $z(x) \ge z(y)$. However, $z(y) \ge -\varepsilon\psi_1(y)$ in $A_y^+$ proving that $ z(x) \ge - \varepsilon\psi_1(x)$ in $A_x^+$. 
		Furthermore, we infer that 
		$-\left[ z(x)-z(y)\right]^- \left[ z(x) + \varepsilon\psi_1(x)\right] \le 0, \;\;\; \mbox{in}\;\;\; A_x^- \times A_y^+$. The last assertion implies that 
		\begin{eqnarray*}
			L_2 &\le& - \varepsilon\int\limits_{A_x^-} \int\limits_{A_y^+} \frac{\left[ z(x)-z(y)\right]^+ \left[ -\psi_1(y) + \psi_1(x)\right]}{\vert x-y\vert^{N+2s}} dydx.
		\end{eqnarray*}
		Under these conditions, we deduce that
		\begin{equation}\label{L2L3sing}
			L_2 + L_3 \le - 2\varepsilon \int\limits_{A_x^-} \int\limits_{A_y^+} \frac{\left[ z(x)-z(y)\right]^+ \left[ -\psi_1(y) + \psi_1(x)\right]}{\vert x-y\vert^{N+2s}} dydx.
		\end{equation}
		Now, putting estimates \eqref{S1}, \eqref{L1sing}, \eqref{L2L3sing} together, we obtain that  
		\begin{eqnarray}\label{Ising}
			I &=& S_1 + S_2\nonumber \\
			& = & \iint \frac{\left[ z(x)-z(y)\right] \left[ z(x) - z(y)\right]}{\vert x-y\vert^{N+2s}} dydx  +\; \varepsilon \iint \frac{\left[ z(x)-z(y)\right] \left[\psi_1(x) - \psi_1(y)\right]}{\vert x-y\vert^{N+2s}}dydx + L_1 + L_2 + L_3.
			\nonumber
		\end{eqnarray} 
		Moreover, we obtain that
		\begin{eqnarray}\label{1singus}  
			I & \le &\iint \frac{\left[ z(x)-z(y)\right] \left[ z(x) - z(y)\right]}{\vert x-y\vert^{N+2s}} dydx + \varepsilon\iint \frac{\left[ z(x)-z(y)\right] \left[\psi_1(x) - \psi_1(y)\right]}{\vert x-y\vert^{N+2s}}dydx \nonumber\\
			&& - \varepsilon\int\limits_{A_x^-} \int\limits_{A_y^-} \frac{\left[ z(x)-z(y)\right] \left[ \psi_1(x) - \psi_1(y)\right]}{\vert x-y\vert^{N+2s}} dydx - 2\varepsilon \int\limits_{A_x^-} \int\limits_{A_y^+} \frac{\left[ z(x)-z(y)\right]^+ \left[ -\psi_1(y) + \psi_1(x)\right]}{\vert x-y\vert^{N+2s}} dydx.\nonumber    
		\end{eqnarray}
		Similarly, considering the the estimates for $I'$, we write $B_x^+ = \{x \in \mathbb{R}^N / (w + \varepsilon \psi_2)(x) \ge 0\}$ and 
		$B_x^- = \{x \in \mathbb{R}^N / (w + \varepsilon\psi_2)(x) \le 0\}$. Hence, using $z$ instead of $w$, $\phi_1$ for $\phi_2$, $\psi_1$ for $\psi_2$, $A_x^+$ for $B_x^+$ and $A_x^-$ for $B_x^-$, respectively, we deduce that
		\begin{eqnarray}\label{I'sing}
			I' & \le &\iint \frac{\left[ w(x)-w(y)\right] \left[ w(x) - w(y)\right]}{\vert x-y\vert^{N+2s}} dydx + \varepsilon \iint \frac{\left[ w(x)-w(y)\right] \left[\psi_2(x) - \psi_2(y)\right]}{\vert x-y\vert^{N+2s}}dydx \nonumber\\
			&&- \varepsilon\int\limits_{B_x^-} \int\limits_{B_y^-} \frac{\left[ w(x)-w(y)\right] \left[ \psi_2(x) - \psi_2(y)\right]}{\vert x-y\vert^{N+2s}} dydx - 2\varepsilon \int\limits_{B_x^-} \int\limits_{B_y^+} \frac{\left[ w(x)-w(y)\right]^+ \left[ -\psi_2(y) + \psi_2(x)\right]}{\vert x-y\vert^{N+2s}} dydx.\nonumber
		\end{eqnarray}
		Let us estimate the integral given in $II$ by using \eqref{prodintsing}. It is not hard to see that
		\begin{eqnarray}\label{IIestsing}
			II &=& \int  V_1 z \phi_1dx = \int  V_1 z (z + \varepsilon\psi_1)^+dx = \int  V_1 z [(z + \varepsilon\psi_1) - (z + \varepsilon\psi_1)^-]dx\nonumber \\ 
			&=& \int  V_1 z^2dx + \varepsilon\int  V_1z\psi_1dx - \int\limits_{A_x^-}V_1z^2dx - \varepsilon\int\limits_{A_x^-}V_1z\psi_1dx\nonumber\\&\leq&\int  V_1 z^2dx + \varepsilon\int  V_1z\psi_1dx - \varepsilon\int\limits_{A_x^-}V_1z\psi_1dx. 
		\end{eqnarray}
		Similarly, we are able to prove that
		\begin{equation}\label{II'estsing}
			II' \le \int  V_2 w^2dx + \varepsilon\int  V_2w\psi_2dx - \varepsilon\int\limits_{B_x^-}V_2w\psi_2dx.
		\end{equation}
		Let us estimate the other terms given in \eqref{eqauximpsing}. Here is important to stress that
		\begin{eqnarray*}
			-\theta\alpha \int  \vert z\vert^{\alpha - 2} z \phi_1 \vert w \vert^\beta dx & = & -\theta\alpha \int  \vert z\vert^{\alpha - 2} z (z + \varepsilon \psi_1)^+ \vert w \vert^\beta dx \nonumber \\
			& = & -\theta\alpha \left[\int  \vert z\vert^{\alpha - 2} z (z + \varepsilon \psi_1)\vert w \vert^\beta dx  - \int  \vert z\vert^{\alpha - 2} z (z + \varepsilon \psi_1)^- \vert w \vert^\beta dx \right]\nonumber\\
			& = & -\theta\alpha B(z, w) - \theta\alpha \varepsilon\int  \vert z\vert^{\alpha - 2} z\psi_1\vert w \vert^\beta dx - \theta\alpha \int\limits_{A_x^-} \vert z\vert^{\alpha} \vert w \vert^\beta dx - \varepsilon\theta\alpha \int\limits_{A_x^-} \vert z\vert^{\alpha - 2} \psi_1\vert w \vert^\beta dx.
		\end{eqnarray*}
		Now, by dropping the term $- \theta\alpha \int\limits_{A_x^-} \vert z\vert^{\alpha} \vert w \vert^\beta dx < 0$, we ensure that 
		\begin{eqnarray}\label{termoteta1}
			-\theta\alpha \int  \vert z\vert^{\alpha - 2} z \phi_1 \vert w \vert^\beta dx
			& \le & -\theta\alpha \int  \vert z\vert^{\alpha} \vert w \vert^\beta dx - \theta\alpha \varepsilon\int  \vert z\vert^{\alpha - 2} z\psi_1\vert w \vert^\beta dx - \varepsilon\theta\alpha \int\limits_{A_x^-} \vert z\vert^{\alpha - 2} \psi_1\vert w \vert^\beta dx \nonumber
		\end{eqnarray}
		Similarly, we are able to prove that
		\begin{eqnarray}\label{acopsingbeta}
			- \theta\beta \int  \vert z\vert^{\alpha} \vert w \vert^{\beta - 2} w \phi_2dx 
			& \le & -\theta\beta \int  \vert z\vert^{\alpha} \vert w \vert^\beta dx - \theta\beta \varepsilon\int  \vert z\vert^{\alpha} \vert w \vert^{\beta - 2}w \psi_2 dx - \varepsilon\theta\beta \int\limits_{B_x^-} \vert z\vert^{\alpha} \vert w \vert^{\beta - 2} \psi_2 dx.\nonumber
		\end{eqnarray}
		On the other hand, we see that
		\begin{eqnarray}\label{termo acopsing}
			- \int a(x) z^{-p}\phi_1dx& = &- \int a(x) z^{-p}(z + \varepsilon \psi_1)^+dx = - \int a(x) z^{-p}(z + \varepsilon \psi_1)dx - \int a(x) z^{-p}(z + \varepsilon \psi_1)^-dx. \nonumber
		\end{eqnarray}
		Hence, we mention that
		$$- \int a(x) z^{- p}\phi_1dx = - \int a(x) z^{1 - p}dx -\varepsilon\int a(x) z^{-p} \psi_1dx - \int\limits_{A_x^-}a(x) z^{-p}(z + \varepsilon \psi_1)dx.$$
		The last assertion provide us
		\begin{equation}\label{estimsing P(z)}
			- \int a(x) z^{- p}\phi_1dx \le - \int a(x) z^{1 - p}dx -\varepsilon\int a(x) z^{-p} \psi_1dx - \varepsilon \int\limits_{A_x^-}a(x)z^{-p}\psi_1dx.   
		\end{equation}
		Similarly, we mention that
		\begin{equation}\label{estimsing Q(w)}
			-\int b(x) w^{- q}\phi_2dx \le - \int b(x) w^{1 - q}dx -\varepsilon\int b(x) w^{-q} \psi_2dx - \varepsilon \int\limits_{B_x^-}b(x)w^{-q}\psi_2dx. 
		\end{equation}
		Recall also that $(z, w) \in \mathcal{N}_{ \lambda^*}^0$, i.e., we know that $E_{\lambda^*}''(z, w)(z, w)^2 = 0$. Hence, by using \eqref{prodintsing}, \eqref{1singus}, \eqref{I'sing}, \eqref{IIestsing}, \eqref{II'estsing}, \eqref{termoteta1}, \eqref{acopsingbeta},
		\eqref{estimsing P(z)} and \eqref{estimsing Q(w)} together, we prove the following estimates:
		\begin{eqnarray}\label{desimpsing2}
			0& \le & E_{\lambda^*}''(z, w)(z, w)^2 + \varepsilon 2\left <(z, w), (\psi_1, \psi_2)\right> - \varepsilon\theta\alpha \int  \vert z\vert^{\alpha - 2} z \psi_1 \vert w \vert^\beta dx \nonumber\\
			&&- \varepsilon\theta\beta \int  \vert z\vert^{\alpha} \vert w \vert^{\beta - 2}w \psi_2dx 
			- \varepsilon\lambda^* (1 - p)\int a(x) z^{-p}\psi_1 dx \nonumber\\
			&& - \varepsilon\lambda^*(1 - q)\int b(x) w^{-q}\psi_2dx - \varepsilon\int\limits_{A_x^-} \int\limits_{A_y^-} \frac{\left[ z(x)-z(y)\right] \left[ \psi_1(x) - \psi_1(y)\right]}{\vert x-y\vert^{N+2s}} dydx  \nonumber\\
			&&- 2\varepsilon \int\limits_{A_x^-} \int\limits_{A_y^+} \frac{\left[ z(x)-z(y)\right]^+ \left[ -\psi_1(y) + \psi_1(x)\right]}{\vert x-y\vert^{N+2s}} dydx - \varepsilon\int\limits_{A_x^-} \int\limits_{A_y^-} \frac{\left[ w(x)-w(y)\right] \left[ \psi_2(x) - \psi_2(y)\right]}{\vert x-y\vert^{N+2s}} dydx\nonumber\\
			&&- 2\varepsilon \int\limits_{A_x^-} \int\limits_{A_y^+} \frac{\left[ w(x)-w(y)\right]^+ \left[ -\psi_2(y) + \psi_2(x)\right]}{\vert x-y\vert^{N+2s}} dydx - \varepsilon\int\limits_{A_x^-}V_1z\psi_1dx - \varepsilon\int\limits_{B_x^-}V_2w\psi_2dx  \nonumber\\
			&&- \varepsilon\theta\alpha \int\limits_{A_x^-} \vert z\vert^{\alpha} \vert w \vert^{\beta - 2} \psi_1 dx- \varepsilon\theta\beta \int\limits_{B_x^-} \vert z\vert^{\alpha} \vert w \vert^{\beta - 2} \psi_2 dx  - \varepsilon\int\limits_{A_x^-}a(x)z^{-p}\psi_1dx - \varepsilon \int\limits_{B_x^-}b(x)w^{-q}\psi_2dx. 
		\end{eqnarray}
		From non on, by using the last estimate divided by $\varepsilon$ and doing $\varepsilon \to 0$, we infer that all characteristic functions for the sets $A_x^-$,$A_y^-$, $B_x^-$ and $B_y^-$ goes to zero. The key point here is to apply  the Dominated convergence Theorem. In conclusion, we obtain that
		\begin{eqnarray*}
			0& \le & 2\left <(z, w), (\psi_1, \psi_2)\right> - \theta\alpha \int  \vert z\vert^{\alpha - 2} z \psi_1 \vert w \vert^\beta dx - \theta\beta \int  \vert z\vert^{\alpha} \vert w \vert^{\beta - 2}w\psi_2dx \nonumber\\
			&&
			- \lambda^* (1 - p)\int a(x) z^{-p}\psi_1dx - \lambda^* (1 - q)\int b(x) w^{-q}\psi_2dx,\;\;(\psi_1, \psi_2) \in X.
		\end{eqnarray*}
		Now, using $(-\psi_1, -\psi_2) \in X$ instead of $(\psi_1, \psi_2)$, we guarantee that
		\begin{eqnarray*}
			0& = & 2\left <(z, w), (\psi_1, \psi_2)\right> - \theta\alpha \int  \vert z\vert^{\alpha - 2} z \psi_1 \vert w \vert^\beta dx - \theta\beta \int  \vert z\vert^{\alpha} \vert w \vert^{\beta - 2}w \psi_2dx \\
			&&- \lambda^* (1 - p)\int a(x) z^{-p}\psi_1dx - \lambda^* (1 - q)\int b(x) w^{-q}\psi_2dx,\;\;\; (\psi_1, \psi_2) \in X.
		\end{eqnarray*}
		This ends the proof.
	\end{proof}
	\begin{cor}\label{incompatibilidade}
		Suppose ($P_0$), ($P_1$), ($P$), ($V_0$) and ($V_1'$). Then the System $(S_{\lambda^*})$ does not admit weak solutions $(z, w) \in \mathcal{N}^0_{\lambda^*}$.
	\end{cor}
	\begin{proof}
		The proof follows arguing by contradiction. Assume that there exists a weak solution $(z, w) \in \mathcal{N}^0_{\lambda^*}$ for the System $(S_{\lambda^*})$. Now, by using Proposition \ref{Neharilambda*}, we infer that  
		\begin{eqnarray}\label{solufracaauxiliar}
			2\left <(z, w), (\psi_1, \psi_2)\right> - \theta\alpha \int  \vert z\vert^{\alpha - 2} z \psi_1 \vert w \vert^\beta dx - \theta\beta \int  \vert z\vert^{\alpha} \vert w \vert^{\beta - 2} w \psi_2dx  && \nonumber\\
			- \lambda^* \left[(1 - p)\int a(x) z^{-p}\psi_1dx + (1 - q)\int b(x) w^{-q}\psi_2dx \right]= 0,\;\;\; (\psi_1, \psi_2) \in X&&
		\end{eqnarray}
		Furthermore, we know that  
		\begin{eqnarray}\label{solfracasisprinc*}
			\left <(z, w), (\psi_1, \psi_2)\right> & = & \lambda^*\left( \int a(x) z^{-p}\psi_1dx + \int b(x) w^{-q}\psi_2dx\right) + \theta\frac{\alpha}{\alpha + \beta} \int  \vert z\vert^{\alpha - 2} z \psi_1 \vert w \vert^\beta dx\nonumber \\
			&& + \theta\frac{\beta}{\alpha + \beta} \int  \vert z \vert^{\alpha}\vert w \vert^{\beta - 2} w\psi_2dx ,\;\;\; (\psi_1, \psi_2) \in X. 
		\end{eqnarray}
		Hence, by using \eqref{solfracasisprinc*} and \eqref{solufracaauxiliar}, we are able to use the testing function $(\psi_1, \psi_2) = (\psi_1, 0)$ proving that
		\begin{eqnarray*}
			2 \left (\lambda^* \int a(x) z^{-p}\psi_1dx + 
			\theta \frac{\alpha}{\alpha + \beta} \int  \vert z\vert^{\alpha - 2} z \psi_1 \vert w \vert^\beta dx \right)   - \theta\alpha \int  \vert z\vert^{\alpha - 2} z \psi_1 \vert w \vert^\beta dx 
			- \lambda^* [(1 - p)\int a(x) z^{-p}\psi_1dx = 0.&&
		\end{eqnarray*}
		Now, we rewrite the last identity as follows:
		$$\int \left[ a(x) z^{-p} \lambda^*(1 - p) + \theta\alpha\left(\frac{2 -(\alpha + \beta)}{\alpha + \beta}\right) \vert z\vert^{\alpha - 1} \vert w \vert^\beta \right] \psi_1 dx = 0.$$
		In particular, we mention that
		$$a(x) z^{-p} \lambda^*(1 - p) + \theta\alpha\left(\frac{2 -(\alpha + \beta)}{\alpha + \beta}\right) \vert z\vert^{\alpha - 1} \vert w \vert^\beta = 0, \;\;\;\;  \mbox{a.e. in} \;\;\;\; \mathbb{R}^N .$$
		Hence, we obtain that 
		$$a(x) = \left(\frac{\theta\alpha(\alpha + \beta - 2) 
		}{\lambda^*(1 - p)(\alpha + \beta)} \right)z^{p + \alpha - 1} w^{\beta}\;=:\; C_{\alpha, \beta, p, \lambda^*, \theta}\;z^{p + \alpha - 1} w^{\beta}.$$
		Under these conditions, by using Young's inequality with $r' = 2^*_s/\beta$ and $r = 2^*_s/(2^*_s - \beta)$, there exists a constant $\tilde{C}>0$ such that
		$$\int a(x)dx  \le \tilde{C} \int  \left(\vert w \vert^{2^*_s} + \vert z \vert^{(p + \alpha - 1)r}\right)dx\;\;<\;\;\infty.$$
		In view of hypotheses ($P$)  and ($P_1$) we mention that
		$(p + \alpha - 1) r = (p + \alpha - 1)\frac{2^*_s}{2^*_s - \beta} \le 2^*_s$ holds whenever $2 < \alpha + \beta < 2^*_s$ and $0 < p < 1$. Similarly, we obtain that $(p + \alpha - 1)\frac{2^*_s}{2^*_s - \beta} \ge 2$ holds assuming that $\beta \ge \frac{2^*_s}{2}(3 - \alpha -p).$
		This is a contradiction due to the fact that $a\notin L^1(\mathbb{R}^N)$. This ends the proof.
	\end{proof}
	
	Now, we are looking for weak solutions for the System \eqref{sistema Principal singular} taking into account the following minimization problems:
	\begin{equation}\label{C^N-sing} C_{\mathcal{N}_{\lambda}^-} :=  \inf\limits_{(u, v) \in \mathcal{N}_{\lambda}^-} E_{_\lambda}(u, v);
	\end{equation}
	and
	\begin{equation}\label{C^N+sing} 
		C_{N^{+}_{\lambda}\cap\mathcal{A}} :=  \inf\limits_{(u, v) \in \mathcal{N}_{\lambda}^+\cap \mathcal{A}} E_{_\lambda}(u, v).
	\end{equation}
	
	\begin{rmk}\label{N- cont Asing} It is important to stress that $\mathcal{N}^-_{\lambda} \subset \mathcal{A}$. In fact, for any $(u, v) \in \mathcal{N}^-_{\lambda}$, we prove that $B(u, v) > 0$. Here was used the that that $E'_{_{\lambda}}(u, v)(u, v)=0$ implies
		$$0> E''_{_{\lambda}}(u, v)(u, v)^2 =   \lambda(1 + p)P(u) + \lambda (1 + q)Q(v) + \theta (2-\alpha - \beta) B(u, v)$$
		holds true for each $(u,v) \in \mathcal{N}^-_{\lambda}$.
	\end{rmk} 
	
	\begin{prop}\label{E coersivasing}
		Suppose ($P_0$), ($P$), ($V_0$) and ($V_1'$). The energy functional $E_{\lambda}: X \to \mathbb{R}$ is coercive in the Nehari set.
	\end{prop}
	\begin{proof}    
		Let $(u, v) \in \mathcal{N}_{\lambda}$ be fixed. Hence, we mention that $E'_{_{\lambda}}(u, v)(u, v) = 0$. As a consequence, we see that 
		$$\theta\int  |u|^\alpha|v|^\beta dx = \Vert (u, v) \Vert ^2 - \lambda\int  a(x) \vert u \vert^{1 - p}dx - \lambda\int  b(x) \vert v \vert^{1 - q}dx.$$
		Now, we write $E_{\lambda}(u, v) = E_{\lambda}(u, v) - \frac{1}{\alpha + \beta}E_{\lambda}'(u, v)(u, v)$. The last assertion implies that
		\begin{eqnarray*}
			E_{_{\lambda}}(u, v) &=& \left ( \frac{1}{2} - \frac{1}{\alpha + \beta} \right )A(u, v) + \lambda\left ( -\frac{1}{1 - p} + \frac{1}{\alpha + \beta}\right )P(u) +  \lambda\left ( -\frac{1}{1 - q} + \frac{1}{\alpha + \beta}\right )Q(v).
		\end{eqnarray*}
		Now, by using \eqref{bounded}, we infer that 
		\begin{eqnarray}\label{coerciva}
			E_{_{\lambda}}(u, v) & \ge & C_1 \Vert(u, v)\Vert^2 - C_2 \Vert (u, v) \Vert^{1 - p} - C_3\Vert (u, v) \Vert^{1 - q}.
		\end{eqnarray} 
		As a consequence, we deduce that $E_{_{\lambda}}(u,v) \to +\infty$ as $\Vert (u, v) \Vert \to +\infty$. This ends the proof.
	\end{proof}

	\begin{prop}\label{convfortN-}
		Assume that ($P_0$), ($P$), ($V_0$), ($V_1'$) and $\lambda \in (0, \lambda^*)$ hold. Consider a minimizer sequence $(u_k, v_k) \in \mathcal{N}_{\lambda}^-$ for $C_{\mathcal{N}_{\lambda}^-}$. Then there exists $(u, v) \in \mathcal{N}_{\lambda}^-$ such that $(u_k, v_k) \to (u, v)$ in $X$. 
	\end{prop}
	\begin{proof}
		Consider a sequence $(u_k, v_k) \in \mathcal{N}_{\lambda}^-$ such that $E_{\lambda}(u_k, v_k) \to C_{\mathcal{N}_{\lambda}^-}$. Here we observe that the sequence $(u_k, v_k)$ is bounded, see Proposition \ref{E coersivasing}. Up to a subsequence, we know that $(u_k, v_k) \rightharpoonup (u, v)$ for some $(u,v)\in X$. Now, by using Proposition \ref{N-longe zerosing}, together with \eqref{limitação u vsing}, we obtain that $u \ne 0$ and $v \ne 0$. According to Proposition \ref{tn-,tn+sing}, there exist $t_n^+(u, v) > 0$ and $t_n^-(u, v) > 0$ such that $t_n^+(u, v) < t_n(u, v) < t_n^-(u, v)$ where $t_n^+(u, v)(u, v) \in \mathcal{N}_{\lambda}^+$ and $t_n^-(u, v)(u, v) \in \mathcal{N}_{\lambda}^-$.
		Under these conditions, we deduce that  
		\begin{equation}\label{CN-uv}
			C_{ \mathcal{N}_{_{\lambda}}^-} \le E_{\lambda}(t_n^-(u, v)(u, v)). 
		\end{equation}
		On the other hand, we see that 
		$E_{\lambda}'(u_k, v_k)(u_k, v_k) = 0,\, \, E_{\lambda}''(u_k, v_k)(u_k, v_k)^2 < 0,$
		$$\Vert (u_k, v_k)\Vert \ge C > 0, \,\,  \int \vert u_k \vert^{\alpha}\vert v_k \vert^{\beta}dx \ge \delta_c > 0, \,\, k \in \mathbb{N}.$$
		Now we claim that $(u_k, v_k) \to (u,v)$ in $X$. The proof follows arguing by contradiction assuming that $(u_k, v_k)$ does not converge. Therefore, 
		$$\Vert(tu, tv)\Vert^2 < \liminf\limits_{k \to \infty} \Vert(tu_k, tv_k)\Vert^2, \;\;\; \forall\;\; t > 0.$$
		In particular, we obtain that
		\begin{equation}\label{EuvEukvk}
			E_{\lambda}(t(u, v)) < \liminf\limits_{k \to \infty} E_{\lambda}(t(u_k, v_k)),  \;\;\; \forall\;\; t > 0.
		\end{equation}
		Hence, for $t = t_n^-(u, v)$, we deduce that
		\begin{equation}\label{EuvEukvk}
			E_{\lambda}(t_n^-(u, v)(u, v)) < \liminf\limits_{k \to \infty} E_{\lambda}(t_n^-(u, v)(u_k, v_k)).
		\end{equation}
		Thus, using \eqref{CN-uv} and \eqref{EuvEukvk}, we infer that
		\begin{equation}\label{c,liminEsing}  
			C_{\mathcal{N}_{_{\lambda}}^-} < \liminf\limits_{k \to \infty}E_{_{\lambda}}(t_n^-(u, v)(u_k, v_k)). 
		\end{equation}
		Now, by using Remark \ref{E E' E'' infer sem frac}, we see that 
		$$E_{\lambda}'(u, v)(u, v) < \liminf\limits_{k \to \infty}E_{\lambda}'(u_k, v_k)(u_k, v_k) = 0 \;\;\;\; \mbox{and}\;\;\;\;E_{\lambda}''(u, v)(u, v)^2 < \liminf\limits_{k \to \infty}E_{\lambda}''(u_k, v_k)(u_k, v_k)^2 \le 0.$$
		It is not hard to verify that
		$\gamma_{\lambda}'(t) = E_{\lambda}'(tu, tv)(u, v) < 0, \;\;\; 0 < t < t_n^+(u, v)\;\;\; \mbox{and}\;\;\; t > t_n^-(u, v).$
		Recall also that
		$$0 = E_{\lambda}'(tu, tv)(u, v) \le \liminf\limits_{k \to \infty} E_{\lambda}'(tu_k, tv_k)(u_k, v_k), \;\;\; \mbox{for} \;\;\; t = t_n^+(u, v).$$
		Thus, $1 > t_n^-(u, v)$ holds. 
		Notice also that 
		$E_{\lambda}'(tu_k, tv_k)(u_k, v_k) < 0$ for each  $0 < t < t_n^+(u_k, v_k)$. Moreover, we mention that
		$$E_{\lambda}'(tu, tv)(u, v) < \liminf\limits_{k \to \infty}E_{\lambda}'(tu_k, tv_k)(u_k, v_k).$$
		Now, by using the last assertion, we obtain that 
		$$0 = E_{\lambda}'(t_n^+(u, v)(u_k, v_k))(u_k, v_k) < \liminf\limits_{k \to \infty}E_{\lambda}'(t_n^+(u, v)(u_k, v_k))(u_k, v_k).$$
		Hence, we infer that $t_n^+(u_k, v_k) \le t_n^+(u, v) \le t_n^-(u, v) \le t_n^-(u_k, v_k) = 1.$ As a consequence, we obtain
		$$C_{\mathcal{N}^-_{\lambda}} < \liminf\limits_{k \to \infty}E_{_{\lambda}}(t_n^-(u, v)(u_k, v_k)) \le \liminf\limits_{k \to \infty}E_{_{\lambda}}(u_k, v_k) = C_{\mathcal{N}^-_{\lambda} }.$$ 
		This is is a contradiction showing that $\Vert (u, v) \Vert^2 = \liminf\limits_{k \to \infty}\Vert (u_k, v_k) \Vert^2$.  Therefore, $(u_k, v_k)\to (u, v)$ in $X$. This ends the proof.
	\end{proof}
	
	\begin{prop}\label{convforte N^+sing}
		Suppose ($P_0$), ($P$), ($V_0$) and ($V_1'$). Assume that $\lambda \in (0, \lambda^*)$. Let $(u_k, v_k)\in \mathcal{N}^+_{\lambda}$ be a minimizing sequence for $C_{\mathcal{N}^+_{\lambda}}$. Then, up to a subsequence, there exists $(u, v) \in \mathcal{A}$ such that $(u_k, v_k) \to (u, v)$ in $X$ where $(u, v) \in \mathcal{N}^+_{\lambda}$.
	\end{prop}
	\begin{proof}
		First, we consider the singular elliptic problems:
		\begin{equation}\label{eqsingu}
			(-\Delta)^su +V_1(x)u = \lambda a(x)u^{-p}, 
		\end{equation} 
		\begin{equation}\label{eqsingv}
			(-\Delta)^sv +V_2(x)v = \lambda b(x)v^{-q}.
		\end{equation} 
		Recall that $V_i: \mathbb{R}^N \to \mathbb{R}, i = 1, 2$ are continuous potentials. 
		Furthermore, we consider the following associated energy functionals:
		$$J_{1, \lambda}(u) = \frac{1}{2}\Vert u \Vert^2 - \frac{\lambda}{1 - p}\int a(x)\vert u \vert^{1 -p}dx,\;\;\;\; \mbox{and}\;\;\;\;J_{2, \lambda}(v) = \frac{1}{2}\Vert v \Vert^2 - \frac{\lambda}{1 - q}\int b(x)\vert v \vert^{1 - q}dx.$$
		Let us consider also the Nehari sets as follows
		$$\mathcal{N}_{1, \lambda} = \{u \in X_1\setminus_{\{0\}}; J_{1, \lambda}'(u)u = 0 \};\;\;\;\;\;\mathcal{N}_{2, \lambda} = \{v \in X_2\setminus_{\{0\}}; J_{2, \lambda}'(v)(v) = 0 \};$$$$\mathcal{N}_{1, \lambda}^+ = \{u \in \mathcal{N}_{1, \lambda}; J_{1, \lambda}''(u)u^2 > 0 \};\;\;\;\;\;\mathcal{N}_{2, \lambda}^+ = \{v \in \mathcal{N}_{2, \lambda}; J_{2, \lambda}''(v)(v)^2 > 0 \}.$$
		Consider $w_1 \in \mathcal{N}_{1, \lambda}^+$ and $w_2 \in \mathcal{N}_{2, \lambda}^+$ be weak solutions for the singular elliptic problem given in \eqref{eqsingu} and \eqref{eqsingv}, respectively. Now, applying the ideas developed by \cite{Yijing2001,silvasing2018}, we see that $w_1, w_2 > 0$. It is important to emphasize that problems \eqref{eqsingu} and \eqref{eqsingv} does not admit $u = 0$ as trivial weak solution. Thus, we deduce that $(w_1, w_2) \in \mathcal{A}$. 
		Furthermore, assuming that $(u, v)\in \mathcal{N}_{\lambda}^+$ is a weak solution for the System \eqref{sistema Principal singular}, we deduce that \eqref{equfracasingula} is verified. Furthermore, we mention that $E_{\lambda}''(u, v)(u, v)^2 > 0.$
		In particular, we have that
		$E_{\lambda}'(u, v)(u, v) = 0 \;\;\;\mbox{and}\;\;\; C_{\mathcal{N}_{\lambda}^+} = E(u, v).$
		It remains to show that $(u_k, v_k) \to (u, v)$.  Recall that $E_{\lambda}(u_k, v_k) \to C_{\mathcal{N}^+_{\lambda}}$. Furthermore, we see that
		$$E_{\lambda}(u, v) \le \liminf\limits_{k \to \infty} E_{_{\lambda}}(u_k, v_k) = C_{\mathcal{N}^+_{\lambda}}\;\;\;\mbox{and}\;\;\;E'_{\lambda}(u, v)(u, v) \le \liminf\limits_{k \to \infty} E'_{_{\lambda}}(u_k, v_k)(u_k, v_k) = 0.$$
		In the same way, we obtain that $E'_{\lambda}(u, v)(u, v) \le 0$. Therefore, $E'_{\lambda}(tu, tv)(tu, tv) \le 0$ for each  $t \in (0, t_n^+(u, v))$. The last assertion implies that $t_n^+(u, v) \ge 1$. Moreover, $t_n^+(u, v)(u, v) \in \mathcal{N}^+_{\lambda}$. Hence, 
		$$C_{\mathcal{N}^+_{\lambda}} \le E_{\lambda}(t_n^+(u, v)(u, v)) \le E_{\lambda}(u, v) \le \liminf\limits_{k \to \infty} E_{\lambda}(u_k, v_k) = C_{\mathcal{N}^+_{\lambda}}.$$
		As a consequence, $(u_k, v_k) \to (u, v)$ in $X$. In particular, we obtain that $(u, v) \in \mathcal{N}_{\lambda}^+\cap \mathcal{A}$. This ends the proof.
	\end{proof}
	\begin{prop}\label{CN+ negsig}
		Suppose ($P_0$), ($P$), ($V_0$), ($V_1'$) and $\lambda \in (0, \lambda^*)$. Then $C_{\mathcal{N}^+_{\lambda}\cap \mathcal{A}}< 0$.
	\end{prop}
	\begin{proof}
		Firstly, by using the Proposition \ref{tn-,tn+sing}, the $t \mapsto E_{\lambda}(tu, tv)$ is decreasing for each $0 < t < t_n^+(u, v)$. As a consequence, $E_{\lambda}(t_n^+(u, v)(u, v)) < 0$. In particular, we obtain that $C_{\mathcal{N}^+_{\lambda}\cap \mathcal{A}} < E_{\lambda}(t_n^+(u, v)(u, v)) < 0$. This finishes the proof. 
	\end{proof}
	\begin{prop}\label{elemN+N-inf}
		Suppose ($P_0$), ($P$), ($V_0$) and ($V_1'$). Assume also that $0 < \lambda < \lambda^*$ holds. Then for some $(u, v) \in \mathcal{N}_{\lambda}^+$ and $(z, w) \in \mathcal{N}_{\lambda}^-$, 
		\begin{equation}\label{pribminsing}
			E_{\lambda}(u, v) = C_{\mathcal{N}_{\lambda}^+} \;\;\;\; \mbox{and}\;\;\;\; E_{\lambda}(z, w) = C_{\mathcal{N}_{\lambda}^-}.
		\end{equation}
	\end{prop}
	\begin{proof}
		Let us consider a minimizing sequence $(u_k, v_k) \in \mathcal{N}_{_{\lambda}}^+$. Here we observe that $(u_k, v_k) \to (u, v)$ in $X$, see Proposition \ref{convforte N^+sing}. Furthermore, we obtain that
		$$E_{\lambda}(u, v) = \liminf\limits_{k\to \infty}E_{\lambda}(u_k, v_k) = C_{\mathcal{N}_{\lambda}^+}.$$
		Similarly, by using Proposition \ref{convfortN-}, the sequence $(z_k, w_k)$ satisfies 
		$(z_k, w_k) \to (z, w)$ in $X$ for some $(z, w) \in \mathcal{N}_{\lambda}^-$. Under these conditions, we observe that
		$$ E_{\lambda}(z, w) = \liminf\limits_{k \to \infty}E_{\lambda}(z_k, w_k) = C_{\mathcal{N}_{\lambda}^-}.$$ This ends the proof. 
	\end{proof}
	\begin{prop}\label{passo1sing}
		Suppose ($P_0$), ($P$), ($V_0$), ($V_1'$) and $0 < \lambda < \lambda^*$. Let $(u, v)\in \mathcal{N}_{\lambda}^+$ and $(z, w) \in \mathcal{N}_{\lambda}^-$ such that $E_{\lambda}(u, v) = C_{\mathcal{N}_{\lambda}^+}$ and $E_{\lambda}(z, w) = C_{\mathcal{N}_{\lambda}^-}$ Then for every pair $(\psi_1, \psi_2) \in X_+$ there exists $\varepsilon_0 > 0$ such that 
		\begin{itemize}
			\item[$i)$]$E_{\lambda}(u, v) \le E_{\lambda}(u + \varepsilon\psi_1, v + \varepsilon \psi_2)$,  $\forall \varepsilon \in [0, \varepsilon_0]$; 
			\item[$ii)$] Given $(z_{\varepsilon}, w_{\varepsilon}) = (z, w) + \varepsilon(\psi_1, \psi_2)$ we obtain that $ t_n^-(z_{\varepsilon}, w_{\varepsilon}) \to 1$ as $\varepsilon \to 0$. 
		\end{itemize}
	\end{prop}
	\begin{proof}
		$i)$ Let $(\psi_1, \psi_2) \in X_+$ be fixed and $(u_{\varepsilon}, v_{\varepsilon}) = (u, v) + \varepsilon(\psi_1, \psi_2)$. Now, we write
		$$\gamma_{_{\lambda, (u, v)}} (t) : = E_{\lambda}(tu, tv) \;\;\; \mbox{and}\;\;\; \gamma_{_{\lambda, (u_{\varepsilon}, v_{\varepsilon})}} (t):= E_{\lambda}(tu_{\varepsilon}, tv_{\varepsilon}). $$ 
		Now, by using \eqref{E'' cap 2}, we have that 
		$$\gamma_{_{\lambda, (u_{\varepsilon}, v_{\varepsilon})}}(1) = 2A(u_{\varepsilon}, v_{\varepsilon}) - \lambda(1 - p)P(u_{\varepsilon}) -\lambda(1 - q)Q( v_{\varepsilon}) - \theta (\alpha + \beta)B(u_{\varepsilon}, v_{\varepsilon}). $$
		Notice also that $\varepsilon \mapsto \gamma_{_{\lambda, (u_{\varepsilon}, v_{\varepsilon})}}(1)$ is continuous for each $\varepsilon > 0$. Hence, there exists $\varepsilon_0 > 0$ such that
		$\gamma_{_{\lambda, (u_\varepsilon, v_\varepsilon)}}''(1) > 0$ holds for each $\varepsilon \in [0, \varepsilon_0]$. As a consequence, we mention that 
		$$E_{\lambda}''(u_{\varepsilon}, v_{\varepsilon})(u_{\varepsilon}, v_{\varepsilon})^2 > 0,\;\;\;\; \mbox{for}\;\;\; \varepsilon \in [0, \varepsilon_0].$$
		The last assertion implies that $1 \in (0, t_n^-(u_{\varepsilon}, v_{\varepsilon}))$. Recall also that $t\mapsto E_{\lambda}(tu{\varepsilon}, t v_{\varepsilon})$ with $t \in [0, t_n^-(u_{\varepsilon}, v_{\varepsilon})]$ assumes its minimum at $t_n^+(u_{\varepsilon}, v_{\varepsilon})$. Therefore, we obtain that
		$$E_{\lambda}(u, v) \le E_{\lambda}(t_n^+(u_{\varepsilon}, v_{\varepsilon})(u_{\varepsilon}, v_{\varepsilon})) \le E_{\lambda}(u_{\varepsilon}, v_{\varepsilon}).$$
		
		$ii)$ The main goal here is to show that given $(z, w) \in \mathcal{N}^-_{\lambda}$ there exists a neighborhood of $(z, w)$ such that all elements can be projected into $\mathcal{N}^-_{\lambda}$. Consider $z_\varepsilon = z + \varepsilon \psi_1$ and $w_\varepsilon = w + \varepsilon \psi_2$ where $(z, w) \in \mathcal{N}^-_{\lambda}$. Define the function $F: (0, \infty) \times X \to \mathbb{R}$ given by: 
		$$F(t, (\psi_1, \psi_2)) = A(t(z_\varepsilon, w_\varepsilon)) - \lambda P(t(z_\varepsilon
		)) - \lambda Q(t(w_\varepsilon
		)) - \theta B(t(z_\varepsilon, w_\varepsilon)).$$
		Since $(z, w) \in \mathcal{N}_{\lambda}$ we obtain that
		$F(1, (0, 0)) = E_{\lambda}'(z, w)(z, w) = 0.$ Furthermore, we know that $(z, w) \in \mathcal{N}_{\lambda}^-.$ It is easy to see that 
		$$\frac{\partial F}{\partial t} (1, (0, 0))= \gamma''_{_{\lambda, (z, w)}}(1) < 0.$$
		Hence, by using the Implicit Function Theorem, there exist open $I \subset \mathbb{R}$ a neighborhood of $1$, $\Omega$ a neighborhood of $(z, w)$ and $f: \Omega \to I$ such that $f(z, w) = 1$ is verified. Notice also that $f \in C^0 (\Omega, I)$ is unique and  
		$$F(f(z_\varepsilon, w_\varepsilon), (z_\varepsilon, w_\varepsilon)) = 0, \;\;\;\; (z_\varepsilon, w_\varepsilon) \in \Omega, \,\, \partial_t F(f(z_\varepsilon, w_\varepsilon), (z_\varepsilon, w_\varepsilon)) < 0,\;\;\; \mbox{for all}\;\;\; \varepsilon \in (0, \varepsilon_0).$$
		In particular, we see that $f(z_\varepsilon, w_\varepsilon)(z_\varepsilon, w_\varepsilon) \in \mathcal{N}_{\lambda}^-.$
		In fact, for each $\varepsilon$ small enough, we infer that
		$$\frac{\partial F}{\partial t}(f(z_\varepsilon, w_\varepsilon),(z_\varepsilon, w_\varepsilon)) < 0.$$
		Under these conditions, we prove that
		\begin{eqnarray*}
			2f(z_\varepsilon, w_\varepsilon)A(z_\varepsilon, w_\varepsilon) - \lambda (1 - p) f(z_\varepsilon, w_\varepsilon)^{-p}P(z_\varepsilon)- \lambda (1 - q) f(z_\varepsilon, w_\varepsilon)^{-q}Q(w_\varepsilon) 
			- \theta (\alpha + \beta)f(z_\varepsilon, w_\varepsilon)^{\alpha + \beta - 1}B(z_\varepsilon, w_\varepsilon) < 0. &&
		\end{eqnarray*}
		Therefore, we deduce that 
		$$2A(f(z_\varepsilon, w_\varepsilon)(z_\varepsilon, w_\varepsilon)) - \lambda (1 - p) P(f(z_\varepsilon, w_\varepsilon) z_\varepsilon)- \lambda (1 - q) Q(f(z_\varepsilon, w_\varepsilon)w_\varepsilon) - \theta (\alpha + \beta)B(f(z_\varepsilon, w_\varepsilon)(z_\varepsilon, w_\varepsilon)) = \gamma''_{\lambda, (z_\varepsilon, w_\varepsilon)}(1) < 0.$$
		Thus, we obtain that $f(z_\varepsilon, w_\varepsilon)(z_\varepsilon, w_\varepsilon) \in \mathcal{N}_{\lambda}^-$. Obviously, we mention that $t_n^-(z_\varepsilon, w_\varepsilon) = f(z_\varepsilon, w_\varepsilon)$ in $\Omega$. Furthermore, we observe that 
		$$\lim\limits_{\varepsilon \to 0} (z_\varepsilon, w_\varepsilon) = (z, w) \;\;\;\; \mbox{and}\;\;\;\; \lim\limits_{\varepsilon \to 0}t_n^-(z_\varepsilon, w_\varepsilon) = 1 = t_n^-(z, w).$$
		This completes the proof.
	\end{proof}
	
	\begin{lem}\label{minimizasing} 
		Suppose $(P_0),(P), (V_0),(V_1')$ and $\lambda\in (0,\lambda^*)$. Let $(u, v)\in \mathcal{N}_{\lambda}^+$ and $(z, w)\in \mathcal{N}_{\lambda}^-$ be such that $ E_{\lambda}(u, v) = C_{\mathcal{N}_{\lambda}^+}$ and $ E_{\lambda}(z, w) = C_{\mathcal{N}_{\lambda}^-}$. Then $u, v, z, w > 0$ a.e. in $\mathbb{R}^N$ and for each pair $(\psi_1, \psi_2) \in X_+$, we obtain
		\begin{eqnarray}\label{IneqsolfracaN+}
			\left<(u, v), (\psi_1, \psi_2)\right> -  \frac{\theta}{\alpha + \beta}\int \alpha\vert u\vert^{\alpha - 2}u \psi_1 \vert v \vert^{\beta}+ \beta\vert u\vert^{\alpha} \vert v \vert^{\beta - 2} v \psi_2dx   - \lambda\int a(x)\vert u\vert^{-p}\psi_1 +b(x)\vert v\vert^{-q}\psi_2dx \ge 0 .
		\end{eqnarray}    \begin{eqnarray}\label{IneqsolfracaN-}
			\left<(z, w), (\psi_1, \psi_2)\right> - \frac{\theta}{\alpha + \beta}\int \alpha\vert z\vert^{\alpha - 2}z \psi_1 \vert w \vert^{\beta} + \beta\vert z\vert^{\alpha} \vert w \vert^{\beta - 2} w \psi_2dx   - \lambda\int a(x)\vert z\vert^{-p}\psi_1 +  b(x)\vert w\vert^{-q}\psi_2dx \ge 0 .
		\end{eqnarray}
	\end{lem}
	\begin{proof}
		The proof follows the same ideas discussed in Step 2 of Proposition \ref{Neharilambda*}. Indeed, by using Proposition \ref{passo1sing}, we deduce that $E_\lambda(u_\varepsilon, v_\varepsilon) - E_\lambda(u, v) \ge 0.$
		As a consequence, we infer that
		\begin{eqnarray*} 		0 & \le &   \frac{1}{2}\Vert (u_\varepsilon, v_\varepsilon) \Vert^2 - \frac{\lambda}{1 - p}\int a(x)\vert u_\varepsilon \vert ^{1 - p}dx - \frac{\lambda}{1 - q} \int b(x) \vert v_\varepsilon \vert ^{1 - q}dx + \nonumber\\
			& - &  \frac{\theta}{\alpha + \beta}\int \vert u_\varepsilon \vert^{\alpha}\vert v_\varepsilon \vert^{\beta} dx - \frac{1}{2}\Vert (u, v) \Vert^2 + \frac{\lambda}{1 -p}\int a(x)\vert u \vert ^{1 - p}dx + \nonumber\\
			&& + \frac{\lambda}{1 - q} \int b(x) \vert v \vert ^{1 - q}dx +  \frac{\theta}{\alpha + \beta}\int \vert u \vert^{\alpha}\vert v \vert^{\beta} dx.
		\end{eqnarray*}
		The last expression implies that
		\begin{eqnarray*} 		\frac{\lambda}{1 - p}\left(\int a(x)\left[\vert u_\varepsilon \vert ^{1 - p} - \vert u \vert ^{1 - p}\right]dx \right)+ \frac{\lambda}{1 - q}\left(\int b(x) \left[\vert v_\varepsilon \vert ^{1 - q} - \vert v \vert ^{1 - q}\right]dx\right)
			& \le &    \nonumber\\
			\frac{1}{2}\Vert (u_\varepsilon, v_\varepsilon) \Vert^2	- \frac{1}{2}\Vert (u, v) \Vert^2+
			\frac{\theta}{\alpha + \beta}\int \vert u_\varepsilon \vert^{\alpha}\vert v_\varepsilon \vert^{\beta} dx -\frac{\theta}{\alpha + \beta}\int \vert u \vert^{\alpha}\vert v \vert^{\beta} dx.&&
		\end{eqnarray*}
		Now, using the last assertion and doing $\varepsilon \to 0$ we obtain that
		\begin{eqnarray*}
			\lambda\left[\int \frac{a(x)\left(\vert u_\varepsilon\vert^{1 - p} - \vert u \vert^{1 - p}\right)}{(1 - p)\varepsilon}dx + \int \frac{b(x)\left(\vert v_\varepsilon\vert^{1 - q} - \vert v \vert^{1 - q}\right)}{(1 - q)\varepsilon}dx\right]
			\le \lim\limits_{\varepsilon \to 0} \left [ \frac{\Vert (u_\varepsilon, v_\varepsilon)\Vert^2 - \Vert (u, v)\Vert^2}{2 \varepsilon} - \theta\int \frac{\vert u_\varepsilon\vert^{\alpha}\vert v_\varepsilon\vert^{\beta} -\vert u\vert^{\alpha}\vert v\vert^{\beta}}{(\alpha + \beta)\varepsilon}dx \right].&&
		\end{eqnarray*}
		From now on, using the similar ideas employed in the proof of Proposition \ref{Neharilambda*}, the singular term has Gateaux derivative for any direction $(\psi_1, \psi_2) \in X_+$. In fact, we observe that 
		$$R_n(u, v) = \frac{A(u, v) - \theta B(u, v)}{P(u) + Q(v)} = \lambda.$$
		Thus, choosing $(\psi_1, \psi_2) \in X_+$, we infer that 
		$\left<R_n'(u, v),(\psi_1, \psi_2)\right>\ge 0.$
		Using the same ideas provided in the proof of Proposition \ref{Neharilambda*} we obtain that
		$$f(u, v) = \frac{1}{ \int a(x)\vert u \vert^{1 
				- p}dx +  \int b(x)\vert v \vert^{1 
				- q}dx} \;\;\; \mbox{and} \;\;\; g(u, v) = \Vert(u, v)\Vert^2 - \theta \int \vert u \vert^{\alpha}\vert v \vert^{\beta}dx.$$
		Under these conditions, by using the directional derivatives $f'(u, v)(\psi_1, \psi_2)$, $g'(u, v)(\psi_1, \psi_2)$, we obtain that
		\begin{eqnarray}\label{ineqsolsing}
			2\left<(u, v), (\psi_1, \psi_2)\right> - \theta \alpha \int \vert u \vert^{\alpha -2}u \psi_1\vert v\vert^{\beta}dx - \theta \beta \int \vert u \vert^{\alpha}\vert v\vert^{\beta - 2}v\psi_2dx &&\nonumber \\
			- \lambda\left((1 - p)\int a(x)\vert u \vert^{-p}\psi_1dx + (1 - q)\int b(x)\vert v \vert^{-q}\psi_2dx\right) \ge 0,\;\;\;\forall\;(\psi_1, \psi_2) \in X_+.&&
		\end{eqnarray}
		Now, we shall prove that \eqref{IneqsolfracaN-} holds true. Firstly, by using Proposition \ref{passo1sing}, item $ii)$, given any $(z, w) \in \mathcal{N}_{\lambda}^-$ such that $E_{\lambda}(z, w) = C_{\mathcal{N}_{\lambda}^-}$, there exists $t_n^-(z_\varepsilon, w_\varepsilon)$ such that $t_n^-(z_\varepsilon, w_\varepsilon)(z_\varepsilon, w_\varepsilon) \in \mathcal{N}_{\lambda}^-$. Moreover, we know that $t_n^-(z_\varepsilon, w_\varepsilon) \to 1$ as $t_n^-(z, w)$. Hence, we infer that
		\begin{eqnarray*}
			E_{\lambda}(t_n^-(z_\varepsilon, w_\varepsilon)(z_\varepsilon, w_\varepsilon)) &\ge& E_{\lambda}(z, w) = \gamma_{\lambda, (z, w)}(1) \ge \gamma_{\lambda, (z, w)}(t_n^-(z_\varepsilon, w_\varepsilon))= E_{\lambda}(t_n^-(z_\varepsilon, w_\varepsilon)(z, w)).
		\end{eqnarray*}
		Moreover, we observe that
		$E_{\lambda}(t_n^-(z_\varepsilon, w_\varepsilon)(z_\varepsilon, w_\varepsilon)) -  E_{\lambda}(t_n^-(z_\varepsilon, w_\varepsilon)(z, w)) \ge 0.$
		Using the last inequality we obtain the following estimate:
		\begin{eqnarray*}
			&&\left[t_n^-(z_\varepsilon, w_\varepsilon)\right]^2\left[\frac{\Vert(z_\varepsilon, w_\varepsilon)\Vert^2 - \Vert(z, w)\Vert^2}{2 \varepsilon}\right]  - \left[t_n^-(z_\varepsilon, w_\varepsilon)\right]^{\alpha + \beta}\left[\int \frac{\vert z_\varepsilon \vert^{\alpha}\vert w_\varepsilon \vert^{\beta} - \vert z \vert^{\alpha}\vert w\vert^{\beta}}{(\alpha + \beta)\varepsilon}dx\right]\\&&
			\geq\;\;\; \left[t_n^-(z_\varepsilon, w_\varepsilon)\right]^{1 - p}\lambda\left[\int a(x)\frac{\vert z_\varepsilon \vert^{1 - p} - \vert z \vert^{1 - p}}{(1 -p)\varepsilon}dx\right] + \left[t_n^-(z_\varepsilon, w_\varepsilon)\right]^{1 - q}\lambda\left[\int b(x)\frac{\vert w_\varepsilon \vert^{1 - q} - \vert w \vert^{1 - q}}{(1 -q)\varepsilon}dx\right]
		\end{eqnarray*}
		where $\varepsilon > 0$ is small enough. Now, for any $(\psi_1, \psi_2) \in X_+$ and doing as $\varepsilon \to 0$, we obtain that
		\begin{eqnarray*}
			\left<(z, w), (\psi_1, \psi_2)\right> - \frac{\theta}{\alpha + \beta}\int \alpha\vert z\vert^{\alpha - 2}z \psi_1\vert w \vert^{\beta} + \beta \vert z\vert^{\alpha} \vert w \vert^{\beta - 2}w\psi_2dx 
			-\lambda\int a(x) \vert z \vert^{-p}\psi_1 + b(x)\vert w\vert^{- q} \psi_2dx \geq 0,
		\end{eqnarray*}
		This ends the proof. 
	\end{proof}
	
	\begin{prop}\label{solucaosing}
		Suppose ($P_0$), ($P$), ($V_0$), ($V_1'$) and $0 < \lambda < \lambda^*$. Then $(u, v) \in \mathcal{N}^+_{\lambda}$ and $(z, w) \in \mathcal{N}^-_{\lambda}$ where $(u,v)$ and $(z,w)$ given by \eqref{pribminsing} are weak solutions for the System \eqref{sistema Principal singular}.
	\end{prop}
	\begin{proof}
		Firstly, we show that $(u, v) \in \mathcal{N}_{\lambda}^+$ is a weak solution for \eqref{sistema Principal singular}. The main idea here is apply the Step $2$ of Proposition \ref{Neharilambda*}. Let us consider $(\phi_1, \phi_2) \in X$ any fixed function. Define $\psi_1 = (u + \varepsilon\phi_1)^+ \ge 0$ and $\psi_2 = (v + \varepsilon\phi_2)^+ \ge 0$ where $\varepsilon > 0$. Hence, we obtain that \eqref{ineqsolsing} holds for $(\psi_1, \psi_2) \in X$. Now, using the same ideas employed in Step 3 for the proof of Proposition \ref{Neharilambda*}, we deduce that \eqref{IneqsolfracaN+} holds for $(\psi, \psi) \in X$. The reverse inequality is obtained using $(-\phi_1, -\phi_2) \in X$ as a test function. Hence, we obtain that
		\begin{eqnarray*}
			\left<(u, v), (\varphi_1, \varphi_2)\right> - \frac{\theta}{\alpha + \beta}\int  \alpha\vert u\vert^{\alpha - 2}u\varphi_1\vert v\vert^{\beta}+\beta\vert u\vert^{\alpha}\vert v\vert^{\beta-2}v\varphi_2dx 
			- \lambda_k \int  a(x)\vert u\vert^{-p}\varphi_1dx + b(x)\vert v\vert^{-q}\varphi_2dx = 0.
		\end{eqnarray*}
		Similarly, \eqref{IneqsolfracaN-} holds for any $(\phi_1, \phi_2) \in X$. This ends the proof.
	\end{proof}

	\section{Multiplicity of solutions for $\lambda = \lambda^*$}
	In the present section we shall consider our main problem assuming that $\lambda = \lambda^*$. In order to clarify the notation, for each $\lambda \in (0, \lambda^*)$ and given any $(u, v)\in X$, we shall use $t_{\lambda}^+(u, v)$ and $t_{\lambda}^-(u, v)$ instead of $t_n^+(u, v)$ and $t_n^-(u, v)$, respectively.
	
	\begin{lem}\label{Cn+e Cn-decresc}
		Suppose ($P_0$), ($P$), ($V_0$), and ($V_1'$). Let $(u, v) \in \mathcal{A}$ and $I \subset \mathbb{R}$ be an open interval such that $t_{\lambda}^+(u, v)$ and $t_{\lambda}^-(u, v)$ are well-defined for every $\lambda \in I$. Then, we obtain that
		\begin{itemize}
			\item[$a)$] The functionals $\lambda \mapsto t_{\lambda}^\pm(u, v)$ are $C^{1}(I, \mathbb{R})$. Moreover, $\lambda \mapsto t_{\lambda}^-(u, v)$ is decreasing and $\lambda \mapsto t_{\lambda}^+(u, v)$ is increasing.
			\item[$b)$] For each $(u, v) \in \mathcal{A}$, the functional $\lambda \mapsto E_{(u, v)}^+(\lambda) := E_{\lambda}(t_{\lambda}^+(u, v)(u, v))$ is $C^{1}(I, \mathbb{R})$ and decreasing.
			\item[$c)$] For each $(u, v) \in \mathcal{A}$, the functional $\lambda \mapsto E_{(u, v)}^-(\lambda) := E_{\lambda}(t_{\lambda}^-(u, v)(u, v))$ is $C^{1}(I, \mathbb{R})$ and decreasing.
		\end{itemize}
	\end{lem}
	\begin{proof}
		Define the function $F: (0, \infty)\times(0, \infty)\times X \to \mathbb{R}$ in the following form 
		$$F(\lambda, t, (u, v)) = \gamma_{\lambda}'(t) = E_{\lambda}'(t(u, v))(u, v) = t A(u, v) - \lambda t^{-p}P(u) - \lambda t^{-q}Q(v) - \theta t^{\alpha + \beta - 1}B(u, v)$$
		where $\lambda_i \in I$. Recall also that $t_{_{\lambda_i}}^+(u, v)$ is well-defined and $t_{_{\lambda_i}}^+(u, v)(u, v) \in \mathcal{N}_{\lambda_i}^+, i = 1,2$. Moreover,
		$$\frac{\partial F}{\partial t}(\lambda_i, t_{_{\lambda_i}}^+(u, v), (u, v)) = \gamma_{\lambda_i}''(t_{_{\lambda_i}}^+(u, v)) > 0 \,\,\mbox{and} \,\, F(\lambda_i, t_{_{\lambda_i}}^+(u, v), (u, v)) = \gamma_{\lambda_i}'(t_{_{\lambda_i}}^+(u, , v)) = 0.$$
		It follows from the Implicit Function Theorem (\cite[Theorem 4.2.1]{Dra}) that there exists a unique $\varphi(\lambda)$ such that $\varphi(\lambda) = t_{_{\lambda}}^+(u, v)$, $\lambda \in (\lambda_i - \varepsilon, \lambda_i + \varepsilon)$ e $\varphi \in C^{\infty}((\lambda_i - \varepsilon, \lambda_i + \varepsilon), \mathbb{R})$ for some $\varepsilon > 0$ and for each $\lambda$. Since $I\subset\mathbb{R}$ is an open and $\lambda_i$ is arbitrary, we conclude that $\varphi \in C^{1}(I, \mathbb{R})$. Furthermore, we mention that
		$$\frac{\partial \varphi}{\partial \lambda}(\lambda) 
		= \frac{\partial t_{_{\lambda}}^+}{\partial \lambda}(u, v) = - \frac{\frac{\partial F}{\partial \lambda}(\lambda, t_{_{\lambda}}^+(u, v), (u, v))}{\frac{\partial F}{\partial t}(\lambda, t_{_{\lambda}}^+(u, v), (u, v))}.$$
		As a consequence, we obtain that
		$$\frac{\partial t_{_{\lambda}}^+}{\partial \lambda}(u, v) =  \frac{(t_{_{\lambda}}^+(u, v))^{-p} P(u) + (t_{_{\lambda}}^+(u, v))^{-q} Q(v)}{\gamma_{\lambda}''(t_{_{\lambda}}^+(u, v))} > 0.$$
		Hence, the function $\lambda \mapsto t_{\lambda}^+(u, v)$ is increasing. Similarly, we have that $t_{\lambda}^-(u, v)$ is decreasing and $C^{1}(I, \mathbb{R})$. 
		
		Now, we shall prove that $\lambda \mapsto E_{(u, v)}^+(\lambda)$ is decreasing. More specifically, we analyze the derivative with respect to $\lambda$. Notice also that
		$E_{(u, v)}^+(\lambda) = E_{\lambda}(t_{_{\lambda}}^+(u, v)(u, v))$ with $(u, v) \in X.$ In particular, we infer that
		\begin{eqnarray*}
			E_{(u, v)}^+(\lambda) & = & \frac{(t_{_{\lambda}}^+(u, v))^2}{2}A(u, v) - \frac{\lambda}{1 - p}(t_{_{\lambda}}^+(u, v))^{1 - p}P(u) - \frac{\lambda}{1 - q}(t_{_{\lambda}}^+(u, v))^{1 - q}Q(v) - \frac{\theta}{\alpha + \beta}(t_{_{\lambda}}^+(u, v))^{\alpha + \beta}B(u, v).
		\end{eqnarray*}
		Hence, we obtain that
		\begin{eqnarray*}
			\frac{d E_{(u, v)}^+}{d\lambda}(\lambda) & = &t_{_{\lambda}}^+(u, v)\frac{\partial}{\partial \lambda}t_{\lambda}^+(u, v)A(u, v) - \lambda t_{\lambda}^+(u, v)^{-p}\frac{\partial}{\partial \lambda}t_{\lambda}^+(u, v)P(u) + \nonumber \\ 
			&&-\lambda t_{\lambda}^+(u, v)^{-q} \frac{\partial}{\partial \lambda}t_{\lambda}^+(u, v) Q(v) - \theta t_{\lambda}^+(u, v)^{\alpha + \beta - 1} \frac{\partial}{\partial \lambda}t_{\lambda}^+(u, v) B(u, v) + \nonumber \\
			&&- \left(\frac{(t_{_{\lambda}}^+(u, v))^{1 - p}}{1 - p}P(u) + \frac{(t_{_{\lambda}}^+(u, v))^{1 - q}}{1 - q}Q(v) \right) .
		\end{eqnarray*}
		As a product, we see that
		\begin{eqnarray*}
			\frac{d E_{(u, v)}^+}{d\lambda}(\lambda) & = &\frac{1}{t_{\lambda}^+(u, v)}\frac{\partial}{\partial \lambda}t_{\lambda}^+(u, v) \left (A(t_{\lambda}^+(u, v)(u, v)) - \lambda P(t_{\lambda}^+(u, v)u) + \right. \nonumber \\ 
			&&\left.-\lambda Q(t_{\lambda}^+(u, v)v) - \theta B(t_{\lambda}^+(u, v)(u, v)\right) - \left(\frac{(t_{_{\lambda}}^+(u, v))^{1 - p}}{1 - p}P(u) + \frac{(t_{_{\lambda}}^+(u, v))^{1 - q}}{1 - q}Q(v) \right).  
		\end{eqnarray*}
		Notice also that $t_{\lambda}^+(u, v)(u, v) \in \mathcal{N}_{\lambda}^+$. It follows also that
		$E_{\lambda}'(t_{\lambda}^+(u, v)(u, v))(t_{\lambda}^+(u, v)(u, v)) = 0.$
		Therefore, we infer that 
		\begin{eqnarray*}
			\frac{d E_{(u, v)}^+}{d\lambda}(\lambda) & = &- \left(\frac{(t_{_{\lambda}}^+(u, v))^{1 - p}}{1 - p}P(u) + \frac{(t_{_{\lambda}}^+(u, v))^{1 - q}}{1 - q}Q(v) \right) < 0.  
		\end{eqnarray*}
		As a consequence, $E_{(u, v)}^+(\lambda)$ is decreasing with respect to the parameter $\lambda$. Similarly, we have that $E_{(u, v)}^-(\lambda)$ is decreasing with respect to the parameter $\lambda$. Recall also $t_{\lambda}^\pm(u, v) > 0$ for $\lambda \in I$. Hence, we obtain that
		$$E_{(u, v)}^+(\lambda) = E_{_{\lambda}}(t_{_{\lambda}}^+(u, v)(u, v)) \;\;\;\; \mbox{and} \;\;\;\; E_{(u, v)}^- (\lambda) = E_{_{\lambda}}(t_{_{\lambda}}^-(u, v)(u, v)).$$ 
		As a consequence, we have that $\lambda \mapsto E_{(u, v)}^+$ and $\lambda \mapsto E_{(u, v)}^-$ are in $C^{1}(I, \mathbb{R}).$ This ends the proof. 
	\end{proof}
	
	In order to ensure that there exist two weak solutions to the System $(S_{\lambda^*})$ we use Proposition \ref{limCn-}. Let $\lambda \in (0, \lambda^*)$ be fixed. Consider $(u_{\lambda}, v_{\lambda}) \in \mathcal{N}_{\lambda}^+$ and $(z_{\lambda}, w_{\lambda}) \in \mathcal{N}_{\lambda}^-$ as weak solutions for the System $(S_{\lambda})$, see for instance Proposition \ref{solucaosing}. For simplicity, we shall use $(u_k, v_k)$ instead of $(u_{\lambda_k}, v_{\lambda_k})$. 
	
	\begin{prop}\label{limCn-}
		Suppose ($P_0$), ($P_1$), ($P$), ($V_0$), and ($V_1'$). Let $\tilde{\lambda} \in (0, \lambda^*]$. Suppose $(\lambda_k) \subset (0, \lambda^*)$ such that $\lambda_k \to \tilde{\lambda}$. Then, we obtain the following statements:\\
		\noindent a) It holds that the functional $\lambda \mapsto C_{\mathcal{N}_{\lambda}^\pm}$ is decreasing for $0 < \lambda \leq \lambda^*$.\\
		\noindent  b) The functionals $\lambda\! \mapsto\! (u_{\lambda}, v_{\lambda})$ and $\lambda \!\mapsto \!(z_{\lambda}, w_{\lambda})$ are continuous;\\ 
		\noindent  c) It holds that the functional $\lambda \mapsto C_{\mathcal{N}_{\lambda}^\pm}$ is left-continuous for $0 < \lambda < \lambda^*$.\\
		\noindent  d) There holds $\lim\limits_{\lambda \to \lambda^*} C_{\mathcal{N}_{\lambda}^\pm} = C_{\mathcal{N}_{\lambda^*}^\pm}$.    
	\end{prop}
	\begin{proof}
		\textbf{a)} Firstly, we shall prove that $\lambda \mapsto C_{\mathcal{N}_{\lambda}^-}$ is decreasing. Now, for any
		$\lambda_1 < \lambda_2$, we infer that $E_{\lambda_2}(tu, tv) < E_{\lambda_1}(tu, tv)$ and $E_{\lambda_2}'(tu, tv)(u, v) < E_{\lambda_1}'(tu, tv)(u, v)$ for all $t > 0$.
		Consider $(u_i,v_i)$ in such way that
		$$C_{\mathcal{N}_{\lambda_1}^-} =  E_{\lambda_1}(t_{\lambda_1}^-(u_1, v_1)(u_1, v_1))\;\;\;\; \mbox{and}\;\;\;\;C_{\mathcal{N}_{\lambda_2}^-} =  E_{\lambda_2}(t_{\lambda_2}^-(u_2, v_2)(u_2, v_2)).$$
		Recall also that 
		$E_{\lambda_2}'(t(u_1, v_1))(u_1, v_1) < E_{\lambda_1}'(t(u_1,v_1))(u_1,v_1) < 0$ for each  $0 < t < t_{\lambda_1}^+(u_1,v_1)$ or $t > t_{\lambda_1}^-(u_1, v_1)$.
		Under these conditions, we infer that
		\begin{equation}\label{reltlamb}
			t_{\lambda_1}^+(u_1, v_1) < t_{\lambda_2}^+(u_1, v_1) < t_{\lambda_2}^-(u_1, v_1) < t_{\lambda_1}^-(u_1, v_1).
		\end{equation}
		Thus, we see that 
		\begin{eqnarray*}
			C_{\mathcal{N}_{\lambda_2}^-} = E_{\lambda_2}(t_{\lambda_2}^-(u_2, v_2)(u_2, v_2))& \le & E_{\lambda_2}(t_{\lambda_2}^-(u_1,v_1)(u_1,v_1)) <  E_{\lambda_1}(t_{\lambda_2}^-(u_1,v_1)(u_1,v_1))
			< E_{\lambda_1}(t_{\lambda_1}^-(u_1,v_1)(u_1,v_1)) = C_{\mathcal{N}_{\lambda_1}^-}.
		\end{eqnarray*}
		Here was used the fact that the fibering map $t \mapsto E_{\lambda_1}(tu_1, tv_1)$ is increasing over the interval $[t_{\lambda_1}^+(u_1, v_1), t_{\lambda_1}^-(u_1, v_1)]$. 
		
		Similarly, we deduce also that $\lambda \mapsto C_{\mathcal{N}_{\lambda}^+}$ is decreasing. In fact, the function $t \mapsto E_{\lambda_2}(tu_1,tv_1)$ is decreasing in the set $t \in (0, t_{\lambda_2}^+(u_1, v_1))$ and taking into account \eqref{reltlamb},  we deduce that
		\begin{eqnarray*}
			C_{\mathcal{N}_{\lambda_1}^+}  =  E_{\lambda_1}(t_{\lambda_1}^+(u_1, v_1)(u_1, v_1)) > E_{\lambda_2}(t_{\lambda_1}^+(u_1, v_1)(u_1, v_1)) >  E_{\lambda_2}(t_{\lambda_2}^+(u_1, v_1)(u_1, v_1)) >E_{\lambda_2}(t_{\lambda_2}^+(u_2, v_2)(u_2, v_2)) = C_{\mathcal{N}_{\lambda_2}^+}.
		\end{eqnarray*}
		
		Now, we shall prove the item \textbf{b)}. Define \( E:(0, +\infty) \times X \to \mathbb{R} \) given by \( E(\lambda, (u, v)) := E_{\lambda}(u, v) \). It is important to mention that \( E \) is continuous in the variable \( (u, v) \). Now, we shall verify that $\lambda \mapsto E_\lambda(u,v)$ is continuous. Consider a sequence $(\lambda_k)$ such that \( \lambda_k \to \lambda \). Hence, we are able to prove the following assertion:
		\begin{equation}\label{Econtinua}
			\vert E(\lambda_k, (u, v)) - E(\lambda, (u, v)) \vert \le \left|\lambda_k - \lambda\right| \left| \frac{1}{1 - p}P(u) + \frac{1}{1 - q}Q(v)\right| \to 0.  
		\end{equation}
		As a consequence, the functional $\lambda \mapsto E(\lambda, u, v)$ is continuous. Recall also that $(u_{\lambda_k}, v_{\lambda_k}) \in \mathcal{N}_{\lambda_k}^+$ and $(z_{\lambda_k}, w_{\lambda_k}) \in \mathcal{N}_{\lambda_k}^-$. Moreover, the functional $E_{\lambda_k}$ is coercive in the Nehari set. Therefore, $(u_{\lambda_k}, v_{\lambda_k})$ and $(z_{\lambda_k}, w_{\lambda_k})$ are bounded, see Proposition \ref{E coersivasing}. Furthermore, we observe that
		$\lambda \mapsto t_{\lambda}^+(u, v)\;\;\; \mbox{and}\;\;\; (\lambda, (u, v)) \mapsto E(\lambda, (u, v))$ are continuous functions.
		Notice also that $t_{\lambda_k}^+(u_{\tilde{\lambda}}, v_{\tilde{\lambda}})(u_{\tilde{\lambda}}, v_{\tilde{\lambda}}) \in \mathcal{N}_{\lambda_k}^+$. Hence, $C_{\mathcal{N}_{\lambda_k}^+} \le E_{\lambda_k}(t_{\lambda_k}^+(u_{\tilde{\lambda}}, v_{\tilde{\lambda}})(u_{\tilde{\lambda}} v_{\tilde{\lambda}})).$ Moreover, we obtain that 
		$$\limsup_{\lambda_k \to \tilde{\lambda}^-}C_{\mathcal{N}_{\lambda_k}^+} \le \limsup_{\lambda_k \to \tilde{\lambda}^-}E_{\lambda_k}(t_{\lambda_k}^+(u_{\tilde{\lambda}}, v_{\tilde{\lambda}})(u_{\tilde{\lambda}}, v_{\tilde{\lambda}})) = E_{\tilde{\lambda}}(t_{\tilde{\lambda}}^+(u_{\tilde{\lambda}}, v_{\tilde{\lambda}})(u_{\tilde{\lambda}}, v_{\tilde{\lambda}})).$$
		As a byproduct, we obtain that 
		\begin{equation}\label{cn+limtd}
			\limsup_{\lambda_k \to \tilde{\lambda}^-}C_{\mathcal{N}_{\lambda_k}^+} \le C_{\mathcal{N}_{\tilde{\lambda}}^+}.
		\end{equation}
		Now, by using Proposition \ref{E coersivasing}, we mention also that
		\begin{equation}\label{coerciva2}
			C_{\mathcal{N}_{\lambda_k}^+} = E_{_{\lambda_k}}(u_k, v_k) \ge  C_1 \Vert(u_k, v_k)\Vert^2 - C_2 \Vert (u_k, v_k) \Vert^{1 - p} - C_3\Vert (u_k, v_k) \Vert^{1 - q}.
		\end{equation}
		Therefore, $(u_k, v_k)$ is bounded in $X$. Similarly, $(z_k, w_k)$ is also bounded. In this way, $(u_k, v_k) \rightharpoonup (u_{\tilde{\lambda}}, u_{\tilde{\lambda}})$ for some
		$(u_{\tilde{\lambda}}, u_{\tilde{\lambda}}) \in X.$ It remains to prove that $(u_k, v_k) \to (u_{\tilde{\lambda}}, u_{\tilde{\lambda}})$ in $X$.
		Notice also that $(u_{\lambda_k}, v_{\lambda_k})$ is a weak solution for Systems \eqref{equfracasingula} with $\lambda = \lambda_k$.
		Analogously, by using the same ideas discussed in the proof of Proposition \ref{solucaosing}, we infer that
		\begin{eqnarray}\label{pontcriffunen}
			0 \! =\! \left<(u_{k}, v_{k}), (\varphi_1, \varphi_2)\right> \!-\!  \frac{\theta}{\alpha + \beta}\int \!\alpha\vert u_{k}\vert^{\alpha - 2}u_{k}\varphi_1\vert v_{k}\vert^{\beta}-\beta\vert u_{k}\vert^{\alpha}\vert v_{k}\vert^{\beta-2}v_{k}\varphi_2\! - \!\lambda_k \int\! a(x)\vert u_{k}\vert^{-p}\varphi_1- b(x)\vert v_{k}\vert^{-q}\varphi_2.
		\end{eqnarray}
		In particular, choosing $(\psi_1, \psi_2) = (u_{\lambda_k} - u_{\tilde{\lambda}}, v_{\lambda_k} - v_{\tilde{\lambda}})$, we see that
		\begin{eqnarray}\label{prodconvfortuv}
			\left<(u_{k}, v_{k}), (u_{k} - u_{\tilde{\lambda}}, v_{k} - v_{\tilde{\lambda}})\right> &=&  \frac{\theta}{\alpha + \beta}\int \alpha\vert u_{k}\vert^{\alpha - 2}u_{k}(u_{k} - u_{\tilde{\lambda}})\vert v_{k}\vert^{\beta}+\beta \vert u_{k}\vert^{\alpha}\vert v_{k}\vert^{\beta-2}v_{k}(v_{k} - v_{\tilde{\lambda}}) \nonumber\\ 
			&&   + \lambda_k \int a(x)\vert u_{k}\vert^{-p}(u_{k} - u_{\tilde{\lambda}})+ b(x)\vert v_{k}\vert^{-q}(v_{k} - v_{\tilde{\lambda}}).
		\end{eqnarray}
		Now, we shall analyze each term on the right-hand side of \eqref{prodconvfortuv}. In order to do that we mention 
		\begin{eqnarray*}
			\left | \int  \vert u_{k}\vert^{\alpha - 2}u_{k}(u_{k} - u_{\tilde{\lambda}})\vert v_{k}\vert^{\beta}dx\right | 
			&\le& \int  \vert u_{k}\vert^{\alpha - 2}\vert u_{k}\vert \vert(u_{k} - u_{\tilde{\lambda}})\vert \vert v_{k}\vert^{\beta}dx \leq \int  \vert u_{k}\vert^{\alpha - 1}\vert(u_{k} - u_{\tilde{\lambda}})\vert \vert v_{k}\vert^{\beta}dx.
		\end{eqnarray*}
		Hence, by applying Hölder's inequality with exponents $r_1 = \frac{\alpha + \beta}{\alpha - 1}$, $r_2 = \alpha + \beta$ and $r_3 = \frac{\alpha + \beta}{\beta}$, we ensure that 
		\begin{eqnarray*}
			\left|\int  \vert u_{k}\vert^{\alpha - 2}u_{k}(u_{k} - u_{\tilde{\lambda}})\vert v_{k}\vert^{\beta}dx\right| & \le & \Vert u_{k}\Vert^{\alpha - 1}_{\alpha + \beta}\Vert u_{k} - u_{\tilde{\lambda}}\Vert_{\alpha + \beta} \Vert v_{k}\Vert^{\beta}_{\alpha + \beta}.
		\end{eqnarray*}
		In view of hypothesis ($P$) we have that $2 < \alpha + \beta < 2^*_s$. Now, by using Lemma \ref{imersoes Sobolev 2, 2*s}, we have that $X\hookrightarrow\hookrightarrow L^{\alpha + \beta}(\mathbb{R}^N)$. Therefore, $\Vert u_{k} - u_{\tilde{\lambda}}\Vert_{\alpha + \beta} \to 0$ as $k \to \infty$.
		As a consequence, we obtain that
		\begin{equation}\label{ineqIuv}
			\limsup\limits_{k \to \infty} \left |\int  \vert u_{k}\vert^{\alpha - 2}u_{k}(u_{k} - u_{\tilde{\lambda}})\vert v_{k}\vert^{\beta}dx\right| = 0. 
		\end{equation}
		Similarly, we also mention that
		\begin{equation}\label{ineqIIuv}
			\limsup\limits_{k \to \infty} \left |\int  \vert u_{k}\vert^{\alpha}\vert v_{k}\vert^{\beta - 2}v_{k}(v_{k} - v_{\tilde{\lambda}})dx \right| = 0.
		\end{equation}
		Under these conditions, we estimate \eqref{prodconvfortuv} as follows:
		\begin{eqnarray*}
			\limsup\limits_{k \to \infty}\left<(u_{k}, v_{k}), (\Gamma_k(u), \Gamma_k(v))\right> & \le & \limsup\limits_{k \to \infty}\left[\lambda_k \int a(x)\vert u_{k}\vert^{-p}(\Gamma_k(u))  + b(x)\vert v_{k}\vert^{-q}(\Gamma_k(v)) dx\right],
		\end{eqnarray*}
		where $ \Gamma_k(u) = u_{k} - u_{\tilde{\lambda}}, \;\Gamma_k(v) = v_{k} - v_{\tilde{\lambda}}$. According to Proposition \ref{N-longe zerosing}, $u_{k}, v_{k} > 0$ in $\mathbb{R}^N$. Therefore, we see that 
		\begin{eqnarray*}
			\limsup\limits_{k \to \infty}\left<(u_{k}, v_{k}), (u_{k} - v_{\tilde{\lambda}}, v_{k} - v_{\tilde{\lambda}})\right> & \le & \limsup\limits_{k \to \infty}\lambda_k \int a(x)\vert u_{k}\vert^{1-p}dx  + \limsup\limits_{k \to \infty}\lambda_k \int  - a(x)\vert u_{k}\vert^{-p}u_{\tilde{\lambda}}dx \nonumber \\
			& & +\limsup\limits_{k \to \infty}\lambda_k \int b(x)\vert v_{k}\vert^{1 - q}dx+ \limsup\limits_{k \to \infty}\lambda_k \int  - b(x)\vert v_{k}\vert^{- q}v_{\tilde{\lambda}}dx.
		\end{eqnarray*}
		Furthermore, we obtain
		\begin{eqnarray*}
			\limsup\limits_{k \to \infty}\left<(u_{k}, v_{k}), (u_{k} - u_{\tilde{\lambda}}, v_{k} - v_{\tilde{\lambda}})\right> & \le & \limsup\limits_{k \to \infty}\lambda_k \int a(x)\vert u_{k}\vert^{1-p}dx -\liminf\limits_{k \to \infty}\lambda_k \int a(x)\vert u_{k}\vert^{-p}u_{\tilde{\lambda}}dx \nonumber \\
			&  & + \limsup\limits_{k \to \infty}\lambda_k \int b(x)\vert v_{k}\vert^{1 - q}dx  - \liminf\limits_{k \to \infty}\lambda_k \int b(x)\vert v_{k}\vert^{- q}v_{\tilde{\lambda}}dx.
		\end{eqnarray*}
		It follows from Fatou's Lemma and the last estimates that
		\begin{eqnarray*}
			\limsup\limits_{k \to \infty}\left<(u_{k}, v_{k}), (u_{k} - u_{\tilde{\lambda}}, v_{k} - v_{\tilde{\lambda}})\right> & \le & \tilde{\lambda} \int a(x)\vert u_{\tilde{\lambda}}\vert^{1-p}dx - \int \liminf\limits_{k \to \infty} \lambda_k 
			a(x)\vert u_{k}\vert^{-p}u_{\tilde{\lambda}}dx  \nonumber \\
			& & + \tilde{\lambda} \int b(x)\vert v_{\tilde{\lambda}}\vert^{1 - q}dx  - \int \liminf\limits_{k \to \infty} 
			\lambda_k b(x)\vert v_{k}\vert^{- q}v_{\tilde{\lambda}}dx .
		\end{eqnarray*}
		Therefore, we deduce that
		\begin{equation}\label{ineqIIIuv}
			\limsup\limits_{k \to \infty}\left<(u_{k}, v_{k}), (u_{k} - u_{\tilde{\lambda}}, v_{k} - v_{\tilde{\lambda}})\right>  \le 0. 
		\end{equation}
		Now, taking into account \eqref{ineqIuv}, \eqref{ineqIIuv} and \eqref{ineqIIIuv}, we infer that
		\begin{eqnarray*}
			\limsup\limits_{k \to \infty} \Vert(\Gamma_k(u), \Gamma_k(v))\Vert^2 & = & \limsup\limits_{k \to \infty}\left<(\Gamma_k(u), \Gamma_k(v)) , (\Gamma_k(u), \Gamma_k(v))\right> \nonumber \\
			& \le & \limsup\limits_{k \to \infty}\left<(u_{k}, v_{k}), (u_{k} - u_{\tilde{\lambda}}, v_{k} - v_{\tilde{\lambda}})\right> + \limsup\limits_{k \to \infty} \left(- \left<(u_{\tilde{\lambda}} , v_{\tilde{\lambda}}), (u_{k} - u_{\tilde{\lambda}}, v_{k} - v_{\tilde{\lambda}})\right>\right)\nonumber \\   &=&\limsup\limits_{k \to \infty}\left<(u_{k}, v_{k}), (u_{k} - u_{\tilde{\lambda}}, v_{k} - v_{\tilde{\lambda}})\right>  - \liminf\limits_{k \to \infty} \left<( u_{\tilde{\lambda}} , v_{\tilde{\lambda}}), (u_{k} - u_{\tilde{\lambda}}, v_{k} - v_{\tilde{\lambda}})\right>.
		\end{eqnarray*}
		On the other hand, by using that $(u_{k}, v_{k}) \rightharpoonup (u_{\tilde{\lambda}}, v_{\tilde{\lambda}})$
		and  \eqref{ineqIIIuv}, we see that
		$\limsup\limits_{k \to \infty} \Vert(u_{k} - u_{\tilde{\lambda}}, v_{k} - v_{\tilde{\lambda}})\Vert = 0.$
		As a consequence, $(u_{k},v_{k}) \to (u_{\tilde{\lambda}}, v_{\tilde{\lambda}})$ in $X$. The same argument can be applied for the sequence $(z_k, w_k)$. Therefore, $\lambda \mapsto (u_\lambda, v_\lambda)$ and $\lambda \mapsto (z_\lambda, w_\lambda)$ are continuous functions.
		
		Now, we shall prove the item \textbf{c)}. Here we need to show that $(u_{\tilde{\lambda}}, v_{\tilde{\lambda}}) \in \mathcal{N}_{\tilde{\lambda} }^+$. Recall that the functional $(\lambda, (u, v)) \mapsto E(\lambda, (u, v))$ is continuous. Then, for each sequence $\lambda_k \to \tilde{\lambda}$ and $(u_k, v_k) \to (u_{\tilde{\lambda}}, v_{\tilde{\lambda}})$, we deduce that
		$$\lim\limits_{k \to \infty}E(\lambda_k, (u_k, v_k)) = E(\tilde{\lambda}, (u_{\tilde{\lambda}}, v_{\tilde{\lambda}})) = E_{\tilde{\lambda}}(u_{\tilde{\lambda}}, v_{\tilde{\lambda}}).$$
		Now, by using \eqref{cn+limtd}, we infer that
		$$E_{\tilde{\lambda}}(u_{\tilde{\lambda}}, v_{\tilde{\lambda}}) = \lim\limits_{k \to \infty} E(\lambda_k, (u_k, v_k)) = \lim\limits_{k \to \infty} C_{\mathcal{N}_{\lambda_k}^+} \le C_{\mathcal{N}_{\tilde{\lambda}}^+} < 0.$$
		Hence, $(u_{\tilde{\lambda}}, v_{\tilde{\lambda}}) \ne (0, 0)$. Consider also
		$\lambda_k \to \tilde{\lambda} \in (0, \lambda^*)$, $(u_k, v_k) \to (u_{\tilde{\lambda}}, v_{\tilde{\lambda}})$ in $X$ and $(u_k, v_k)\in \mathcal{N}_{\lambda_k}^+$, $u_{\tilde{\lambda}} \ne 0$ and $v_{\tilde{\lambda}} \ne 0$ in such way that
		$$E(\lambda_k, (u_k, v_k)) = E_{\lambda_k}(u_k, v_k) = C_{\mathcal{N}_{\lambda_k}^+},\;\;\;\;\left\{\begin{array}{l}E_{\lambda_k}'(u_k, v_k)(u_k, v_k) = 0,\;\;\; E_{\tilde{\lambda}}'(u_{\tilde{\lambda}}, v_{\tilde{\lambda}})(u_{\tilde{\lambda}}, v_{\tilde{\lambda}}) = 0,\\E_{\lambda_k}''(u_k, v_k)(u_k, v_k)^2 > 0,\;\;\; E_{\tilde{\lambda}}''(u_{\tilde{\lambda}}, v_{\tilde{\lambda}})(u_{\tilde{\lambda}}, v_{\tilde{\lambda}})^2 \ge 0.\end{array}\right.$$
		Under these conditions, $(u_{\tilde{\lambda}}, v_{\tilde{\lambda}}) \in \mathcal{N}_{\lambda}^0 \cup \mathcal{N}_{\lambda}^+$. Now, for each $\tilde{\lambda} \in (0, \lambda^*)$, we have that $\mathcal{N}_{\tilde{\lambda}} = \emptyset$. Moreover, for $\tilde{\lambda} = \lambda^*$, by using Proposition \ref{incompatibilidade} together with \eqref{pontcriffunen}, we conclude that $(u_{\tilde{\lambda}}, v_{\tilde{\lambda}}) \notin \mathcal{N}_{\lambda^*}^0$. Hence, $(u_{\tilde{\lambda}}, v_{\tilde{\lambda}}) \in \mathcal{N}_{\tilde{\lambda}}^+$. 
		Summarizing, by using \eqref{cn+limtd}, we obtain that $\limsup\limits_{\lambda_k \to \tilde{\lambda}} C_{\mathcal{N}_{\lambda_k}^+} = C_{\mathcal{N}_{\tilde{\lambda}}^+}$. Here was used the fact that
		$$C_{\mathcal{N}_{\tilde{\lambda}}^+} \le E(\tilde{\lambda}, (u_{\tilde{\lambda}}, v_{\tilde{\lambda}})) = \limsup\limits_{\lambda_k \to \tilde{\lambda}} E(\lambda_k, (u_k, v_k)) = \limsup\limits_{\lambda_k \to \tilde{\lambda}} C_{\mathcal{N}_{\lambda_k}^+} \le C_{\mathcal{N}_{\tilde{\lambda}}^+}.$$
		
		Now, we shall prove item \textbf{d)}. In order to do that we consider a sequence $\lambda_k \to (\lambda^*)^-$. Therefore, $\lim\limits_{\lambda_k \to (\lambda^*)^-}C_{\mathcal{N}_{\lambda_k}^+} = C_{\mathcal{N}_{\lambda^*}^+}.$ Similarly, we deduce that $\lim\limits_{k \to \infty}C_{\mathcal{N}_{\lambda_k}^-} = C_{\mathcal{N}_{\lambda^*}^-}.$
		This finishes the proof.
	\end{proof}
	
	\begin{prop}\label{exist2solu}
		Assume $(P_0), (P_1), (P), (V_0)$ and $(V_1')$.Then, System $\!(S_{\lambda^*})\!$ has at least two solutions $\!(z_{*}, w_{*})\! \in\! \mathcal{N}_{\lambda^*}^-\!$ and $(u_{*}, v_{*})\! \in \!\mathcal{N}_{\lambda^*}^+$.
	\end{prop}
	\begin{proof}
		Firstly, we shall prove that there exists a solution \((z_{*}, w_{*}) \in \mathcal{N}_{\lambda^*}^-\) for System \((S_{\lambda^*})\). Define a sequence \(\{\lambda_k\} \subset (0, \lambda^*)\) such that \(\lambda_k \to (\lambda^*)^-\) as \(k \to \infty\). For each \(k \in \mathbb{N}\), by using Propositions \ref{elemN+N-inf} and \ref{solucaosing}, there exists a solution \((z_{k}, w_{k}) \in \mathcal{N}_{\lambda_k}^-\) for the System \((S_{\lambda_k})\). Now, we claim that \((z_{k}, w_{k})\) is bounded. Otherwise, up to a subsequence, \(\Vert(z_{k}, w_{k})\Vert \to \infty\) as \(k \to \infty\). Under these conditions, we obtain that 
		\begin{equation}\label{Elambdaksing}
			E_{\lambda_k}(z_{k}, w_{k}) = \frac{1}{2}A(z_{k}, w_{k}) - \frac{\lambda_k}{1 - p} P(z_{k}) - \frac{\lambda_k}{1 - q} Q(w_{k}) - \frac{\theta}{\alpha + \beta}B(z_{k}, w_{k}). 
		\end{equation}
		Now, by using \( E_{\lambda_k}'(z_k, w_k)(z_k, w_k) = 0 \) and \eqref{Elambdaksing}, we rewrite the last identity as follows:
		\begin{eqnarray*}
			E_{\lambda_k}(z_{k}, w_{k})&=& \left(\frac{1}{2} - \frac{1}{\alpha + \beta}\right)A(z_{k}, w_{k}) - \lambda_k\left(\frac{1}{1 - p} - \frac{1}{\alpha + \beta} \right)P(z_{k})  -\lambda_k\left(\frac{1}{1 - q} - \frac{1}{\alpha + \beta} \right) Q(w_{k}).
		\end{eqnarray*}
		Hence, by using the Hölder inequality, the Sobolev embedding and the ideas employed in \eqref{coerciva}, we deduce that 
		$$C_{\mathcal{N}_{\lambda_k}^-} = E_{\lambda_k}(z_{k}, w_{k}) \ge  C_1\Vert(z_{k}, w_{k})\Vert^2 - C_2\Vert (z_{k},w_{k})\Vert^{1 - p} - C_3 \Vert(z_{k},w_{k})\Vert^{1 - q}.$$
		It follows from Lemma \ref{Cn+e Cn-decresc} and Proposition \ref{limCn-} that
		$\lambda\mapsto C_{\mathcal{N}_{\lambda}^-}$ is bounded and decreasing. Notice also that the functional \( E_{\lambda_k} \) is coercive. Therefore,  $(z_{k}, w_{k})$ is bounded in $X$. Furthermore, $C_{\mathcal{N}_{\lambda^*}^-}:= \lim\limits_{\lambda_k \uparrow \lambda^*}C_{\mathcal{N}_{\lambda_k}^-}.$
		In this way, there exists \((z_{*}, w_{*})\) in $X$ such that \((z_{k}, w_{k}) \rightharpoonup (z_{*}, w_{*})\) in $X$.

		At this stage, we shall guarantee that $z_{*}, w_{*}$ are positive functions. To begin with, by using Proposition \ref{solucaosing}, for any \((\varphi_1, \varphi_2) \in X\), we infer that
		\begin{eqnarray*}
			\left<(z_{k}, w_{k}), (\varphi_1, \varphi_2)\right> -  \frac{\theta}{\alpha + \beta}\int  \alpha\vert z_{k}\vert^{\alpha - 2}z_{k}\varphi_1\vert w_{k}\vert^{\beta} -\beta\vert z_{k}\vert^{\alpha}\vert w_{k}\vert^{\beta-2}w_{k}\varphi_2dx\;=\; \lambda_k \int  a(x)\vert z_{k}\vert^{-p}\varphi_1+ b(x)\vert w_{k}\vert^{-q}\varphi_2dx.&&\nonumber
		\end{eqnarray*}
		In particular, taking \(\varphi_1, \varphi_2 \ge 0\) and applying the Fatou's Lemma, we obtain that
		\begin{eqnarray*}
			\lambda^* \left(\int  a(x)\vert z_{*}\vert^{-p}\varphi_1dx + \int  b(x)\vert w_{*}\vert^{-q}\varphi_2dx\right)\leq\liminf\limits_{k \to \infty}\lambda_k \int  a(x)\vert z_{k}\vert^{-p}\varphi_1dx + \liminf\limits_{k \to \infty}\lambda_k \int  b(x)\vert w_k\vert^{-q}\varphi_2dx&& \\
			\le\liminf\limits_{k \to \infty}\left(\left<(z_{k}, w_{k}), (\varphi_1, \varphi_2)\right> - \theta \frac{\alpha}{\alpha + \beta}\int  \vert z_{k}\vert^{\alpha - 2}z_{k}\varphi_1\vert w_{k}\vert^{\beta}dx- \theta\frac{\beta}{\alpha + \beta}\int  \vert z_{k}\vert^{\alpha}\vert w_{k}\vert^{\beta-2}w_{k}\varphi_2dx \right) < \infty.&&
		\end{eqnarray*}
		Define the following auxiliary functions
		\begin{equation}
			G_1(x) =
			\left \{
			\begin{array}{cc}
				z_{*}^{-p}(x), & \mbox{if} \;\;\; z_{*}(x) \ne 0; \\
				\infty, & \mbox{if} \;\;\; z_{*}(x) = 0.
			\end{array}
			\right.
			\;\;\;\;\;\;\;\;\;\;
			G_2(x) =
			\left \{
			\begin{array}{cc}
				w_{*}^{-q}(x), & \mbox{if} \;\;\; w_{*}(x) \ne 0; \\
				\infty, & \mbox{if} \;\;\; w_{*}(x) = 0.
			\end{array}
			\right.
		\end{equation}
		Hence, 
		$0 < \lambda^* \int  a(x)\vert z_{*}\vert^{-p}\varphi_1dx + \lambda^* \int  b(x)\vert w_{*}\vert^{-q}\varphi_2dx < \infty.$
		In particular, $z_{*},w_{*} > 0$ a.e. in $\mathbb{R}^N$ and $(z_{*}, w_{*}) \in \mathcal{A}.$
		
		It remains to prove the strong convergence. The main idea here is to repeat the procedure discussed in the proof of Proposition \ref{limCn-}. In fact, we assume that $\tilde{\lambda} = \lambda^*$ and using as test function $(\phi_1,\phi_2) = (z_{k} - z_{*}, w_{k} - w_{*})$, we obtain
		\begin{equation}\label{pertNeha}
			\gamma_{\lambda^*,(z_{*}, w_{*})}^{'} (1) = \lim\limits_{k \to \infty}\gamma_{\lambda_k,(z_{k}, w_{k})}^{'}(1) = 0;\;\;\;\;\;\;\;\;\;\;  
			\gamma_{\lambda^*,(z_{*}, w_{*})}^{''}(1) = \lim\limits_{k \to \infty}\gamma_{\lambda_k,(z_{k}, w_{k})}^{''}(1) \le 0.  
		\end{equation}
		Thus we have that  $(z_{*}, w_{*}) \in \mathcal{N}_{_{\lambda^*}}^-$ or $( z_{*}, w_{*}) \in \mathcal{N}_{_{\lambda^*}}^0$. Recall also that
		\begin{eqnarray*}
			\left<(z_{k}, w_{k}), (\phi_1, \phi_2)\right> - \theta \frac{\alpha}{\alpha + \beta}\int \vert z_{k} \vert^{\alpha -2}z_{k} \phi_1\vert w_{k}\vert^{\beta}dx && \nonumber\\- \theta\frac{\beta}{\alpha + \beta}\int  \vert z_{k} \vert^{\alpha}\vert w_{k}\vert^{\beta - 2}w_{k}\phi_2dx 
			- \lambda_k\left( \int  a(x)\vert z_{k} \vert^{-p}\phi_1dx +  \int  b(x)\vert w_{k} \vert^{-q}\phi_2dx \right)= 0.&&
		\end{eqnarray*}
		holds for all $(\phi_1, \phi_2) \in X$. In particular, for each $(\phi_1, \phi_2) \in X_+$, we see that
		\begin{eqnarray*}
			\left<(z_{*}, w_{*}), (\phi_1, \phi_2)\right> - \theta \frac{\alpha}{\alpha + \beta} \int \vert z_{*} \vert^{\alpha -2}z_{*} \phi_1\vert w_{*}\vert^{\beta}dx - \theta \frac{\beta}{\alpha + \beta} \int \vert z_{*} \vert^{\alpha}\vert w_{*}\vert^{\beta - 2}w_{*}\phi_2dx  &&\nonumber \\
			= \liminf\limits_{k \to \infty}\left[ \lambda_k\int\left( a(x)\vert z_{k} \vert^{-p}\phi_1 + b(x)\vert w_{k} \vert^{-q}\phi_2\right)dx\right]
			\geq \lambda^*\int\left( a(x)\vert z_* \vert^{-p}\phi_1 + b(x)\vert w_* \vert^{-q}\phi_2\right)dx.
		\end{eqnarray*}
		Hence, we obtain that
		\begin{eqnarray*}
			\left<(z_{*}, w_{*}), (\phi_1, \phi_2)\right> - \theta \frac{\alpha}{\alpha + \beta} \int \vert z_{*} \vert^{\alpha -2}z_{*} \phi_1\vert w_{*}\vert^{\beta}dx - \theta \frac{\beta}{\alpha + \beta} \int \vert z_{*} \vert^{\alpha}\vert w_{*}\vert^{\beta - 2}w_{*}\phi_2dx  &&\nonumber \\
			- \lambda^* \int \left(\; a(x)\vert z_* \vert^{-p}\phi_1+ b(x)\vert w_* \vert^{-q}\phi_2\right)dx \ge 0
		\end{eqnarray*}
		is satisfied for each $(\phi_1, \phi_2) \in X$, $\phi_1, \phi_2 > 0$. Now,  choosing $(\phi_1, \phi_2) = \left((z_{*} + \varepsilon\psi_1)^+, (w_{*} + \varepsilon\psi_2)^+\right)$ , $(\psi_1, \psi_2) \in X$ as a test function, we deduce that
		\begin{eqnarray*}
			\left<(z_{*}, w_{*}), ((z_{*}\! +\! \varepsilon\psi_1)^+, (w_{*}\! +\! \varepsilon\psi_2)^+)\right> - \frac{\theta}{\alpha + \beta} \int \alpha\vert z_{k} \vert^{\alpha -2}z_{*} (z_{*} \!+\! \varepsilon\psi_1)^+\vert w_{*}\vert^{\beta}+\beta\vert z_{*} \vert^{\alpha}\vert w_{k}\vert^{\beta - 2}w_{*}(w_{*} \!+\!\varepsilon\psi_2)^+dx 
			\\
			- \lambda^*\int \left(\;a(x)\vert z_{*} \vert^{-p}(z_{*} \!+ \!\varepsilon\psi_1)^+ +  b(x)\vert w_{*} \vert^{-q}(w_{*}\! + \!\varepsilon\psi_2)^+\;\right)dx \ge 0.
		\end{eqnarray*} 
		Notice also that $(z_{*}, w_{*}) \in \mathcal{N}_{\lambda^*}^- \cup \mathcal{N}_{\lambda^*}^0$. Then, $E''_{\lambda^*}(z_{*}, w_{*})(z_{*}, w_{*})^2 \le 0$. Hence, by using the same estimates as was done in Step 3 of Proposition
		\ref{Neharilambda*}, we prove that \eqref{desimpsing2} is satisfied. Now, by dividing last estimate by $\varepsilon$, we deduce that
		\begin{eqnarray*}
			\left<(z_{*}, w_{*}), (\psi_1, \psi_2)\right> -  \frac{\theta}{\alpha + \beta} \int \left(\alpha\vert z_{*} \vert^{\alpha -2}z_{*} \psi_1\vert w_{*}\vert^{\beta} +\beta\vert z_{*} \vert^{\alpha}\vert w_{*}\vert^{\beta - 2}w_{*}\psi_2\right)dx &&\\
			- \lambda^*\int \left(a(x)\vert z_{*} \vert^{-p}\psi_1 +b(x)\vert w_{*} \vert^{-q}\psi_2\right)dx \geq 0.
		\end{eqnarray*}
		Now, using $(-\psi_1, -\psi_2) \in X$ as testing function, we obtain that $(z_{*}, w_{*})$ is a weak solution for the System ($S_{\lambda^*}$). Furthermore, by using Corollary \ref{incompatibilidade}, does not exist any weak solution of ($S_{\lambda^*}$) in $\mathcal{N}_{\lambda^*}^0$. Hence, any nontrivial solution for the System ($S_{\lambda^*}$) belongs to $\mathcal{N}_{\lambda^*}^-$ or $\mathcal{N}_{\lambda^*}^+$.
		
		At this stage, by using Propositions \ref{convfortN-} and \ref{limCn-}, we see that
		$$E_{\lambda^*}(z_{*}, w_{*}) = \lim\limits_{k \to \infty} E_{\lambda_k}(z_{_k}, w_{k}) = \lim\limits_{k\to \infty}C_{\mathcal{N}_{\lambda_k}^-} = C_{\mathcal{N}_{\lambda^*}^-}.$$
		Note that $(z_{*}, w_{*}) \in \mathcal{N}_{\lambda^*}^-$ is a global minimum of $E_{\lambda^*}$ restricted to $\mathcal {N}_{\lambda^*}^-$.
		
		In order to show the existence of a second solution for the System ($S_{\lambda^*}$), we proceed in a similar way. Consider a sequence $\{\lambda_k\} \subset (0, \lambda^*)$, such that $\lambda_k \to \lambda^*$ and $\{(u_{k}, v_{k})\}\subset\mathcal{N}_{_{\lambda_k}}^+$. Therefore, $(u_{k}, v_{k})$ is bounded in $X$ and there exist $u_{*}, v_{*} > 0$ where $(u_{*}, v_{*}) \in \mathcal{N}_{\lambda^*}^+\cup \mathcal{N}_{\lambda^*}^0$ such that $(u_{k}, v_{k}) \to (u_{*}, v_{*})$ in $X$. Furthermore, $(u_{k}, v_{k})$ is a solution for the System ($S_{\lambda^*}$). Hence, we obtain that $(u_{*}, v_{*}) \in \mathcal{A}$, $\gamma_{\lambda^*, (u_{*}, v_{*})}^{'}(1) = 0$ and $\gamma_{\lambda^*, (u_{*}, v_{*})}^{''}(1) \ge 0$. Now, using the same ideas discussed just above, we show that $(u_{*}, v_{*}) \in \mathcal{N}_{\lambda^*}^+$. Under these conditions, we infer that $(u_{*}, v_{*}) \in \mathcal{N}_{\lambda^*}^+$ is a global minimum for the functional $E_{\lambda^*}$ restricted to $\mathcal{N}_{\lambda^*}^+\cup \mathcal{N}_{\lambda^*}^0.$ Thus, by using Corollary \ref{incompatibilidade}, we deduce that $(u_{*}, v_{*})\notin \mathcal{N}_{\lambda^*}^0$. This ends the proof. 
	\end{proof}
	\section{The proof of main theorems}
	\noindent {\bf \it The proof of Theorem \ref{teor principal 01sing}}. Firstly, by using  Propositions \ref{elemN+N-inf} and \ref{uvnaonulas}, for each $\lambda \in (0, \lambda^*)$, there exists a positive weak solution $(u, v) \in \mathcal{N}^+_{\lambda}$ for the minimization problem \eqref{C^N+sing}. According to Proposition \ref{solucaosing}, we know that $(u, v)$ is a weak solution for our main problem. Furthermore, by using Proposition \ref{CN+ negsig}, we obtain $E_{\lambda}(u, v) = C_{\mathcal{N}^+_{\lambda}\cap \mathcal{A}} \le E_{\lambda}(t_n^+(u, v)(u, v)) = C_{\lambda} < 0.
	\;\;\square$\\
	
	\noindent {\bf \it The proof of Theorem \ref{teor principal 02sing}}.
	\textbf{(i)} In view of Proposition \ref{elemN+N-inf} there exists $(z, w) \in \mathcal{N}^-_{\lambda}$ such that $C_{\mathcal{N}^-_{\lambda } \cap \mathcal{A}} = E_{\lambda}(z, w)$. Recall that there is an extreme value $\lambda^* > 0$ such that for all $\lambda \in (0, \lambda^*)$ we obtain that $(z, w)$ is a weak solution for our main problem. Therefore, using Proposition and \eqref{uvnaonulas}, $z > 0$ and $w > 0.$\\
	\textbf{(ii)} For each $\lambda \in (0, \lambda^*)$ we have that $t_n^-(z, w) = 1 > t_e(z, w)$. Thus, $\lambda = R_n(z, w) = R_n(t_n^-(z, w)(z, w) < R_e(t_n^-(z, w)(z, w) = R_e(z, w).$ Hence, $E_{\lambda}(z, w) > 0$ and $C_{\mathcal{N}_{_\lambda}^-} = E_{\lambda}(z, w ) > 0$.\\
	\textbf{(iii)} Assume that $\lambda = \lambda_*$. Hence,  we deduce that  $E_{\lambda}(z,w) = 0$.  Since $\lambda_*= \Lambda_e(z, w) = R_e(t_e(z, w)(z, w)) = R_e(t_n^- (z, w)(z, w)) = R_e(z, w).$\\
	\textbf{(iv)} Let $\lambda \in (\lambda_*, \lambda^*)$ be fixed. Consider $(u, v) \in \mathcal{A}$. Then, there exist two roots for the equation $Q_n(t) = \lambda$, see Proposition \ref{tn-,tn+sing}. Furthermore, we know that $t_n^-(u, v)(u, v) \in \mathcal{N}^-_{\lambda}$. Thus, we have that $\lambda_* \le \Lambda_e (u, v) = R_e(t_e(u, v)(u, v))$. Notice that $t_n^-(u, v) \in (0 , t_e(u, v))$. Moreover, we have that $R_e(t_n^-(u, v)(u, v) < R_n(t_n^-(u, v)(u, v) = \lambda$ . Therefore,  $E_{\lambda}(t_n^-(u, v)(u, v) < 0$. Under these conditions, $C_{\mathcal{N}^-_{_\lambda}} \le E_{\lambda}(t_n^-(u, v)(u, v)) < 0$. This ends the proof. 
	$\;\;\square$\\
	
	\noindent{\it The proof of Theorem \ref{teor principal 03cap2}}.
	Consider $\lambda \in (0, \lambda^*)$. According to Proposition \ref{solucaosing} we obtain two weak solutions of the System \eqref{sistema Principal singular} given by $(u, v) \!\in \!\mathcal {N}_{\lambda}^+$ and $(z, w)\! \in\! \mathcal{N}_{\lambda}^-$. Now,  applying Proposition \ref{uvnaonulas}, we obtain that $u, v, z$ and $w$ are strictly positive a.e. in $\Omega.$ $\;\;\square$\\

	\noindent {\it The proof of Theorem \ref{teor principal 04cap2}}.
	Assume that $\lambda = \lambda^*$. Consider a sequence $\lambda_k \to (\lambda^*)^-,k\to\infty$. Hence, there exist two weak solutions for  $(S_{\lambda_k})$ given by $(u_k, v_k) \in \mathcal{N}_{\lambda}^+$ and $(z_k, w_k) \in \mathcal{N}_{\lambda}^-$. The main idea is to guarantee that the limit of the sequences of solutions converges for weak solutions for the System $S_{\lambda^*}$. Recall that  $\mathcal{N}_{\lambda^*}^0 \ne \emptyset$. Now, by using Proposition \ref{limCn-} item \textbf{b)}, we infer that the sequences of solutions strongly converges. Namely, we obtain that the limits $(u_k, v_k) \to (u_*, v_*)$ and $(z_k, w_k) \to (z_*, w_*)$ in $X$ which are minimizers in $\mathcal{N}_{\lambda^*}^+$ and $\mathcal{N}_{\lambda^*}^-$, respectively.
	In view of Corollary \ref{incompatibilidade} we deduce that also that $(u_*, v_*)$ and $(z_*, w_*)$ does not belong to $\mathcal{N}_{\lambda^*}^0$. Furthermore, by using Proposition \ref{exist2solu}, we show that $(u_*, v_*)$ and $(z_*, w_*)$ are weak solutions for the System $(S_{\lambda^*})$. Under these conditions, applying the Proposition \ref{uvnaonulas}, we obtain that $u_*, v_*, z_*$ and $w_*$ are strictly positive. $\square$ 
	
\section{\textbf{Declarations}}

\textbf{Ethical Approval}

It is not applicable.

\textbf{Competing interests}

There are no competing interests.

\textbf{Authors' contributions}

All authors wrote and reviewed this paper.

\textbf{Funding}

CNPq/Brazil with grant 309026/2020-2.

\textbf{Availability of data and materials}

All the data can be accessed from this article.


\end{document}